\documentclass[a4paper,12pt]{article}
\usepackage[latin1]{inputenc}
\usepackage{amsmath,amsthm,amsfonts}
\usepackage{amssymb,amstext,amsgen}
\usepackage{amsbsy,amsopn}
\usepackage{mathrsfs}
\usepackage{amscd}
\usepackage[all]{xy}
\usepackage{color}

\usepackage{dsfont}

\theoremstyle{plain}
\newtheorem{theorem}{Theorem}[section]
\newtheorem{corollary}[theorem]{Corollary}
\newtheorem{proposition}[theorem]{Proposition}
\newtheorem{lemma}[theorem]{Lemma}

\theoremstyle{definition}
\newtheorem{definition}[theorem]{Definition}
\newtheorem{example}[theorem]{Example}

\theoremstyle{remark}
\newtheorem{remark}[theorem]{Remark}
\numberwithin{equation}{section}
\newcommand{\itemcr}{\hrule height 0pt width 40pt\hfill\break\vspace*{-\baselineskip}}

\newcommand{\R}{\ensuremath{\mathbb R}}
\newcommand{\N}{\ensuremath{\mathbb N}}

\newcommand{\al}{\ensuremath{\alpha}}
\newcommand{\bet}{\ensuremath{\beta}}

\newcommand{\de}{\ensuremath{\delta}}
\newcommand{\eps}{\ensuremath{\varepsilon}}
\newcommand{\io}{\ensuremath{\iota}}
\newcommand{\la}{\ensuremath{\lambda}}
\newcommand{\si}{\ensuremath{\sigma}}
\newcommand{\vphi}{\ensuremath{\varphi}}
\newcommand{\om}{\ensuremath{\omega}}
\newcommand{\Ga}{\ensuremath{\Gamma}}
\newcommand{\De}{\ensuremath{\Delta}}
\newcommand{\Om}{\ensuremath{\Omega}}

\newcommand{\ca}{\ensuremath{\mathcal{A}}}
\newcommand{\cb}{\ensuremath{\mathcal{B}}}
\newcommand{\cc}{\ensuremath{\mathcal{C}}}
\newcommand{\cd}{\ensuremath{\mathcal{D}}}
\newcommand{\ce}{\ensuremath{\mathcal{E}}}

\newcommand{\ct}{\ensuremath{\mathcal{T}}}

\newcommand{\core}{\ensuremath{{\mathop{\mathrm{core}}}}}
\newcommand{\ev}  {\ensuremath{{\mathop{\mathrm{ev}}}}}

\newcommand{\evbar} {\ensuremath{\overline{\mathop{\mathrm{ev}}}}}

\newcommand{\evsr}{\ensuremath{{\ev^s_r}}}
\newcommand{\fl}{\ensuremath{{\mathop{\mathrm{fl}}}}}

\newcommand{\pro} {\ensuremath{{\mathop{\mathrm{pr}_1}}}}
\newcommand{\prt} {\ensuremath{{\mathop{\mathrm{pr}_2}}}}
\newcommand{\spr}{\ensuremath{{\mathop{\mathrm{spr}}}}}
\newcommand{\supp}{\ensuremath{\mathop{\mathrm{supp}}}}
\newcommand{\Fl}{\ensuremath{{\mathop{\mathrm{Fl}}}}}
\newcommand{\FlX}{\ensuremath{{{\mathop{\mathrm{Fl}}}^X}}}
\newcommand{\FlXt}{\ensuremath{{{\mathop{\mathrm{Fl}}}^X_\tau}}}
\newcommand{\FlXmt}{\ensuremath{{{\mathop{\mathrm{Fl}}}^X_{-\tau}}}}

\newcommand{\FlYt}{\ensuremath{{{\mathop{\mathrm{Fl}}}^Y_\tau}}}
\newcommand{\FlYmt}{\ensuremath{{{\mathop{\mathrm{Fl}}}^Y_{-\tau}}}}
\newcommand{\FlZt}{\ensuremath{{{\mathop{\mathrm{Fl}}}^Z_\tau}}}
\newcommand{\FlhXt}{\ensuremath{{{\mathop{\widehat{\mathrm{Fl}}}}{}^X_\tau}}}
\newcommand{\FlhXYZt}
  {\ensuremath{{{\mathop{\widehat{\mathrm{Fl}}}}{}^{X,Y,Z}_\tau}}}
\newcommand{\Lin}{\ensuremath{{\mathop{\mathrm{L{}}}}}}
\newcommand{\Tan}{\ensuremath{{\mathop{\mathrm{T{}}}}}}

\newcommand{\TP}{\ensuremath{{\mathop{\mathrm{TP}}}}}
\newcommand{\TO}{\ensuremath{{\mathop{\mathrm{TO}}}}}

\newcommand{\Lie}{\ensuremath{\mathrm{L}}}
\newcommand{\LX}{\ensuremath{\Lie_X}}
\newcommand{\LY}{\ensuremath{\Lie_Y}}

\newcommand{\LXX}{\ensuremath{\Lie_{X,X}}}
\newcommand{\LXY}{\ensuremath{\Lie_{X,Y}}}

\newcommand{\LXz}{\ensuremath{\Lie_{X,0}}}
\newcommand{\LzX}{\ensuremath{\Lie_{0,X}}}

\newcommand{\LzY}{\ensuremath{\Lie_{0,Y}}}

\newcommand{\LhX}{\ensuremath{\hat{\Lie}_X}}

\newcommand{\LhXYZ}{\ensuremath{\hat{\Lie}_{X,Y,Z}}}
\newcommand{\LhXzz}{\ensuremath{\hat{\Lie}_{X,0,0}}}
\newcommand{\LhzXz}{\ensuremath{\hat{\Lie}_{0,X,0}}}
\newcommand{\LhzzX}{\ensuremath{\hat{\Lie}_{0,0,X}}}
\newcommand{\id}{\ensuremath{\mathrm{id}}}
\newcommand{\diff}{\ensuremath{\mathrm{d}}}

\newcommand{\alb}{\ensuremath{{\al\bet}}}
\newcommand{\bal}{\ensuremath{{\bet\al}}}
\newcommand{\Ualb}{\ensuremath{{U_{\alb}}}}
\newcommand{\pa}{\ensuremath{\partial}}

\newcommand{\ddtz}{\ensuremath{\left.\textstyle{\frac{\mathrm{d}}{\mathrm{d}\tau
}} \right|_0 }}
\newcommand{\ddttz}{\ensuremath{\left.\frac{\mathrm{d}}{\mathrm{d}\tau}\right|_{
\tau=0}}}

\newcommand{\lh}{\ensuremath{\hat{\mathrm{L}}}}

\newcommand{\ti}{\ensuremath{\tilde}}
\newcommand{\tit}{\ensuremath{\ti t}}
\newcommand{\adj}{\ensuremath{^\mathrm{ad}}}
\newcommand{\comp}{\ensuremath{\subset\subset}}
\newcommand{\lgl}{\ensuremath{\langle}}
\newcommand{\rgl}{\ensuremath{\rangle}}

\newcommand{\ep}{\hspace*{\fill}$\Box$}

\newcommand{\nn}{\ensuremath{\nonumber}}

\newcommand{\atopmg}[2]{\ensuremath{\genfrac{}{}{0pt}{1}{#1}{#2}}}
\newcommand{\TM}{\ensuremath{\mathrm{T}M}}
\newcommand{\TSM}{\ensuremath{\mathrm{T}^*M}}
\newcommand{\TN}{\ensuremath{\mathrm{T}N}}

\newcommand{\TrsM}{\ensuremath{\mathrm{T}^r_s M}}
\newcommand{\cTrsM}{\ensuremath{\mathcal{T}^r_s(M)}}

\newcommand{\cTrsN}{\ensuremath{\mathcal{T}^r_s(N)}}
\newcommand{\TsrM}{\ensuremath{\mathrm{T}^s_r M}}
\newcommand{\cTsrM}{\ensuremath{\mathcal{T}^s_r(M)}}
\newcommand{\TsrN}{\ensuremath{\mathrm{T}^s_r N}}
\newcommand{\cTsrN}{\ensuremath{\mathcal{T}^s_r(N)}}
\newcommand{\TpM}{\ensuremath{\mathrm{T}_pM}}

\newcommand{\TrspM}{\ensuremath{(\mathrm{T}^r_s)_p M}}
\newcommand{\TsrpM}{\ensuremath{(\mathrm{T}^s_r)_p M}}
\newcommand{\TqM}{\ensuremath{\mathrm{T}_qM}}

\newcommand{\TqN}{\ensuremath{\mathrm{T}_qN}}

\newcommand{\TsrqM}{\ensuremath{(\mathrm{T}^s_r)_q M}}
\newcommand{\Cinf}{\ensuremath{{\mathcal{C}^\infty}}}
\newcommand{\CinfM}{\ensuremath{{\mathcal{C}^\infty(M)}}}

\newcommand{\OmncM}{\ensuremath{\Omega^n_\mathrm{c}(M)}}
\newcommand{\XM}{\ensuremath{\mathfrak{X}(M)}}
\newcommand{\XN}{\ensuremath{\mathfrak{X}(N)}}
\newcommand{\proS} {\ensuremath{{\mathop{\mathrm{pr}_1^*}}}}
\newcommand{\prtS} {\ensuremath{{\mathop{\mathrm{pr}_2^*}}}}
\newcommand{\poTSM}{\ensuremath{\mathrm{pr}_1^*(\mathrm{T}^*M)}}

\newcommand{\ptTN}{\ensuremath{\mathrm{pr}_2^*(\mathrm{T}N)}}

\newcommand{\poTM}{\ensuremath{\mathrm{pr}_1^*(\mathrm{T}M)}}

\newcommand{\GaE}{{\ensuremath\Ga(E)}}

\newcommand{\GacE}{{\ensuremath\Ga_\mathrm{c}(E)}}

\newcommand{\GaME}{{\ensuremath\Ga(M,E)}}
\newcommand{\GacME}{{\ensuremath\Ga_\mathrm{c}(M,E)}}
\newcommand{\GacKME}{{\ensuremath\Ga_{\mathrm{c},K}(M,E)}}

\newcommand{\TPMN}{\ensuremath{\mathop{\mathrm{TP}}(M,N)}}
\newcommand{\GaTPMN}{\ensuremath{\Gamma(\mathop{\mathrm{TP}}(M,N))}}

\newcommand{\GacTOMM}{\ensuremath{\Gamma_\mathrm{c}(\mathop{\mathrm{TO}}(M,M))}}
\newcommand{\TOMN}{\ensuremath{\mathop{\mathrm{TO}}(M,N)}}
\newcommand{\GaTOMN}{\ensuremath{\Gamma(\mathop{\mathrm{TO}}(M,N))}}

\newcommand{\LTMTN}{\ensuremath{\mathop{\mathrm{L}_{M\times
N}}(\mathrm{T}M,\mathrm{T}N)}}

\newcommand{\Aad}{\ensuremath{A^{\mathrm{ad}}}}

\newcommand{\Asr}{\ensuremath{A^s_r}}
\newcommand{\Apq}{\ensuremath{A(p,q)}}

\newcommand{\Asrpq}{\ensuremath{A^s_r(p,q)}}
\newcommand{\Aqp}{\ensuremath{A(q,p)}}

\newcommand{\Upsbl}{\ensuremath{\Upsilon_\bullet}}

\newcommand{\TsrqN}{\ensuremath{(\mathrm{T}^s_r)_q N}} 
\newcommand{\Dp}{\ensuremath{\mathcal{D}'}}
\newcommand{\Dprs}{\ensuremath{{\mathcal{D}'}^r_s}}
\newcommand{\DpM}{\ensuremath{\mathcal{D}'(M)}}
\newcommand{\DprsM}{\ensuremath{{\mathcal{D}'}^r_s(M)}}
\newcommand{\DprsN}{\ensuremath{{\mathcal{D}'}^r_s(N)}}

\newcommand{\sirs}{\ensuremath{\sigma^r_s}}

\newcommand{\iors}{\ensuremath{\iota^r_s}}
\newcommand{\rhors}{\ensuremath{\rho^r_s}}

\newcommand{\ahat}{\ensuremath{\hat{\mathcal{A}}_0(M)}}
\newcommand{\atil}{\ensuremath{\tilde{\mathcal{A}}_0(M)}}

\newcommand{\amtil}{\ensuremath{\tilde{\mathcal{A}}_m(M)}}
\newcommand{\bhat}{\ensuremath{\hat{\mathcal{B}}(M)}}
\newcommand{\eh}{\ensuremath{\hat{\mathcal{E}}}}
\newcommand{\ehm}{\ensuremath{\hat{\mathcal{E}}_\mathrm{m}}}
\newcommand{\nh}{\ensuremath{\hat{\mathcal{N}}}}
\newcommand{\gh}{\ensuremath{\hat{\mathcal{G}}}}

\newcommand{\ehzzm}{\ensuremath{(\hat{\mathcal{E}}^0_0)_\mathrm{m}}}
\newcommand{\nhzz}{\ensuremath{\hat{\mathcal{N}}^0_0}}

\newcommand{\ehrs}{\ensuremath{\hat{\mathcal{E}}^r_s}}
\newcommand{\ehrsm}{\ensuremath{(\hat{\mathcal{E}}^r_s)_\mathrm{m}}}
\newcommand{\nhrs}{\ensuremath{\hat{\mathcal{N}}^r_s}}

\newcommand{\ehM}{\ensuremath{\hat{\mathcal{E}}(M)}}
\newcommand{\ehmM}{\ensuremath{\hat{\mathcal{E}}_\mathrm{m}(M)}}
\newcommand{\nhM}{\ensuremath{\hat{\mathcal{N}}(M)}}
\newcommand{\ghM}{\ensuremath{\hat{\mathcal{G}}(M)}}

\newcommand{\ehzzM}{\ensuremath{\hat{\mathcal{E}}^0_0(M)}}
\newcommand{\ehzzmM}{\ensuremath{(\hat{\mathcal{E}}^0_0)_\mathrm{m}(M)}}
\newcommand{\nhzzM}{\ensuremath{\hat{\mathcal{N}}^0_0(M)}}
\newcommand{\ghzzM}{\ensuremath{\hat{\mathcal{G}}^0_0(M)}}

\newcommand{\ehrsM}{\ensuremath{\hat{\mathcal{E}}^r_s(M)}}
\newcommand{\ehrsmM}{\ensuremath{(\hat{\mathcal{E}}^r_s)_\mathrm{m}(M)}}
\newcommand{\nhrsM}{\ensuremath{\hat{\mathcal{N}}^r_s(M)}}
\newcommand{\ghrsM}{\ensuremath{\hat{\mathcal{G}}^r_s(M)}}
\newcommand{\gloc}{\ensuremath{\mathcal{G}}}

\newcommand{\appsection}[1]{\let\oldthesection\thesection
  \renewcommand{\thesection}{Appendix \oldthesection}
  \section{#1}\let\thesection\oldthesection}

\begin{document}
\title{A global theory of algebras of generalized functions II: tensor
distributions}
\author{Michael Grosser, Michael Kunzinger, \\ Roland Steinbauer, James Vickers}
\date{}
\maketitle

\begin{abstract}
We extend the construction of \cite{vim} by introducing spaces of
generalized tensor fields on smooth manifolds that possess optimal
embedding and consistency properties with spaces of tensor
distributions in the sense of L. Schwartz.  We thereby obtain a
universal algebra of generalized tensor fields canonically containing
the space of distributional tensor fields. The canonical
embedding of distributional tensor fields also commutes with the Lie
derivative.  This construction
provides the basis for applications of algebras of generalized
functions in nonlinear distributional geometry and, in particular,
to the study of spacetimes of low differentiability in general
relativity.
\medskip


\end{abstract}

\vskip6pt
\section{Introduction}\label{Introduction}
The classical theory of distributions has long
proved to be a powerful
tool in the analysis of linear partial differential equations.
The fact that there can in principle be no general multiplication of
distributions (\cite{Schw}), however, makes them of limited
use in the context of nonlinear theories. On the other hand, in the early 1980's J.F.
Colombeau (\cite{c1,c2,Cbull,c3}) constructed algebras of generalized
functions $\gloc(\R^n)$ on Euclidean space, con\-taining the vector space
$\Dp(\R^n)$ of distributions as a subspace and the space of smooth
functions as a subalgebra.
Colombeau algebras combine a maximum of favorable differential
algebraic properties with a maximum of consistency properties with
respect to classical analysis in the light of Laurent Schwartz'
impossibility result (\cite{Schw}). They have since found
diverse applications in analysis, in particular in linear
and nonlinear PDE with non-smooth data or coefficients (cf., e.g.,
\cite{MOBook,KK,NPS,HOP,G08,NP06,D07,MR08} and references therein) and
have increasingly been used in a geometrical context (e.g.,
\cite{found,symm,book,vim,gprg,conn,Jel}) and in general relativity
(see e.g.\ \cite{clarke,herbertgeo,genhyp,penrose,waveq} and
\cite{SV06} for a survey).

In this work we shall focus exclusively on so-called {\em full}
Colombeau algebras which possess a canonical embedding of
distributions. One drawback of the early approaches (given e.g.\ in
\cite{c2}) was that they made explicit use of the linear structure of
$\R^n$, obstructing the construction of an algebra of generalized
functions on differentiable manifolds. This is in contrast to the
situation with the so-called {\em special} algebras \cite[Sec.\
3.2]{book} which are diffeomorphism invariant but do not allow a
canonical embedding. It was only after a considerable effort that the
{\em full} construction could be suitably modified to obtain
diffeomorphism invariance: Building on earlier works of J.F.\
Colombeau and A.\ Meril (\cite{cm}) and J.\ Jel{\'{\i}}nek (\cite{Jel}) a
diffeomorphism invariant (full) Colombeau algebra $\gloc^d(\Om)$ on open
subsets $\Om\subseteq\R^n$ was constructed in \cite{found}. In this
work a complete classification of full Colombeau-type algebras was
given, resulting in two possible versions of the theory.  In
\cite{Jel3,Jel2}, J.\ Jel{\'{\i}}nek was then able to prove that these
algebras are, in fact, isomorphic, thereby providing a unique
diffeomorphism invariant local theory. We will frequently refer to
this construction as the ``local theory''. Finally, the construction
of a full Colombeau algebra $\gh(M)$ on a manifold $M$ based on
intrinsically defined building blocks was given in \cite{vim}. Note
that such an intrinsic construction is vital for applications in a
geometric context: the two main fields of applications we have in mind
are general relativity and Lie group analysis of differential
equations.  For applications in these fields, however, a theory of
generalized tensor fields extending the above scalar construction is
essential. In this paper we develop such a theory.

One might expect that going from generalized scalar fields to generalized
tensor fields is straightforward and could be accomplished by considering
generalized tensor fields as tensor fields with $\gh(M)$-functions as
coefficients. However, the Schwartz impossibility result excludes such a
construction as will be demonstrated in Section~\ref{nogo}. More generally,
we derive a Schwartz-type impossibility result for the tensorial case which
applies to any natural (in the sense specified below) algebra of
generalized functions.

To circumvent this road block we introduce an additional geometric structure
into the theory which allows us to maintain the maximal possible differential
algebraic properties and compatibility with the smooth case.

In more detail, the plan of this paper is as follows. We
begin, in Section \ref{notation}, by introducing some
concepts and notation used throughout the paper.  In
Section~\ref{scalar} we present a new geometric approach to
the scalar construction of \cite{vim} and point out some features
which are essential in the context
of the present work. In particular, we lay the foundations for establishing
the impossibility results for the tensor case which are
presented in Section~\ref{nogo}. Section~\ref{preview} exemplifies
the guiding ideas of the tensorial theory by the special case of
distributional vector fields and demonstrates the basic strategy for
circumventing the no-go results alluded to above.
Sections~\ref{kin} and~\ref{dynamics} form the core of our
construction. The technically demanding proof of the fact that the
embedded image of a distributional tensor field is smooth in the sense
of \cite{KM} is given in Section \ref{sectioniorsvsmooth}.  The
concept of association---which provides `backwards compatibility'
of the new setting with the theory of distributional tensor
fields---is the topic of Section \ref{association}. In the appendices
we collect material on the key notion of transport operators
(\ref{appendixA}) as well as some fundamental results on calculus in
convenient vector spaces in the sense of \cite{KM} (\ref{appendixB}).

\section{Notation}\label{notation}
Here we fix some notation used throughout this article. $I$
always stands for the interval $(0,1]$. Unless otherwise
stated, $M$ will denote an orientable,
paracompact smooth Hausdorff
manifold of (finite) dimension $n$. For subsets
$A$, $B$ of a topological space, we write $A\comp B$ if A is a compact subset of
the interior of $B$. Concerning locally convex vector
spaces (which we always assume to be Hausdorff) we use the terminology and
the results of \cite{Schaefer}. In particular, ``(F)-space'' and
``(F)-topology'' abbreviate ``Fr\'echet space'' resp.\ ``Fr\'echet
topology''. An (LF)-space is a strict inductive limit of an increasing
sequence of (F)-spaces. A bornological isomorphism between locally convex
spaces is a linear isomorphism respecting the families of bounded sets, in
both directions. For details on the notion of smoothness in the sense of \cite{KM},
see \ref{appendixB}.

For any vector
bundle $E$ over $M$, we denote by $\GaME$ resp.\ $\GacME$ the linear
spaces of smooth sections of $E$ resp.\ of smooth sections of $E$
having compact support. For $K\comp M$, $\GacKME$ stands for the
subspace of $\GacME$ consisting of all sections having their support
contained in $K$.  On $\GaME$, we consider the standard system of
seminorms
\begin{equation}\label{seminormssection}
        p_{l,\Psi,L}(u)\,:=\,\sum_{j=1}^{\dim E}\, \sup_{x\in L,
|\nu|\leq l}
        |\partial^\nu (\psi^j\circ (u|_V)\circ\psi^{-1}
        (x))|\,,
\end{equation}
where $l\in\N_0$, $(V,\Psi)$ is a vector bundle chart with component
functions $\psi^1,\dots,\psi^{\dim E}$ over some chart $(V ,\psi)$ on
$M$ and $L\comp \psi(V)$ (cf.\ \cite[p.\ 229]{book}). This leads
to the usual (F)- resp.\ (LF)-topologies on $\GaME$ resp.\ $\GacME$ if
$M$ is separable (i.e., second countable). For general $M$, $\GaME$
becomes a product of (F)-spaces in this way, while the obvious
inductive limit topology renders $\GacME$ a direct topological sum of
(LF)-spaces. By a slight abuse of language, we will speak of (F)-
resp.\ of (LF)-topologies also in the general case, being cautious
when employing standard results on (F)- resp.\ (LF)-spaces.  When
there is no question as to the base space we will sometimes write
$\Ga(E)$ and $\GacE$ rather than $\GaME$ resp. $\GacME$. Finally,
for an open subset $U$ of the manifold $M$, we denote by $E|U$ the
restriction of the bundle $E$ to $U$. For some relevant basic facts on
pullback bundles, two-point tensors and transport operators we refer
to~\ref{appendixA}.

Specializing to the tensor case, we denote by $\TrsM$ the bundle of
$(r,s)$-tensors over $M$ and by $\cTrsM$ the linear space of smooth tensor
fields of type $(r,s)$. Also we write $\XM$ resp.\ $\Om^1(M)$ for the space of
smooth vector fields resp.\ one-forms on $M$. By $\Om_\mathrm{c}^n(M)$ we denote
the space
of compactly supported (smooth) $n$-forms. The vector space of (scalar)
distributions on $M$ is then defined by
$\Dp(M):=(\Om_\mathrm{c}^n(M))'$.
Following \cite{Marsden}, we will view $\DprsM$, the space of
distributional tensor fields of type $(r,s)$, as the dual space of
tensor densities of type $(s,r)$. Since the manifold $M$ is orientable we
may therefore write
\[\DprsM:= \Big(\cTsrM\otimes_{\Cinf(M)}\Omega^n_\mathrm{c}(M)\Big)'.\]
We  denote the action of the distributional tensor field $v\in
\DprsM$ on the $(s,r)$-density $\tit \otimes \omega$
by $\langle v, \tit\otimes\omega \rangle$. 

Moreover, tensor distributions can be viewed as tensor fields
with (scalar) distributional coefficients via the $\cc^\infty(M)$-module isomorphism
(cf., e.g., \cite[Cor.\ 3.1.15]{book})
\begin{equation}\label{distrtensor}
{{\mathcal D}'}^r_s(M)\cong{\mathcal D}'(M)
                      \otimes_{{\mathcal C}^\infty(M)}{\mathcal T}^r_s(M).
\end{equation}
We also mention the following useful representation of
${{\mathcal D}'}^r_s(M)$ as space of linear maps on dual tensor fields
(\cite[Th.\ 3.1.12]{book}):
\begin{equation} \label{dplin}
\DprsM\cong\Lin_\CinfM(\cTsrM,\DpM).
\end{equation}

For the natural pullback action of a diffeomorphism $\mu$ on smooth or distributional 
sections of vector bundles we will write $\mu^*$, the corresponding push-forward
$(\mu^{-1})^*$ will be denoted by $\mu_*$.

\section{The scalar theory}\label{scalar}

To begin with we recall the following natural list of requirements
for any algebra of generalized functions $\ca(M)$
on a manifold $M$ (cf.\ \cite{found} for a full discussion of the local case):
$\ca(M)$ should be an associative, commutative unital algebra satisfying
\begin{itemize}
\item[(i)] There exists a linear embedding $\iota: \cd'(M) \to \ca(M)$ such that $\iota(1)$
is the unit in $\ca(M)$.
\item[(ii)] For every smooth vector field $X\in {\mathfrak X}(M)$ there exists a Lie derivative
$\lh_X: \ca(M) \to \ca(M)$ which is linear and satisfies the Leibniz rule.
\item[(iii)] $\iota$ commutes with Lie derivatives: $\iota({\mathrm L}_X v) = \lh_X \iota(v)$
for all $v\in \cd'(M)$ and all $X\in {\mathfrak X}(M)$.
\item[(iv)] The restriction of the product in $\ca(M)$ to $\cc^\infty(M)$
coincides with the pointwise product of functions: $\iota(f\cdot g) = \iota(f)\iota(g)$ for
all $f,g \in \cc^\infty(M)$.
\end{itemize}
In addition, for the purpose of utilizing such algebras of generalized functions in non-smooth differential
geometry we will assume the following equivariance properties:
\begin{itemize}
\item[(v)] 
There is a natural operation $\hat \mu^*$ of pullback under diffeomorphisms on $\ca(M)$ that commutes
with the embedding: $\iota(\mu^*v) = \hat\mu^*(\iota(v))$ for all $v\in \cd'(M)$ and all
diffeomorphisms $\mu: M\to M$.
\end{itemize}
Due to (iv), $\ca(M)$ becomes a $\cc^\infty(M)$-module by setting $f\cdot u := \iota(f) u$ for
$f\in \cc^\infty(M)$ and $u\in \ca(M)$.

The celebrated impossibility result of L.\ Schwartz \cite{Schw} states that
there is no algebra $\ca(M)$ satisfying (i)--(iii) and (iv'), where (iv')
is a stronger version of (iv) in which one requires compatibility
with the pointwise product of continuous (or $\cc^k$, for some finite $k$)
functions.

We now begin by recalling the construction of the intrinsic full Colombeau
algebra $\gh(M)$ of generalized functions of \cite{vim} which possesses the
distinguishing properties (i)--(v) above. We will put special emphasis on
the geometric nature of the construction and point out the naturality
of our definitions (see also \cite{ro-bed})---as these are also essential features
in the tensor case. The construction basically consists of the following two
steps:
\begin{enumerate}
 \item [(A)] Definition of a basic space $\ehM$ that is an
  algebra with unit, together with linear embeddings $\iota:\DpM \to \ehM$
  and $\sigma:{\mathcal C}^\infty(M) \to\ehM$ where $\sigma$ is an algebra
  homomorphism and both $\sigma$ and $\iota$ commute with the action of diffeomorphisms.
  Definition of Lie derivatives $\LhX$ on $\ehM$
  that coincide with the usual Lie derivatives on $\DpM$
  (via $\iota$) resp.\ on $\CinfM$ (via $\si$).
 \item [(B)] Definition of the spaces $\ehmM$ of moderate and
  $\nhM$ of negligible elements of the basic space $\ehM$ such that
  $\ehmM$ is a subalgebra of $\ehM$ and
  $\nhM$ is an ideal in $\ehmM$ containing
  $(\iota-\sigma)({\mathcal C}^\infty(M))$.  Definition of the algebra as
  the quotient $\ghM:={\ehmM/\nhM}$.
\end{enumerate}
Observe that step (A) serves to implement properties (i)--(iii) and (v) of the above
list while step (B) guarantees the validity of (iv).
Since step (A) describes the basic space underlying our construction of
generalized functions we refer to this step (by analogy with
analytic mechanics) as giving the {\em ``kinematics''} of the construction,
and since step (B) refers to additional (asymptotic) conditions
which we impose on
the objects, we will refer to this step as giving the {\em ``dynamics''} of
the construction.

To introduce the kinematics part of the theory we discuss the
question of the embeddings which will lead us to a natural choice of
the basic space. We wish to embed both the space of smooth functions
$C^\infty(M)$ and the space of distributions $\Dp(M)$.
Since smooth functions depend upon
points $p \in M$ and distributions depend upon compactly supported
$n$-forms it is natural to take our space of generalized functions to depend
upon both of these. However, for technical reasons it is convenient to only
use normalized $n$-forms.
\begin{definition}\label{a0hatdef} \itemcr
\begin{itemize}
\item [(i)] The space of compactly supported $n$-forms with unit integral is denoted by
\[\ahat:=\{\om\in\Om^n_\mathrm{c}(M):\int_M\om=1\}.\]
\item [(ii)] The basic space of generalized scalar fields is given by
 \[\eh(M):={\mathcal C}^\infty(\hat{\mathcal A}_0(M)\times M).\]
\end{itemize}
\end{definition}
Here and throughout this paper, smoothness is understood in the sense of calculus in
convenient vector spaces (\cite{KM}), which provides a natural and powerful setting
for infinite-dimensional global analysis. A map between
locally convex spaces is defined to be smooth if it maps smooth curves to
smooth curves. For some facts on convenient calculus in the context of the
scalar theory we refer to \cite[Sec.\ 4]{found}. More specific results pertaining to
the present paper are developed in \ref{appendixB}.
Elements of the basic space will be denoted by $R$ and their arguments
by $\omega$ and $p$. 
\begin{definition}\label{scalarembed}
We define the embedding of smooth functions resp.\ distributions into the basic space
by
\[ \sigma(f)(\omega,p):=f(p)\quad \mbox{and }
\quad \iota(v)(\omega,p):=\langle v,\omega\rangle.\]
\end{definition}
Note that we clearly have $\sigma(fg)=\sigma(f)\sigma(g)$.

The second ingredient of the kinematics part of the
construction is the definition of an appropriate Lie derivative. Given a
complete vector field $X$, the Lie derivative of a geometric object defined
on a natural bundle on a manifold $M$ may be given in terms of the pullback
of the induced flow (Appendix A and \cite{michorbook}).
This geometric approach has the
further advantage that in every instance the Leibniz rule is an immediate
consequence of the chain rule.
In order to define the Lie derivative of an element $R\in\ehM$ we therefore
first need to specify the action of diffeomorphisms on $\ehM$.

Given a diffeomorphism $\mu:M\to M$ we
have the following pullback actions of $\mu$ on the spaces of smooth
functions resp.\ of distributions:
\[\mu^*f(p):=f(\mu p)\quad \mbox{and}\quad \langle\mu^*v,\omega\rangle:=\langle v,\mu_*\omega\rangle,\]
where $\mu p:=\mu(p)$ and $\mu_*\omega$ denotes the
push-forward of the $n$-form $\omega$.  Hence the natural choice of
definitions is the following.
\begin{definition}\itemcr
\begin{itemize}
\item [(i)] The action of a diffeomorphism $\mu$ of $M$ on elements of $\ehM$ is given by
 \[(\hat\mu^*R)(\omega,p):=R(\mu_*\omega,\mu p).\]
\item [(ii)] The Lie derivative on $\ehM$ with respect to a complete smooth vector
field $X$ on $M$ is
 \[\hat \Lie_X
R:=\left.\frac{d}{d\tau}\right|_{\tau=0}(\widehat{\Fl^X_\tau})^*\, R,\]
where $\Fl^X_\tau$ denotes the flow induced by $X$ at time $\tau$.
\end{itemize}
\end{definition}

\noindent It is now readily shown that
\begin{equation*}
\hat\mu^*\circ\sigma=\sigma\circ\mu^*\quad \mbox{and}\quad \hat\mu^*\circ\iota=\iota\circ\mu^*
\end{equation*}
which immediately implies
\[\hat \Lie_X\circ\sigma=\sigma\circ \Lie_X\quad \mbox{and}\quad \hat
\Lie_X\circ\iota=\iota\circ \Lie_X.\]
Moreover, an explicit calculation gives
\[
\hat \Lie_XR(\omega,p)=
-d_1R(\omega,p)\,\Lie_X\omega+\Lie_XR(\omega,.)\mid_p
\]
which is precisely the definition of the Lie derivative in the general case given in
equation (14) of \cite{vim}.

Having established (i)--(iii) and (v) we now turn to step (B), i.e.,
the dynamics part of our construction. The key idea in establishing
(iv) is to identify, via a quotient construction, the images of
smooth functions under both the embeddings: For smooth $f$
one has $\sigma(f)(\omega,p)=f(p)$, whereas regarding $f$ as a distribution,
one has $\iota(f)(\omega,p)=\int f(q)\omega(q)$. In order to identify these two expressions
we would like to set $\omega(q)=\delta_p(q)$. Clearly this is not possible in a
strict sense, but replacing the $n$-form $\omega$ by a net of $n$-forms
$\Phi(\eps,p)$ which tend to $\delta_p$ appropriately as $\eps \to 0$
and using suitable asymptotic estimates shows the right way to proceed.

We begin by defining an appropriate space of delta nets 
(see \cite{vim} for details).
\begin{definition} \itemcr\label{kernels}
\begin{enumerate}
 \item[(1)] An element $\Phi\in \cc^\infty(I \times M,\ahat)$ 
is called a smoothing kernel
if it satisfies the following conditions
\begin{itemize}
  \item [(i)] $\forall K\subset\subset M$ $\exists\, \eps_0$,$C>0$
              $\forall p\in K$ $\forall \eps \le \eps_0$:
              $\supp\Phi(\eps,p)\subseteq B_{\eps C}(p)$
  \item [(ii)] $\forall K \subset\subset M$ $\forall k,l\in \N_0$
        $\forall X_1,\dots,X_k,Y_1,\dots,Y_l\in\mathfrak{X}(M)$
        \[
                  \sup_{\atopmg{p\in K}{q\in M}}
\|\Lie_{Y_1}\dots
                  \Lie_{Y_l}(\Lie'_{X_1}+\Lie_{X_1})\dots
            (\Lie'_{X_k}+\Lie_{X_k})
                        \Phi(\eps,p)(q) \| = O(\eps^{-(n+l)})
                \]
\end{itemize}
where $\Lie'_X$ is the Lie derivative of the map
$p \mapsto \Phi(\eps,p)(q)$ and $\Lie_X$ is the Lie derivative of the map
$q \mapsto \Phi(\eps,p)(q)$.
The space of smoothing kernels on $M$ is denoted by \atil. We will 
use the notations $\Phi(\eps,p)$ and $\Phi_{\eps,p}$ interchangeably.
\item[(2)] For each $m\in\N$ we denote by $\tilde{\mathcal{A}}_m(M)$ the
set of all $\Phi\in \atil$ such that for all $f\in \cc^\infty(M)$ and
all $K\subset\subset M$
\[
\sup_{p\in K}  \left|f(p)- \int_M f(q)\Phi(\eps,p)(q) \right | = O(\eps^{m+1})
\]
\end{enumerate}
\end{definition}
The norms and metric balls in this definition are to be understood 
with respect to some Riemannian
metric, but the asymptotic estimates are independent of the choice of metric.

We may now define the subspaces of moderate and negligible elements of
$\eh(M)$ and carry out the announced quotient construction.
\begin{definition}\label{modnegscalar} \itemcr
\begin{itemize}
\item [(i)] $R\in \eh(M)$ is called moderate if \medskip\\
 $\forall K\subset\subset M$  $\forall k\in \N_0$  $\exists N\in \N$
  $\forall\ X_1,\dots,X_k\in
                \mathfrak{X}(M)$
  $\forall\ \Phi \in \atil$
\begin{equation*} 
\sup_{p\in K} |\Lie_{X_1}\dots \Lie_{X_k}(R(\Phi(\eps,p),p)) | = O(\eps^{-N}).
\end{equation*}
The subset of moderate elements of $\eh(M)$ is denoted by $\ehm(M)$.
\item [(ii)] $R\in \ehm(M)$ is called negligible if \medskip\\
 $\forall K\subset\subset M$ $\forall k, l \in \N_0$ $\exists m \in \N$
    $\forall \  X_1,\dots,X_k\in \mathfrak{X}(M)$  $\forall \Phi \in  \tilde{\mathcal{A}}_m(M)$ 
\begin{equation*} 
     \sup_{p\in K} |\Lie_{X_1}\dots \Lie_{X_k}(R(\Phi(\eps,p),p)) | =
O(\eps^{l}).
\end{equation*}
The subset of negligible elements of $\ehm(M)$ is denoted by \nhM.
\item [(iii)] The Colombeau algebra of generalized functions on $M$ is defined by
\[
\ghM := \ehm(M) / \nhM.
\]
\end{itemize}
\end{definition}

One now proves that $(\iota-\sigma)(\Cinf(M))\in\nhM$ by
recourse to the local theory (\cite{found}).
So we obtain (iv) and since the properties obtained in step
(A) are not lost in the quotient construction we indeed have
(i)--(v). Note, however, the following subtlety: The fact that $\ghM$ is a
{\it differential algebra} depends on the invariance of the tests for
moderateness and negligibility under the action of the generalized Lie
derivative $\LhX$. This, however, is surprisingly hard to prove and has
been done in \cite{vim} by recourse to the local theory as well.

We conclude this section with a lemma which will turn out to be useful
for proving the analogue of $(\iota-\sigma)(\Cinf(M))\in\nhM$ in the
tensor case (Theorem \ref{T1}\,(iii)).
\begin{lemma}\label{lemmagtozero}
Let $g\in\Cinf(M\times M)$ satisfy $g(p,p)=0$ for all $p\in M$, and let $m\in\N_0$.
Then for every $\Phi\in\tilde\ca_m(M)$ and every $K\comp M$ we have
\begin{equation}\label{gtozero}
\sup\limits_{p\in K}\left|\int_Mg(p,q)\Phi(\eps,p)(q)\right|=O(\eps^{m+1}).
\end{equation}
\end{lemma}
\begin{proof}
Without loss of generality we may assume that $K$ is contained in some open set
$W$ where $(W,\psi)$ is a chart on $M$. Fixing $L$ such that $K\comp L\comp W$
there is an $\eps_0>0$ such that for all $\eps\leq\eps_0$ and all $p\in K$ we
have $\supp \Phi(\eps,p)\subseteq L$, by (1)(i) of Definition \ref{kernels}.
Hence the integral in (\ref{gtozero}) may be written in local coordinates as
$$\int_{\psi(W)}\tilde
g(x,y)\eps^{-n}\phi(\eps,x)\left(\textstyle{\frac{y-x}{\eps}}\right)\diff^n
y$$
where $\tilde g=g\circ (\psi\times\psi)^{-1}\in\Cinf(\psi(W)\times\psi(W))$ and
$\phi:D\,(\subseteq I\times\psi(W))\,\to\ca_0(\R^n)$ has the properties
specified in \cite[Lemma 4.2 (A)(i)(ii)]{vim}. In particular, $D$ contains
$(0,\eps_1]\times\psi(K)$ for some $\eps_1\leq\eps_0$ in its interior and
we have $\sup_{x\in
K'}|\int_{\R^n}\phi(\eps,x)(y)y^\bet\diff^ny|=O(\eps^{m+1-|\bet|})$ for all
multiindices $\bet$ with $1\leq|\bet|\leq m$ and all $K'\comp \psi(W)$. Now a
Taylor argument (analogous to the one in the proof of 
\cite[Th.\ 7.4\,(iii)]{found}, with $\al$ set equal to $0$) establishes (\ref{gtozero}).
\end{proof}
Note that for $g(p,q)=f(p)-f(q)$ where $f\in\CinfM$, the asymptotic estimate
(\ref{gtozero}) is nothing but the condition defining the space $\amtil$.

\section{No-Go results in the tensorial setting}\label{nogo}

In this section we establish some general no-go results in the spirit of
the Schwartz impossibility theorem \cite{Schw}, valid for tensorial extensions of
{\em any} algebra $\ca(M)$ of generalized functions satisfying the set of requirements stated in
Section \ref{scalar}. For a comprehensive discussion tailored to the special case 
$\ca(M)=\gh(M)$ we refer to \cite{mg-bed}.

Throughout this section we suppose that $\ca(M)$ is any associative, commutative
unital algebra with embedding $\iota: \cd'(M) \to \ca(M)$ satisfying 
conditions (i)--(v) from Section \ref{scalar}.

We first note that such an $\iota$ cannot
be $\cc^\infty(M)$-linear. In fact, let $M= \R$. Then supposing that $\iota$ is
$\cc^\infty(\R)$-linear we derive the following contradiction:
\begin{eqnarray*}
\iota(\delta) = \iota (1) \iota(\delta) =
\iota (v.p.\frac{1}{x} \cdot x)  \iota(\delta) =  \iota (v.p.\frac{1}{x})
\iota(x\delta) = 0.
\end{eqnarray*}
Clearly this calculation can be pulled back to any manifold. Thus, in general,
\begin{equation}\label{different}
\iota(fv)\neq\iota(f)\cdot\iota(v)\qquad\qquad(f\in{\mathcal C}^\infty(M),\
v\in{\mathcal D}'(M)),
\end{equation}
or, $\iota(fv)\neq f\cdot\iota(v)$, for any algebra $\ca(M)$ of generalized
functions as above.

As we shall demonstrate, this basic observation forecloses the most obvious
way of extending a given scalar theory of algebras of generalized functions
to the tensorial setting.

To this end, we write the natural embedding
$\rhors:\cTrsM\to\DprsM$ given by
$$\lgl\rhors(t),\tit\otimes\om\rgl:=\textstyle{\int\limits_M}
(t\cdot\tit)\,\om  \qquad(t\in\cTrsM,\ \tit\in\cTsrM,\ \om\in\OmncM)$$
in a different manner: Recall from (\ref{distrtensor}) that
$$
{{\mathcal D}'}^r_s(M)\cong{\mathcal D}'(M)
                      \otimes_{{\mathcal C}^\infty(M)}{\mathcal T}^r_s(M).
$$
Denoting by $\rho$ the standard embedding of $\cc^\infty(M)$ into $\cd'(M)$,
the fact that $\rho$ is $\cc^\infty(M)$-linear
allows one to rewrite $\rhors$ as
\begin{equation}\label{rhors}
\rho^r_s = \rho\otimes_{\cc^\infty(M)}\mathrm{id}:
   {\cc^\infty}(M)\otimes_{\cc^\infty(M)}{\mathcal T}^r_s(M)\to
   \cd'(M)\otimes_{\cc^\infty(M)} {\mathcal T}^r_s(M).
\end{equation}
Given $\ca(M)$ as above it is therefore natural to define the space of
tensor-valued generalized functions as the $\cc^\infty(M)$-module of
tensor fields with generalized coefficients from $\ca(M)$, i.e.,
\begin{equation}\label{generaltensor}
{\ca}^r_s(M):={\ca}(M)
                      \otimes_{{\mathcal C}^\infty(M)}{\mathcal T}^r_s(M).
\end{equation}
It is then tempting to mimic (\ref{rhors}) and
define an embedding of $\DprsM$ into $\ca^r_s(M)$ by
\begin{equation}\label{iorswrongdef}
\iota\otimes\mathrm{id}:
   {\mathcal D}'(M)\otimes_{\cc^\infty(M)}{\mathcal T}^r_s(M)\to
   \ca(M)\otimes_{\cc^\infty(M)} {\mathcal T}^r_s(M).
\end{equation}
The following result, however, shows that this map is not well-defined
(not even in the scalar case $r=s=0$) and therefore cannot
serve as the desired embedding of $\DprsM$ into $\ca^r_s(M)$:
\begin{proposition}\label{cinflinprop} For any algebra $\ca(M)$ which satisfies
(i) from Section \ref{scalar} and is
a unital $\cc^\infty(M)$-module, the following are equivalent:
\begin{itemize}
\item[(i)] $\iota \otimes \mathrm{id}: \cd'(M)\otimes_\R \cc^\infty(M) \to
\ca(M)\otimes_{\cc^\infty(M)}
\cc^\infty(M)$  is $\cc^\infty(M)$-balanced, i.e., $\iota\otimes\mathrm{id}$
from
(\ref{iorswrongdef}) is well-defined.
\item[(ii)] $\iota$ is $\cc^\infty(M)$-linear.
\end{itemize}
\end{proposition}
\begin{proof}
Let $v\in \cd'(M)$ and $f,g \in \cc^\infty(M)$. \\
(i)$\Rightarrow$(ii):  $\iota(fv)\otimes 1 = \iota\otimes \mathrm{id} (fv\otimes
1) =
\iota\otimes \mathrm{id} (v\otimes f) = \iota(v)\otimes f = f \iota(v)\otimes
1$. Thus, since
$u\otimes f \mapsto fu$ is an isomorphism from $\ca(M)\otimes_{\cc^\infty(M)}
\cc^\infty(M)$ to
$\ca(M)$, (ii) follows. \\
(ii)$\Rightarrow$(i): $\iota\otimes \mathrm{id} (fv\otimes g) = \iota(fv)
\otimes g = f\iota(v) \otimes g
= \iota(v)\otimes fg = \iota\otimes \mathrm{id} (v\otimes fg)$.
\end{proof}

It is instructive to take a look at the coordinate version of the 
impossibility of (\ref{iorswrongdef}).
Indeed as we shall show below condition (i) of Proposition \ref{cinflinprop} is
equivalent to the
statement that coordinate-wise embedding of distributional tensor fields is
independent of the choice of
a local basis (cf.\ also \cite{RD}). 

To this end, assume
that $M$ can be described by a single chart. Then ${\mathcal T}^r_s(M)$ has
a ${\mathcal C}^\infty(M)$-basis consisting of (smooth) tensor fields, say,
$e_1,\dots,e_m$ $\in$ ${\mathcal T}^r_s(M)$ with $m=n^{r+s}$. By
(\ref{distrtensor}), every $v\in{{\mathcal D}'}^r_s(M)$ can be written as
$v=v^i\otimes e_i$ (using summation convention) with $v^i\in{\mathcal
D}'(M)$.  Consider a change of basis given by
$e_i=a^j_i\hat e_j$, with $a^j_i$ smooth. Then $v=\hat v^j\otimes \hat e_j$
with $\hat v^j = a^j_i v^i$. Applying $\iota\otimes\mathrm{id}$ to both
representations of $v$, we obtain
$$(\iota\otimes\mathrm{id})(v^i\otimes e_i)=
     \iota (v^i)\otimes (a^j_i\hat e_j)=
     (\iota (v^i)a^j_i)\otimes \hat e_j=
     (\iota(a^j_i)\iota(v^i))\otimes \hat e_j$$
resp.\
$$(\iota\otimes\mathrm{id})(\hat v^j\otimes \hat e_j)=
     \iota(a^j_i v^i)\otimes \hat e_j$$
which are different in general due to (\ref{different}).
It follows that coordinate-wise embedding is not feasible for obtaining an
embedding
of tensor distributions.

The following example gives an explicit contradiction for the case $\ca(M)$ $=$
$\ghM$.
\begin{example} Set $M=\R$, and let $v\in {\cd'}^1_0(\R) =
\cd'(\R)\otimes_{\cc^\infty(\R)} \mathfrak{X}(\R)$
be given by $v=\delta' \otimes \partial_x$. Then
$$
v = (1+x^2)\delta' \otimes \frac{1}{1+x^2} \partial_x
$$
and we note that $(1+x^2)$ is in fact the transition function of the underlying
vector bundle $TM$ with respect to the coordinate transformation $x\mapsto
x+x^3$.
With $\iota: \cd'(\R) \to \gh(\R)$, suppose that
$$
\iota(\delta')\otimes \partial_x = \iota((1+x^2)\delta') \otimes \frac{1}{1+x^2}
\partial_x.
$$
Then since $x^2\delta' = 0$ in $\cd'(\R)$, this would amount to
$(1+x^2)\iota(\delta') = \iota(\delta')$.
However, it is easily seen that $x^2$ is not a zero-divisor in $\gh(\R)$ (adapt 
\cite[Ex.\ 1.2.40]{book} by choosing
an appropriate smoothing kernel), so we arrive at a contradiction.
\end{example}

In order to circumvent the ``domain obstruction'' met in
(\ref{iorswrongdef}) (which arose from $\iors\otimes\id$ not being
$\CinfM$-balanced) one might try to switch to isomorphic representations of the
spaces involved: By (\ref{distrtensor}) and (\ref{dplin}), we have
$\DpM\otimes_\CinfM \cTrsM\cong\Lin_\CinfM(\cTsrM,\DpM)$, and similarly
$\ca(M)\otimes_\CinfM \cTrsM\cong\Lin_\CinfM(\cTsrM,\ca(M))$ holds (the
latter is proved analogously to the corresponding statement in \cite[Th.\
3.1.12]{book}). The most plausible candidate for an embedding of
$\Lin_\CinfM(\cTsrM,\DpM)$ into
$\Lin_\CinfM(\cTsrM,\ca(M))$ certainly is $\io_*$, that is, composition from
the left with $\io:\DpM\to\ca(M)$. Indeed, this choice presents no difficulties
whatsoever with respect to the domain $\Lin_\CinfM(\cTsrM,\DpM)$. However, this
time we encounter a ``range obstruction'' in the sense that we do end up only in
$\Lin_\R(\cTsrM,\ca(M))$, due to the fact that $\io$ is only
$\R$-linear.
Proposition \ref{cinflinprop} demonstrates that the domain
and the range obstructions, though of essentially different appearance, are in
fact equivalent.

It is noteworthy that the range obstruction is
encountered once more
when trying to write down plausible formulae for an embedding of tensor
distributions 
into a na\"{\i}vely defined basic space for generalized tensor fields.
Aiming at minimal changes as compared to the scalar theory it is natural to
start 
out from scalar basic space members $u:\ahat\times M\to \R$, to replace 
the ``scalar'' range space $\R$ by the vector bundle $\TrsM$ and to ask for 
$u(\om,.)$ to be a member of $\cTrsM$, for every $\om\in\ahat$. 
Now when looking for a ``tensor embedding'' $\iors$ 
we aim at guaranteeing $\iors(v)(\om,.)$ (for $v\in\DprsM$) to be a member of
$\cTrsM$ by defining 
it via a $\CinfM$-linear action on $\tit\in\cTsrM$. Virtually the only formula
making 
sense is $\lgl v,\tit\otimes\om\rgl$, forcing us to set
\begin{equation}
(\iors(v)(\om, .)\cdot\tit)(p):=
\lgl v, \tit\otimes\om \rgl.
\label{wrong1}
\end{equation}
At first glance, (\ref{wrong1}) displays a reassuring similarity to the 
scalar case definition $\io(v)(\om,p):=\lgl v,\om\rgl$. In particular, 
both right hand sides do not depend on $p$. This, however, leads to failure 
in the tensor case: Choosing $\tit$ with (nontrivial) compact support, the left
hand 
side also has compact support with respect to $p$, so, being constant
it has to vanish identically, making (\ref{wrong1}) absurd.
On top of this and, in fact, continuing our above discussion we note that 
(\ref{wrong1}) also fails to provide $\CinfM$-linearity of 
$\iors(v)(\om,.)$ since this would imply the contradictory relation 
($f\in\CinfM$)
$$
\lgl v,(f\tit)\otimes\om\rgl =
(\iors(v)(\om,.)\cdot(f\tit))(p)=
f(p)\,(\iors(v)(\om,.)\cdot\tit)(p)=
f(p)\,\lgl v,\tit\otimes\om\rgl.
$$
Finally, (\ref{wrong1}) turns out to be nothing but a reformulation of the range
obstruction: The element $\bar v$ of $\Lin_\CinfM(\cTrsM,\DpM)$ corresponding to
$v\in\DprsM$ by $\lgl\bar v(\tit),\om\rgl=\lgl v,\tit\otimes\om\rgl$ satisfies
$((\io\circ\bar v)(\tit))(\om,p)=\lgl\bar v(\tit),\om\rgl
=\lgl v, \tit\otimes\om \rgl$. Hence defining $\iors(v)$ by (\ref{wrong1})
corresponds to composing $\bar v$ with $\iota$ from the left which is the move
leading straight into the range obstruction.

These considerations show that emulating the scalar case by
na\"{\i}ve  manipulation of formulae has to be abandoned.
In the next section we show how the introduction of an additional geometric
structure allows one to circumvent this problem. In particular, we will arrive
at a formula for the embedding of tensor distributions ((\ref{embedformula})) which 
allows a clear view on the failure of \eqref{wrong1} and which, in fact,
provides a remedy.

\section{Previewing the construction}\label{preview}
The obstructions to a component-wise embedding of distributional
tensor fields discussed in the preceding section are essentially
algebraic in nature.  However, there is also a purely
geometric reason for objecting to such an approach.  We illustrate this
below since it points the way toward the resolution of the problem,
the basic idea going back to~\cite{sotonTF}.

Let us begin by reviewing the embedding of a (regular) scalar distribution
given by a continuous function $g$ on $M$ (see Definition~\ref{scalarembed}).
Pick some $n$-form $\omega$ viewed as approximating the Dirac measure $\delta_p$ around
$p\in M$. Then
\[
(\iota g)(\omega,p) =\langle g,\om\rangle =\int_M g(q)\omega(q)
\]may be seen as collecting values of $g$ around $p$ and forming a smooth average
(recall that $\int\omega=1$).
Now, in case $v$ is a continuous vector field, then
its values $v(q)$ do not lie in the same tangent space for different $q$
and there is in general no way of defining an embedding $\iota^1_0$ of
continuous vector fields via an integral of the form
\begin {equation}\label{wrong}
\iota^1_0(v)(\om,p)=\int_M v(q)\om(q)
\end{equation}
since there is no way of identifying $\TpM$ and $ \TqM$ for $p\not=q$.

However, this observation also points the way to the remedy:
we need some additional geometric structure providing such an
identification. One possibility would be to use a (background)
connection or Riemannian metric. Let $p,q$ lie within a geodesically
convex neighborhood. Then parallel transport along the unique geodesic
connecting $p$ and $q$ defines a map $A(p,q): \TpM \to \TqM$.
In principle it would be possible to employ the shrinking supports of the
smoothing kernels to extend this locally defined ``transport operator''
to the whole manifold using suitable cut-off functions.
However, to avoid technicalities we have chosen
to work directly with compactly supported transport operators $A$ defined as
compactly supported smooth sections of the bundle $\TO(M,M)=\Lin_{M\times
M}(TM,TM)$ (see
\ref{appendixA}), i.e., $A(p,q)$ being a linear map $\TpM\to\TqM$.
This map may be used to ``gather'' at $p$ the values of $v$ (via $A(q,p)v(q)$)
before averaging them, i.e., we may set
\begin{equation}\label{works}
 \iota^1_0(v)(\om,p,A):=\int_M A(q,p)v(q)\,\om(q),
\end{equation}
with the new mechanism becoming most visible by comparing (\ref{wrong}) with
(\ref{works}).

Observe, however, the following important fact: To maintain the spirit of
the full construction, i.e., to provide a canonical embedding independent
of additional choices
we have to make the elements of our basic
space depend on an additional third slot containing $A$.
Indeed, as one can show, $\iota^1_0 (v)$ as defined in (\ref{works}) above depends
smoothly on $\om,p,A$. (In fact, the proof of this statement in the general case is one of the
technically most demanding parts of this paper and will be given in
Section~\ref{sectioniorsvsmooth}.) Thus for each fixed pair $(\om,A)$ we
have that
\[
\iota^1_0 (v)(\om,A):=[p\mapsto \iota^1_0(v)(\om,p,A)]
\]
defines a smooth vector field on $M$.
This strongly suggests that we choose our basic space $\hat{\mathcal
    E}^1_0(M)$ of generalized vector fields to explicitly include
  dependence on the transport operators, i.e.,
\[
\hat{\mathcal E}^1_0(M)
 :=\{u\in\Cinf(\ahat\times M\times\Ga_\mathrm{c}(\TO(M,M)),\TM)\mid
u(\omega,p,A)\in\TpM\}.
\]
In particular, $p\mapsto t(\omega,p,A)$ is a member of $\XM$
for any fixed $\omega,A$. Following this strategy of course means that one
also has to allow for dependence of scalar fields on transport operators
and one must therefore upgrade  the scalar theory from
the old 2-slot version as presented in Section \ref{scalar} to a new 3-slot version.

Finally, we may turn to embedding general distributional vector fields.
By definition of ${{\mathcal D}'}^1_0(M)$, $v$ takes (finite sums of)
tensors $\tilde u\otimes \omega$ with $\tilde u\in\Om^1(M)$ as arguments.
Hence the most convenient way of defining $\iota^1_0(v)(\omega,p,A)$
is to let the prospective smooth vector field $\iota^1_0(v)(\om,A)$ act
on a one-form $\tilde u$. In fact, we may write for continuous $v$
\begin{eqnarray*}
\iota^1_0(v)(\omega,p,A)\cdot\tilde u(p)&=&
           \big(\iota^1_0(v)(\om,A)\cdot\tilde u\big)(p)\\
           &=&\int_M A(q,p)v(q)\cdot\tilde u(p)\,\,\om(q)\\
           &=&\int_M v(q)\cdot A(q,p)^{\mathrm{ad}}\tilde u(p)\,\,\om(q)\\
           &=&\langle v(\,.\,)
                 ,A(\,.\,,p)^{\mathrm{ad}}\tilde u(p)\otimes\om(\,.\,)\rangle.
\end{eqnarray*}
In the last expression above, we are now free to replace the regular
distributional vector field $v$ by any $v\ \in\ {{\mathcal D}'}^1_0(M)$.
This leads to our definition of $\iota^1_0$ by
\begin{equation}\label{embedformula}
\begin{array}{rcl}
\iota^1_0(v)(\om,p,A)\cdot\tilde u(p)
 &:=&\big(\iota^1_0(v)(\om,A)\cdot\tilde u\big)(p) \\[5pt]
 &:=&\langle\ v(.)\ ,A(\,.\,,p)^{\mathrm{ad}}\tilde u(p)\,\otimes\,\om(.)\ \rangle.
\end{array}
\end{equation}

Observe the shift of focus in the above formulas as compared to
(\ref{works}): rather than thinking of the transport operator as
``gathering'' at $p$ the values of the vector field $v$ 
it (more precisely, its flipped and adjoint version) serves to ``spread'' the value
of the ``test one-form'' $\tilde u(p)$ at $p$ to the neighboring
points $q$.

Connecting to Section~\ref{nogo} we point out that the embedding \eqref{embedformula} may be viewed
as a correction of the flawed formula \eqref{wrong1}. Comparison reveals
that the introduction of the transport operator, i.e., the replacement 
of $\tilde u(.)$ by $A(\,.\,,p)^{\mathrm{ad}}\tilde u(p)$, removes both  
failures of \eqref{wrong1}: the right hand side now does depend on $p$ and, moreover,
defines a $\CinfM$-linear mapping on $\Om^1(M)$.

The case of general $(r,s)$-tensor fields can be dealt with by using
appropriate tensor products of the transport operators. The details of
this are given in Section~\ref{kin} on the kinematics part of our
construction.
In particular, this includes the definition of a basic space for
generalized tensor fields of type $(r,s)$ which depend on transport
operators and the general definition of the embeddings $\sirs$ of smooth and
$\iors$ of distributional tensor fields. Furthermore we define the pullback action
as well as the Lie derivative with respect to smooth vector fields for elements of the basic
space in such a way that they commute with the embeddings.
An added complication as compared to the scalar case is the fact that
the transport operators are two-point objects so that the action of
diffeomorphisms needs to be treated with some care. Some basic material on this
topic is collected in \ref{appendixA}.

As already indicated above the proof that the embedded image
$\iors(v)$ of a distributional tensor field $v$ is smooth with respect
to all its three variables (hence belongs to the basic
space) is rather involved.  It builds on some results on calculus in
(infinite-dimensional) convenient vector spaces which are nontrivial
to derive for the following reason: We have to carefully distinguish
(and bridge the gap) between the standard locally convex topologies
defined on the respective spaces of sections
and their 
convenient structures on which the calculus according to \cite{KM} rests.
We provide the proof of smoothness of $\iors(v)$ in Section \ref{sectioniorsvsmooth}
and have deferred some useful results on the calculus to \ref{appendixB}.

The dynamics part of our construction is carried out in
Section~\ref{dynamics}.  The heart of this part is
the quotient construction that allows one to identify $\iors(v)$ and $\sirs(v)$
for smooth tensor fields $v$.  
The introduction of the transport operator as a variable means
that the ``scalar'' space $\eh^0_0(M)$ has to be refined as compared to $\eh(M)$
from \cite{vim} by introducing a third argument. However we can connect the present scalar theory to
that in \cite{vim}
by using an appropriate
saturation principle (Proposition \ref{satprop}).
Since generalized tensor fields depend on transport operators,
derivatives with respect to these have to be taken into account as
well.  Fortunately, due to a reduction principle (Lemma
\ref{00dlemma}) these derivatives decouple from the others. This fact
allows to directly utilize results from \cite{vim} without having to
rework the local theory from \cite{found} in the present context.

An important feature of the Colombeau algebras in the scalar case is
an equivalence relation known as ``association'' which coarse grains the
algebra. As we remarked earlier the Schwartz impossibility result
means that one cannot expect that for general {\it continuous}
functions the pointwise product commutes with the embedding. However
this result is true at the level of association. Furthermore in many 
situations of practical relevance elements of the algebra are associated to conventional
distributions. In applications, this feature has the advantage that in
many cases one may use the mathematical power of the differential
algebra to perform calculations but then use the notion of association
to give a physical interpretation to the answer. In
Section~\ref{association} we extend the definition of association from
the scalar to the tensor case and show in particular that the tensor
product of continuous tensor fields commutes with the embedding at the
level of association.

\section{Kinematics}\label{kin}

In this section we introduce the basic space for the forthcoming
spaces of generalized sections. We also define the embeddings of
smooth and distributional sections as well as the action of
diffeomorphisms and the Lie derivative. The main result of this
section is that the Lie derivative commutes with the embedding of
distributions already at the level of the basic space.

We begin by collecting the ingredients for the definition of
the basic space. For the space 
$\ahat$ we refer to Definition \ref{a0hatdef}\,(i), and for details on the 
space of transport operators $\GaTOMN$ to \ref{appendixA}.

\begin{definition}
We define the space of compactly supported transport operators on $M$
by
\[\bhat:=\GacTOMM.\]
\end{definition}

Elements of $\ahat$ resp.\ $\bhat$ will generically be denoted by $\om$ resp.\ $A$.

\begin{definition}\label{basicspace}
The basic space for generalized sections of type $(r,s)$ on the manifold $M$ is
defined as
\[\ehrs(M):=\{u\in\cc^\infty(\ahat\times M\times\bhat,\TrsM)\mid 
u(\om,p,A)\in\TrspM\}.\]
\end{definition}
Here, both $\ahat$ and $\bhat$ are equipped with their natural (LF)-topologies in 
the sense of Section \ref{notation}. Recall that smoothness is to be understood in the sense 
of \cite{KM}. In particular, $u(\om,A):=p\mapsto u(\om,p,A)$ is a member of $\cTrsM$ for
$\om$, $A$ fixed.

We remark that the definition aimed at in \cite{sotonTF} used two-point tensors
(``TP'', see \ref{appendixA}) rather than transport operators (``TO'').
Of course, it is always possible
to 
switch from the ``TO-picture'' to the ``TP-picture'' by means of the isomorphism 
given in (\ref{blob}).

Next we introduce a core technical device for embedding distributional
sections of $\TrsM$ into the basic space.
\begin{definition}Given $A\in\bhat$ we denote by $\Asrpq$
the induced linear map from $\TsrpM$ to $\TsrqM$, i.e., for any
$\tilde t_p=w_1\otimes\dots\otimes w_s\otimes\bet^1\otimes\dots\otimes\bet^r\in\TsrpM$ we
write
\begin{equation}\label{asr}
\Asrpq(\tilde t_p):=\Apq w_1\otimes\dots
\otimes\Aqp\adj\bet^r\in\TsrqM.
\end{equation}
\end{definition}

Obviously, for all $\tit\in\cTsrM$, the map
$q\mapsto \Asr(p,q){\,}\tit(p):=\Asr(p,q)(\tit(p))$ again defines an element of
$\cTsrM$, for every fixed $p\in M$. Moreover, given a second manifold $N$, it
should be clear how to generalize the definition of $\Asr$ to the case of $A\in\GaTOMN$.
Assigning to $(\tit_p,A)\in\TsrpM\times\GaTOMN$ the (smooth) tensor field
$(q\mapsto\Asrpq{\,}\tit_p)\in\cTsrN$ will be referred to as ``spreading
$\tit_p$ over $N$ via $A$''. Dually, assigning to $(t,p,A)\in\cTrsN\times
M\times\GaTOMN$ the map $q\mapsto \Asrpq\adj{\,} t(q)\in\TrspM$ (being
defined on $N$) will be referred to as ``gathering $t$ at $p$ via $A$'' (compare
also Section~\ref{preview}).

\begin{definition}\label{def-embeddings}\itemcr
 \begin{itemize}
  \item[(i)]\label{sirs}
   We define the embedding $\sirs:\cTrsM\to\ehrsM$ of smooth sections of $\TrsM$
   into the basic space $\ehrsM$ by
   \begin{align*}
    \sirs(t)(\om,A)&:= t\\
    \intertext{resp.} \sirs(t)(\om,p,A)&:= t(p).
   \end{align*}
  \item[(ii)]\label{iors}
   We define the embedding $\iors:\DprsM\to\ehrsM$ of distributional sections
   of $\TrsM$ into the basic space $\ehrsM$ via its action on sections $\tilde
t\in\cTsrM$ by
   \[(\iors(v)(\om,A)\cdot\tilde t)(p)=\iors(v)(\om,p,A)\cdot \tilde t(p)
    :=\langle v(.),\big(\Asr(p,.){\,}\tilde t(p)\big)\otimes\om(.)\rangle.
   \]
\end{itemize}
\end{definition}

In contrast to the case of $\sirs(t)$ where $p\in M$ can simply be plugged into
$t\in\cTrsM$, the
variable $p$ is not a natural ingredient of the argument of a
distribution $v\in\DprsM$. Consequently, it only occurs as a parameter in the
definition of $\iors(v)$. Therefore, a $p$-free version of the definition of
$\iors(v)$ giving meaning directly to $\iors(v)(\om,A)$ is not feasible. On the
other hand, the occurrence of $\tit\in\cTsrM$ in the definition of
$\iors(v)$ is essentially due to the fact that $v$ requires tensors
$\tit\otimes\om$ with $\tit\in\cTsrM$ and $\om\in\OmncM$ to be fed
in as arguments. A $\tit$-free version of the definition of $\iors$, however,
is in fact feasible, cf.\ Remark \ref{titfree} below. 

It is clear that $\sirs$ is
linear, taking elements of $\ehrsM$ as values. As to $\iors$, the map
$\Asr$ given by equation~(\ref{asr}) together with $\tilde t\in \cTsrM$
produce a smooth section $\Asr(p,.){\,}\tilde t(p)$ of $\TsrM$, with $p$
as parameter. Hence the action of $v$ on $\Asr(p,.){\,}\tilde
t(p)\otimes\om(.)$ is defined, giving a complex number depending on
$p$. Since $\iors(v)(\om,p,A)$ is linear in $\tit(p)$ and $\tit(p)$ was
arbitrary, $\iors(v)(\om,p,A)\in\TrspM$.  
To prove the fact that $\iors(v)$ is a smooth function of
its three arguments (in the sense of \cite{KM}), hence in fact takes
values in $\ehrsM$ is more delicate and will
be postponed until Section~\ref{sectioniorsvsmooth}. 
Moreover, equipping $\DprsM$ and $\ehrsM$ with the respective topologies of
pointwise convergence (on $\cTsrM\otimes_\CinfM\OmncM$ resp.\ on $\ahat\times
M\times\bhat$), the embedding $\iors$ is linear and bounded, hence smooth
by \cite[2.11]{KM}. By the uniform boundedness principle stated in 
\cite[30.3]{KM}, $\iors$ remains smooth when the range space is 
equipped with the (C)-topology as defined in \ref{appendixB}. 
By similar (in fact, easier) arguments, $\sirs$ is smooth
in the same sense. Finally, injectivity of $\iors$ is a consequence 
of Theorem \ref{T1}\,(iv) below. A direct
proof, not involving the tools of Section \ref{dynamics}, is possible, yet
for the sake of brevity we refrain from including it.

Next we turn to the action of diffeomorphisms on the basic space and the
diffeomorphism invariance of the embedding $\iors$.  To begin with we take
a look at the transformation behavior of the map $\Asrpq$ under
diffeomorphisms.  In fact, as it turns
out in the context of Lie derivatives (cf.\ the proof of the key 
Proposition \ref{lieembedding} below) it is necessary to use a concept
allowing for the simultaneous action of two different diffeomorphisms at
either slot of $A$. This corresponds to the natural action of pairs of
diffeomorphisms\footnote{yet not of arbitrary diffeomorphisms of $\rho:
M_1\times N_1\to M_2\times N_2$, cf.\ the discussion following (\ref{bulletcommute}) and
(\ref{bulletliecommute})} on transport operators as defined in
(\ref{topullback}).

So let $\mu,\nu:\ M\to N$ be diffeomorphisms. By equation
(\ref{topullback}) we have the following induced action on the factors of
$\Asrpq$:
\begin{eqnarray}
 \big((\mu,\nu)^*A\big)(p,q)&=&(T_q\nu)^{-1}\circ A\big(\mu(p),\nu(q)\big)\circ T_p\mu\\
 \big((\mu,\nu)^*A\big)(q,p)\adj&=&(T_q\mu)\adj \circ
       A\big(\mu(q),\nu(p)\big)\adj\circ(T_p\nu)^{-1,\mathrm{ad}},
\end{eqnarray}
and the action on $\Asr$ is given by
\begin{equation}
 (\mu,\nu)^*\big(\Asr\big)(p,q)=\big((\mu,\nu)^*A\big)^s_r(p,q).
\end{equation}

\begin{definition}\label{def:mu}
 Let $\mu:M\to N$ be a diffeomorphism. We define the induced action of $\mu$ on
the basic space,
 $\hat\mu^*:\ehrs(N)\to\ehrsM$, by
 \begin{eqnarray*}
  \big(\hat\mu^* u\big)(\om,p,A)
  &:=&\mu^*\Big(u\big(\mu_*\om,(\mu,\mu)_*A\big)\Big)(p)\\
  &=&\big(T_{\mu(p)}\mu^{-1}\big)^r_s\, u\big(\mu_*\om,\mu
p,(\mu,\mu)_*A\big).
 \end{eqnarray*}
\end{definition}
It is clear that $\hat\mu^* u$ assigns a member of $\TrspM$ to
every $(\om,p,A)$. In order to obtain $\hat\mu^* u\in\ehrsM$, we have to
establish smoothness in $(\om,p,A)$. Observing support properties and
(\ref{seminormssection}) it follows that  
the
linear maps $\om\mapsto\mu_*\om$ and $A\mapsto(\mu,\mu)_*A$ are bounded
(equivalently, smooth, by \cite[2.11]{KM}) with respect to the (LF)-topologies.
Since $u$ and the action of $\Tan\mu^{-1}$ on $\TrsM$ are also smooth, we 
see that indeed $\hat\mu^* u\in\ehrsM$ holds.

To facilitate the proof of the next proposition we introduce the following notation:
For $A\in\GaTOMN$, $\tit\in\cTsrM$, $p\in M$ denote the spreading 
$q\mapsto \Asrpq{\,}\tit(p)$ of $\tit(p)$ via $A$ by $\theta(A,\tit,p)\in\cTsrN$. 
It is easy to check that for $\mu:M\to N$, $\tit\in \cTsrN$ and $A\in\Gamma(\TO(N,N))$,
\[
\big((\mu,\mu)^*\Asr\big)(p,q)\cdot(\mu^*\tit)(p)=\mu^*\big(\theta(A,\tit,\mu p)\big)(q).
\]
Moreover, using $\theta$ we may write for $A\in\bhat$
\[
\big(\iors(v)(\om,A)\cdot\tit\big)(p)=\langle v,\theta(A,\tit,p)\otimes\om\rangle
\qquad(v\in\DprsM).
\]

\begin{proposition}\label{muiota}
 The action of diffeomorphisms commutes with the embedding $\iors$, that is, we
have for all $v\in\DprsN$
 and all diffeomorphisms $\mu:M\to N$
 \begin{equation}\label{muiotacomplete}
  \hat\mu^*\iors(v)=\iors(\mu^* v).
 \end{equation}
\end{proposition}

\begin{proof}
 Let $\mu:M\to N$ be a diffeomorphism and let $v\in\DprsN$, $\om\in\ahat$,
$A\in\bhat$, $\tilde t\in\cTsrM$,
 and $p\in M$. Then we have
 \begin{equation*}
 \begin{split}
  \Big(\big(\hat\mu^*\iors(v)\big)\big(\om,A\big)&\cdot\tilde t\Big)(p)\\
&=\Big(\mu^*\big(\iors(v)(\mu_*\om,(\mu,\mu)_*A)\big)\cdot\tilde
t\Big)(p)\\
   &=\mu^*\Big(\iors(v)\big(\mu_*\om,(\mu,\mu)_*A\big)\cdot\mu_*\tilde
t\Big)(p)\\
   &=\Big(\iors(v)(\mu_*\om,(\mu,\mu)_*A)\cdot\mu_*\tilde
t\Big)(\mu p)\\
   &=\langle v(.),\big(((\mu,\mu)_*\Asr)(\mu p,.){\,}(\mu_*\tilde t)(\mu
p)\big)\otimes\mu_*\om(.)\rangle\\
   &=
     \langle v(.),\mu_*\big(\theta(A,\tit,p)\big)(.)\otimes\mu_*\om(.)\rangle\\
   &=
     \langle (\mu^*v)(..),\theta(A,\tit,p)(..)\otimes \om(..)\rangle\\
=\Big(\big(\iors(\mu^*v)(\om,A)\big)&\cdot\tilde t\Big)(p).
 \end{split}
 \end{equation*}
\end{proof}

Next we turn to the Lie
derivative on the basic space $\ehrs(M)$. To begin with suppose that $X$ is a
smooth and complete vector field on $M$ so that the flow $\FlX$ is defined
globally on $\R\times M$. Then we may
use Definition \ref{def:mu} to define the Lie derivative of $u\in\ehrs(M)$ via
\begin{equation}\label{liecomplete}
 \lh_X u:=\ddttz(\widehat{\FlXt})^*u.
\end{equation}

In the sequel, we will write $\FlhXt$ instead of (the correct)
$\widehat{\FlXt}$, for the sake of line spacing. For $(\LhX u)(\om,p,A)$ to
exist (as an element of $\TrspM$) it suffices
to know that $\tau\mapsto(\FlhXt)^*u(\om,p,A)$ is smooth. However, for
$\LhX u$ to exist and to be a member of $\ehrsM$ (i.e., to be a smooth function
of its
arguments $(\om,p,A)$) we even need that
$(\tau,\om,p,A)\mapsto((\FlhXt)^*u)(\om,p,A)=
T^r_s\FlXmt(u((\FlXt)_*\om,\FlXt p,(\FlXt)_*A))$ is
smooth on $(-\tau_0,+\tau_0)\times\ahat\times M\times\bhat$ for some
$\tau_0>0$. Indeed, by Proposition \ref{lierefinement}\,(1), $(\tau,\om)\mapsto
(\FlXt)_*\om$ and $(\tau,A)\mapsto (\FlXt)_*A$ are smooth, as is
$(\tau,p)\mapsto\FlXt p$. Moreover, a local argument shows that the
action of $\FlX$ on $\R\times\TrsM$ sending $(\tau,v)$ to $T^r_s\FlXmt v$ is
smooth (this in fact also secures that the assumptions of Proposition
\ref{lierefinement}\,(1) are satisfied). Together with smoothness of $u$, we indeed obtain $\LhX
u\in\ehrsM$.

Note that in order to have $(\FlhXt)^*u$ defined as a member of
$\ehrsM$ even for only one particular value of $\tau$ we need $(\FlXt)_*\om$, $\FlXt
p$ and $(\FlXt)_*A)$ defined for this $\tau$ and for {\it all} $\om,p,A$,
irrespective of the size of the supports of $\om$ and $A$ (as to $p$, we
could content ourselves with some open subset of $M$). This exhibits the
important role of completeness of $X$ in the geometric approach to the Lie
derivative on the basic space taken by (\ref{liecomplete}).

As a direct consequence of (\ref{liecomplete}) and
Proposition~\ref{muiota} we have
that the Lie derivative
commutes with the embedding, i.e., we have for all smooth and complete vector
fields
$X$ and all $v\in\DprsM$
\begin{equation}\label{lieiotacomplete}
   \lh_X\iors(v)=\iors(\Lie_Xv).
\end{equation}
The technical background of passing from (\ref{muiotacomplete})
to (\ref{lieiotacomplete}) again involves calculus in convenient
vector spaces. For $v\in\DprsM$, $\tau\mapsto(\FlXt)^*v$ is smooth and $\LX
v=\ddtz(\FlXt)^*v$ holds with respect to the weak topology, due to (1)(ii) of
Proposition \ref{lierefinement}. Also, $\LhX u=\ddtz(\FlhXt)^*u$ for
$u=\iors(v)$ with respect to the topology of pointwise convergence. Applying
the chain rule \cite[3.18]{KM} to the function $\tau\mapsto
(\iors\circ(\FlXt)^*)v$, we obtain from 
(\ref{muiotacomplete}):
\begin{align*}
\LhX\iors(v)&=\ddtz((\FlhXt)^*(\iors(v)))
=\ddtz(\iors((\FlXt)^*v))
=\iors(\ddtz(\FlXt)^*v)\\
&=\iors(\Lie_Xv).
\end{align*}

For the purpose of extending the definition of the Lie derivative to arbitrary
smooth vector fields, by an application of the chain rule we obtain from
(\ref{liecomplete})
\begin{equation}\label{liechain}
 \begin{array}{rcl}
 \lefteqn{(\lh_Xu)(\om,p,A)}\\[5pt]
 &&=\Lie_X(u(\om,A))(p)-{\mathrm d}_1u(\om,p,A)(\Lie_X\om)-{\mathrm d}_3u(\om,p,A)(\LX A),
 \end{array}
\end{equation}
where we recall from \ref{appendixA} that $\LX A$ is an abbreviation for
$\LXX A$.

Note that we do not need full
manifold versions of local results of infinite-dimensional calculus as, e.g.,
the chain rule \cite[3.18]{KM} since we can replace $\ahat\times M\times \bhat$
by $\hat\ca_{00}(M)\times W\times \bhat$ when dealing with local issues on $M$.
Here, $\hat\ca_{00}(M)$ denotes the linear subspace of $\OmncM$ parallel to
$\ahat$ and $W$ is (diffeomorphic to) some open subset of $\R^n$. In this way,
$\diff_1u(\om,p,A)(\eta)$, for example, can be interpreted locally as $\diff
u(\om,p,A)(\eta,0,0)$ in the above sense or, equivalently, as
$\diff(u^\vee(p,A))(\om)(\eta)$ where $u^\vee(p,A)(\om)=u(\om,p,A)$.

The scalar analogue of (\ref{liechain}) first appeared in the local
setting of \cite[Rem.\ 22]{Jel}, where it arises as an operational consequence of
Jel{\'{\i}}nek's approach; see also the discussion of the scalar case in 
\cite[p.\ 4]{ro-bed}. Here, however, it is a direct consequence of our 
natural choice of definitions, in case $X$ is complete.
In the general case we turn equation (\ref{liechain}) into a definition.

\begin{definition}\label{lh}
 Given a smooth vector field $X$ on $M$ we define the Lie derivative $\lh_X$
with respect to $X$ on the basic
 space by
 \begin{equation}\label{lhformula}
 \begin{array}{rcl}
  (\lh_Xu)(\om,p,A)&=&\Lie_X(u(\om,A))(p) \\[5pt]
  &&-{\mathrm d}_1u(\om,p,A)(\Lie_X\om)-{\mathrm d}_3u(\om,p,A)(\LX A).
 \end{array}
 \end{equation}
\end{definition}

To see that the
first term on the right hand side of (\ref{lhformula}) actually defines an
element of $\ehrsM$ we note that by  
Corollary \ref{expCF} (resp.\ Lemma
\ref{reprbasic} and Remark \ref{CDF}), $\ehrsM$ is linearly isomorphic to
$\Cinf(\ahat\times
\bhat,\cTrsM)$ where $\cTrsM$ carries the (F)-topology. Since $\LX$ is linear
and bounded (hence smooth) on
$\cTrsM$, the map
$(\om,A)\mapsto \LX(u(\om,A))$ is smooth from $\ahat\times\bhat$
into $\cTrsM$. By smoothness of $\om\mapsto \LX \om$ resp.\ $A\mapsto \LX A$,
also the second and the third term define members of the basic space.
Note that in the scalar case $r=s=0$, the first term takes
the form of a directional derivative as well, to wit, $\diff_2u(\om,p,A)(X)$.

For complete $X\in \XM$,
also the three individual terms at the right hand side of (\ref{lhformula})
arise in a geometric way:
By slightly
generalizing the constructions of $\FlhXt$ and $\LhX$ for complete $X$ we can
define, for given complete vector fields $X,Y,Z\in\XM$, also
$(\FlhXYZt)^*u:=T^r_s\FlYmt(u((\FlXt)_*\om,\FlYt p,(\FlZt)_*A))$ and $\LhXYZ\,
u:=\ddtz(\FlhXYZt)^*u$. From
this we obtain, by the chain rule,
\begin{equation}\label{lhattriple}
\LhX u=\LhXzz u + \LhzXz u+\LhzzX u.
\end{equation} 
All the smoothness arguments
referring to $\LhX$ equally apply to each
individual term of this decomposition. Moreover, it is clear that the three terms
occurring on the right hand side of (\ref{lhformula}) correspond to $\LhzXz$,
$\LhXzz$, $\LhzzX$, respectively. Therefore we will retain this notation also
in the case of arbitrary vector fields. It is immediate that $(\LhzXz
u)(\om,p,A)=\LX(u(\om,A))(p)$ where $\LX$ denotes the classical Lie derivative
on $\cTrsM$. On the other hand, in $\LhXzz u$ and $\LhzzX u$ the $p$-slot is
fixed and the differentiation process involves only the fiber
$\TrspM\cong(\R^n)^{r+s}$ as range space. At several places it will be
important to split $\LhX$ in the way just indicated (to wit, in \ref{lieembedding},
\ref{leibnizbasic}, \ref{t4t5lemma}, \ref{t4t5}). 

In the case of an arbitrary smooth vector field $X$ some work
is needed to prove that the embedding commutes with the Lie derivative.
In \ref{appendixA} (see (\ref{tolie}) and the remark following
(\ref{bulletliecommute})) we have defined the Lie derivative of 
$A\in \Gamma(\TO(M,N))$ with respect to any smooth vector fields $X,Y$. We now set
\begin{eqnarray*}
\LXY\Asr&=&\LXY A\otimes\dots\otimes (\Aad\circ\fl)+\dots\\
 &&\qquad\qquad\qquad \dots+A\otimes\dots\otimes \LXY (\Aad\circ\fl),
\end{eqnarray*}
where $\fl$ denotes the flip $(p,q)\mapsto(q,p)$. 
This conforms to viewing $\Asr$ as a section of the vector bundle over
$M\times N$ with fiber $L(\TsrpM,\TsrqN)$ at $(p,q)$ and the obvious transition
functions. We will use the notation $\LXY(\Asrpq{\,}\tit(p))$ 
rather than (the more precise) $\LXY(\Asrpq(\mathrm{pr}_1^*\tit(p,q)))$ 
for the Lie
derivative of the section $(p,q)\mapsto\Asrpq{\,}\tit(p)$ of the pullback bundle.

\begin{proposition}\label{lieembedding}
 The Lie derivative commutes with the embedding $\iors$, i.e., we have, for all smooth
vector fields $X$ and all $v\in\DprsM$,
 \begin{equation}
  \lh_X\iors(v)=\iors(\Lie_Xv).
 \end{equation}
\end{proposition}

\begin{proof}
By definition we have for all $\om\in\ahat$, $p\in M$ and $A\in\bhat$
\begin{align}\label{chr123}
\big(\LhX\iors(v)\big)(\om,p,A)=\LX&\big(\iors(v)(\om,A)\big)(p)\nn\\
&-{\mathrm d}_1\big(\iors(v)\big)(\om,p,A)(\LX\om)\\
&-{\mathrm d}_3\big(\iors(v)\big)(\om,p,A)(\LX A).\nn
\end{align}
We proceed by applying each term individually to $\tilde t\in{\mathcal
T}^s_r(M)$. We find by the chain rule
\begin{eqnarray}\label{1}
\lefteqn{\LX\big(\iors(v)(\om,A)\big)(p)\cdot\tilde t(p)}\nn\\
&&=\LX\big(\iors(v)(\om,A)\cdot \tilde
t\big)(p)-\big(\iors(v)(\om,A)\cdot\LX\tilde t\big)(p)\nn\\
&&=\LX\Big(\langle v(.),\Asr(p,.){\,}\tilde
t(p)\otimes\om(.)\rangle\Big)-\langle v(.),\Asr(p,.){\,}\LX\tilde
t(p)\otimes\om(.)\rangle\nn\\
&&=\langle v(.),\Big(\LXz\big(\Asr(p,.){\,}\tilde
t(p)\big)-\Asr(p,.){\,}\LX\tilde t(p)\Big)\otimes\om(.)\rangle\nn\\
&&=\langle v(.),(\LXz\Asr)(p,.){\,}\tilde t(p)\otimes\om(.)\rangle.
\end{eqnarray}
To see that $\LX\circ v=v\circ\LXz$ in the above calculation,
set $w(p,q):=\Asrpq{\,}\tit(p)$. Then we have $w\in\Ga_\mathrm{c}(M\times M,\prtS\TsrM)$,
with
$w^\vee\in\Cinf(M,\cTsrM)$ corresponding to $w$ according
to Lemma \ref{kriegl}.
On $\Ga_\mathrm{c}(M\times M,\prtS\TsrM)$ flow actions $(\FlXt,\FlYt)^*$
and Lie derivatives $\LXY$ are defined in complete analogy to the case of
transport operators (\ref{appendixA}). Since $\supp w\subseteq\supp A$ there
exists $\tau_0>0$ such that $(\FlXt,\FlYt)^*w$ is defined on $M\times M$ for all $\tau$
with $|\tau|\leq\tau_0$. By Proposition \ref{lierefinement}\,(2),
$\tau\mapsto (\FlXt,\FlYt)^*w$ is smooth into $\Ga_\mathrm{c}(M\times
M,\prtS\TsrM)$ and
$\LXY w=\ddtz (\FlXt,\FlYt)^*w$ in the (LF)-sense. 
Setting $Y=0$, it follows that $(\LXz\,w)^\vee(p)=\ddtz w^\vee(\FlXt p)$ in the (LF)-sense
in $(\ct^s_r)_\mathrm{c}(M)$.
From this we finally arrive at
\begin{align*}
\LX\lgl v(.),\Asr(p,.){\,}\tilde t(p)\otimes\om(.)\rgl
&=\LX\lgl v,w^\vee(p)\otimes\om\rgl\\
&=\ddtz\lgl v, w^\vee(\FlXt p)\otimes\om\rgl\\
&=\lgl v,\ddtz w^\vee(\FlXt p)\otimes\om\rgl\\
&=\lgl v,(\LXz\,w)^\vee(p)\otimes\om\rgl\\
&=\lgl v(.),\LXz(\Asr(p,.){\,}\tilde
t(p))\otimes\om(.)\rgl.
\end{align*}

For the second term on the right hand side of equation (\ref{chr123}) we obtain,
using the fact that $v$ is linear and continuous.
\begin{eqnarray}\label{2}
\lefteqn{{\mathrm d}_1\big(\iors(v)\big)(\om,p,A)(\LX\om)\cdot\tilde t(p)}\nn\\
&&=d\big[\om\mapsto\langle v(.),\Asr(p,.){\,}\tilde
t(p)\otimes\om(.)\rangle\big](\LX\om)\nn\\
&&=\langle v(.),\Asr(p,.){\,}\tilde t(p)\otimes\LX\om(.)\rangle\nn\\
&&=\langle v(.),\LzX\big(\Asr(p,.){\,}\tilde t(p)\otimes\om(.)\big)\rangle
-\langle v(.),\LzX\big(\Asr(p,.){\,}\tilde t(p)\big)\otimes\om(.)\rangle\nn\\
&&=-\langle\LX v(.),\Asr(p,.){\,}\tilde t(p)\otimes\om(.)\rangle-
\langle v(.),(\LzX\Asr)(p,.){\,}\tilde t(p)\otimes\om(.)\rangle\nn\\
&&=-\iors(\LX v)(\om,p,A){\,}\tilde t(p)-\langle
v(.),(\LzX\Asr)(p,.){\,}\tilde t(p)\otimes\om(.)\rangle.
\end{eqnarray}
Since $A\mapsto\Asr$ is the composition of a multilinear map with the diagonal map,
the third term on the right hand side of equation (\ref{chr123}) gives
\begin{eqnarray}\label{3}
\lefteqn{{\mathrm d}_3\big(\iors(v)\big)(\om,p,A)(\LXX A)\cdot\tilde t(p)}\nn\\
&&\hphantom{mmmmm}=d\big[A\mapsto\langle v(.),\Asr(p,.){\,}\tilde
t(p)\otimes\om(.)\rangle\big](\LXX A)\nn\\
&&\hphantom{mmmmm}=\langle v(.),\big(\LXX(\Asr)\big)(p,.){\,}\tilde
t(p)\otimes\om(.)\rangle.
\end{eqnarray}
Combining equations (\ref{1}), (\ref{2}), and (\ref{3}) we obtain the
result.
\end{proof}

Standard operations of tensor calculus carry over to
elements of $\ehrsM$. Thus, for $u_1 \in \ehrsM$, $u_2 \in \hat{\mathcal{E}}^{r'}_{s'}(M)$ we define the
tensor product $u_1\otimes u_2$ by
$$
(u_1\otimes u_2)(\omega,p,A) := (u_1(\omega,A)\otimes u_2(\omega,A))(p).
$$
Then clearly $u_1\otimes u_2 \in  \hat{\mathcal{E}}^{r+r'}_{s+s'}(M)$.
Moreover, if $C^i_j: \cTrsM \to \mathcal{T}^{r-1}_{s-1}(M)$ is any contraction
then for $u\in \ehrsM$ we define $C^i_j(u)\in \hat{\mathcal E}^{r-1}_{s-1}(M)$ by
$$
C^i_j(t)(\omega,p,A) := C^i_j(t(\omega,A))(p).
$$
Contraction $u_1\cdot u_2$ of dual fields $u_1\in\ehrsM$ and $u_2\in \hat{\mathcal E}^s_r(M)$
is then defined as a composition of the above operations. Notationally suppressing
the embedding $\sirs$, we obtain the special case $u\in \ehrsM$, $\tit \in {\mathcal T}^s_r(M)$:
$$
(u\cdot \tit)(\omega,p,A) := (u(\omega,A)\cdot \tit)(p).
$$
\begin{proposition}\label{leibnizbasic}
The Lie derivative $\LhX$ acting on tensor products of arbitrary fields and
on contractions of dual fields satisfies the Leibniz rule, i.e. we have
\begin{align*}
\LhX(u_1\otimes u_2)&=(\LhX u_1)\otimes u_2+u_1\otimes(\LhX u_2)
          &(u_1\in\ehrsM,\ u_2\in\hat\ce^{r'}_{s'}(M))\\
\LhX(u_1\cdot u_2)&=(\LhX u_1)\cdot u_2\hphantom{e}+\,u_1\cdot(\LhX u_2)
          &(u_1\in\ehrsM,\ u_2\in\hat\ce^s_r(M)).
\end{align*}
\end{proposition}
\begin{proof}
We consider the three defining terms adding up to $\LhX$
according to (\ref{lhattriple}) separately. As to $\LhzXz$, we have
$\LhzXz(u_1\otimes u_2)(\om,p,A)=\LX((u_1\otimes u_2)(\om,A))(p)=
\LX(u_1(\om,A)\otimes u_2(\om,A))(p)$; for the
latter the classical Leibniz rule of course holds. Concerning
$\LhXzz$, we note that the corresponding terms are but directional derivatives
of the smooth functions $u_1\otimes u_2$ resp.\ $u_1$ resp.\ $u_2$. Since the
first originates from the remaining two by composition with a bounded
bilinear (equivalently, smooth, by \cite[5.5]{KM}) map the Leibniz rule holds
also in this case. {\it Mutatis mutandis}, the same arguments apply to
$\LhzzX$. The proof for the contraction of dual fields proceeds along the same
lines.
\end{proof}

For the proof that $\LhX$ respects moderateness resp.\
negligibility in Section \ref{dynamics} we will need an explicit expression for
$\diff_3\big((\om,p,A)\mapsto\LX (u(\om,A))(p)\big)(B)$, which will imply that 
directional
derivatives with respect to slots 1 resp.\ 3 can be interchanged with Lie
derivatives with respect to slot 2. We consider the case of
slots 2 and 3; the argument for slots 1 and 2 is similar. To this end, let
$\phi:\ehrsM\to\Cinf(\ahat\times\bhat,\cTrsM)$ denote the linear isomorphism
given by Lemma \ref{reprbasic}. Using 
$\phi$, the map $\LhzXz$
sending $u\in\ehrsM$ to ($(\om,p,A)\mapsto\LX( u(\om,A))(p))$
can be written as $\LhzXz=\phi^{-1}\circ(\LX)_*\circ\phi$ where 
$(\LX)_*(\tilde u):=\LX\circ \tilde u$ for $\tilde
u\in\Cinf(\ahat\times\bhat,\cTrsM)$. Recall that
$(\LX)_*$ maps $\Cinf(\ahat\times\bhat,\cTrsM)$
linearly into 
itself, due to the
boundedness (equivalently, smoothness) of $\LX$ on $\cTrsM$, with respect to the
(F)-topology.

\begin{lemma}\label{difflx}
For $B\in\bhat$ let $\diff_B:\ehrsM\to\ehrsM$ 
denote the directional
derivative defined by $(\diff_Bu)(\om,p,A)=\diff_3u(\om,p,A)(B)$. Then for 
any $X\in \mathfrak{X}(M)$,
\begin{equation*}
\diff_B\circ\LhzXz=\LhzXz\circ\diff_B,
\end{equation*}
hence, neglecting $\phi$,
$$\diff_3\big((\om,p,A)\mapsto\LX
(u(\om,A))(p)\big)(B)=\LX(\diff_2u(\om,A)(B))(p).
$$
\end{lemma}
\begin{proof}
For $B\in\bhat$ let $\tilde\diff_B:\Cinf(\ahat\times\bhat,\cTrsM)
\to\Cinf(\ahat\times\bhat,\cTrsM)$ denote the directional derivative defined by
$(\tilde \diff_B\tilde u)(\om,A)=\diff_2\tilde u(\om,A)(B)$.
We will show the following two relations:
\begin{align}
\tilde\diff_B\circ(\LX)_*&=(\LX)_*\circ\tilde\diff_B\label{difflx1}\\
\phi\circ\diff_B\label{difflx2}&=\tilde\diff_B\circ\phi. 
\end{align}
Transferring (\ref{difflx1}) by means of (\ref{difflx2}) and the defining
relation
\begin{align}
\phi\circ\LhzXz&=(\LX)_*\circ\phi
\end{align}
from $\Cinf(\ahat\times\bhat,\cTrsM)$ to $\ehrsM$ will
accomplish the proof.

(\ref{difflx1}) is a consequence of the chain rule
\cite[3.18]{KM}: For $\tilde u\in\Cinf(\ahat\times\bhat,\cTrsM)$, we obtain
$$(\tilde\diff_B\circ(\LX)_*)(\tilde u)
=\tilde\diff_B\,(\LX\circ \tilde u)
=\LX\circ(\tilde\diff_B \,\tilde u)
=((\LX)_*\circ\tilde\diff_B)(\tilde u).$$
(\ref{difflx2}), on the other hand, follows from the continuity of point
evaluations on the (F)-space $\cTrsM$: For $u\in\ehrsM$, we have
\begin{align*}
(\tilde\diff_B\,\phi\, u)(\om,A)(p)&=\lim\limits_{\tau\to0}
   \textstyle{\frac{1}{\tau}}\big[(\phi\, u)(\om,A+\tau B)-
                (\phi\, u)(\om,A)\big](p)\\
&=\lim\limits_{\tau\to0}
   \textstyle{\frac{1}{\tau}}\big[(\phi\, u)(\om,A+\tau B)(p)-
                (\phi\, u)(\om,A)(p)\big]\\
&=(\diff_B\,u)(\om,p,A)\\
&=(\phi\,\,\diff_B\,u)(\om,A)(p).
\end{align*}
Note that the first limit above is in the (F)-space $\cTrsM$ while the
second one refers to Euclidean topology in $\TrspM\cong(\R^n)^{r+s}$.
\end{proof}

\begin{remark}\label{sheafremark} 
To conclude this section we mention without proof some further
properties of $\ehrs$.
\begin{itemize}
\item[(i)] Since elements of $\ahat$ and $\bhat$ are compactly supported, there is an
obvious notion of restriction of any $u\in \ehrsM$ to open subsets of $M$.
\item[(ii)] For any $v\in \DprsM$, $\iors(v)$ vanishes on the same open subsets
of $M$ as $v$ does, hence $\supp \iors(v)$ is well-defined and equals
$\supp v$.
\item[(iii)] For coverings $\mathcal{U}$ of $M$ directed by inclusion ($U_1,\, U_2\in
\mathcal{U}$ $\Rightarrow$ $\exists U_3 \in \mathcal{U}$, $U_3\supseteq U_1\cup U_2$)
the usual sheaf properties hold for $\ehrs$. A local geometrical definition of the Lie derivative
on $\ehrsM$ with respect to arbitrary smooth vector fields can be based on this.
\end{itemize}
\end{remark}

\section{Smoothness of embedded distributions}\label{sectioniorsvsmooth}
The seemingly innocuous statement of $\iors(v)$ being smooth is, in fact, a deep
result involving the entire range of results assembled in Appendix B.
The difficulties in proving it reflect the interplay between the apparatus of
(smooth as well as distributional) differential geometry and calculus on
(infinite-dimensional) locally convex spaces. Observe that in the scalar case
treated in \cite{vim}, the question of smoothness of $\io(v)$ (for $v\in
\DpM$) reduces to the trivial statement that $\io(v)=v\circ\pro$, being
linear and continuous (hence bounded), is smooth on $\ahat\times M$.

The main difficulty becomes clear from the fact that $\iors$ has to
bridge the ``topology gap'' between two worlds: Its argument $v\in\DprsM$
relates to the domain of linear spaces carrying (F)- resp.\ (LF)-topologies and
their dual spaces whereas the relevant results on the basic space $\ehrsM$
(of which $\iors(v)$ is a member) hold
with respect to the canonical convenient vector space topology on spaces of
smooth functions denoted by the term (C)-topology in the sequel (cf.\
\cite{KM}).

One crucial step of the proof 
consists in getting a
handle on the pa\-ra\-me\-tri\-zed tensor field $\theta(A,\tit,p):
q \mapsto \Asr(p,q){\,}\tilde t(p)$ occurring in the definition of $\iors$,
the spreading of $\tit(p)$ (over $M$) via $A$ (cf.\ Section \ref{kin}). 
Recall that $\theta$ has already 
been used in the proof of Proposition \ref{muiota}. 
Proving the smoothness of
$\theta(A,\tit,p)$ as a function of its three arguments will be
accomplished by Lemmata \ref{fstern} resp.\ \ref{evsr} and Corollary
\ref{corevbar}
of Appendix B, providing the necessary information on spaces of sections
of vector bundles over (possibly infinite-dimensional) manifolds. 
For decomposing $\theta$
into manageable parts, we introduce the following terminology:
\begin{itemize}
\item
For manifolds $B'$ (possibly infinite-dimensional), $B$, a smooth map $f:B'\to B$ and a vector
bundle $E\stackrel{\pi}{\to}B$ let $f^*E$ denote the pullback bundle of $E$
under $f$ (cf.\ \ref{appendixA} and the discussion following Remark \ref{evsr}). 
We define the pullback operator $f^*:\Ga(B,E)\to\Ga(B',f^*E)$ as
follows: For $u\in\Ga(B,E)$ given, the pair $(\id_{B'},u\circ f)$ 
induces the smooth section $f^*u:B'\to f^*E$, 
$f^*u= p\mapsto (p,u(f(p)))$ of $f^*E$.
\item
For formalizing the spreading process based on the action of a transport
operator on tangent vectors we introduce the operator
\begin{eqnarray*}\label{evga10}
\ev^1_0:\, \Ga(\TO(M,N))\times\Ga(\poTM)&\kern-3pt\to&\kern-3pt\Ga(\ptTN)
\end{eqnarray*}
by defining, for $(p,q)\in M\times N$,
$$
(\ev^1_0(A,\ti\xi))(p,q):=\big((p,q)\;,\,A(p,q)
       {\,}\bar\xi(p , q)\big)
$$
where $A\in\Ga(\TOMN)$ and
$\ti\xi\in\Ga(\poTM)$ is of the form $(p,q)\mapsto((p,q),\bar\xi(p,q))$.
By $\evsr$ we denote the obvious extension of $\ev^1_0$ to the $(s,r)$-case,
with $\Asr$ acting fiberwise on a section $\ti\xi\in\Ga(\proS\TsrM)$.
\item
For manifolds $M,N$ and a vector bundle $E\stackrel{\pi}{\to}N$, we define the
operator $\evbar:\Ga(M\times N,\prtS E)\times M\to\Ga(N,E)$ by
$\evbar(u,p)(q):=\prt'(u(p,q))$ where $\prt':\prtS E\to E$ denotes the canonical
projection of the pullback bundle $\prtS E$ (cf.\ Appendix A).
\end{itemize}
With this notational machinery available, we are able to factorize
$\theta$ as
\begin{equation}\label{spreading}
\theta(A,\tit,p)=\Asr(p,.){\,}\tilde t(p)=\evbar(\evsr(A,\proS\tit),p).
\end{equation}

\begin{proposition}[Smoothness of $\iors(v)$]\label{iorsvsmooth}
 For any $v\in\DprsM$ the function $\iors(v)$ introduced in
Definition \ref{iors}\,(ii) is smooth, hence a member of $\ehrsM$.
\end{proposition}

In addition to employing the factorization (\ref{spreading}),
the proof of Proposition \ref{iorsvsmooth} is based upon
an equivalent representation of the basic space $\ehrsM$
as a space of smooth four-slot functions taking $\om,p,A,\tit$ as arguments.
The benefit of such a representation should be clear from 
Definition \ref{iors}.
We abbreviate the property of being ``$\CinfM$-linear in the
$k$-th slot'' as being ``$\cc_k$-linear''.

\begin{lemma}\label{reprbasic}
The basic space $\ehrsM$ has the following equivalent representations which are
mutually isomorphic as linear spaces:
\begin{enumerate}
\item [(0)] $\{u\in\Cinf(\ahat\times M\times\bhat\,,\,\TrsM)\mid
u(\om,p,A)\in\TrspM\}$
\item[(1)] $\Ga(\ahat\times M\times\bhat,\prtS\TrsM)$
\item [(2)] $\Cinf(\ahat\times\bhat\,,\,\cTrsM)$
\item [(3)] $\{u\in\Cinf(\ahat\times\bhat\,,\,
\Cinf(\cTsrM,\CinfM))\mid u(\om,A) \textrm{ is $\cc_1$-linear}\}$
\item [(4)] $\{u\in\Cinf(\ahat\times\bhat\times
\cTsrM\times M\,,\,\R)\mid u \textrm{ is $\cc_3$-linear}\}$
\item [(5)] $\{u\in\Cinf(\ahat\times M\times\bhat\,,\,
\Cinf(\cTsrM,\R))\mid u(\om,p,A) \textrm{ is $\cc_1$-linear}\}$.
\end{enumerate}
The relations between corresponding elements $t^{[i]}$ from space ($i$),
respectively, of the above list are given by
\begin{align*}
((\om,p,A),u^{[0]}(\om,p,A))&=u^{[1]}(\om,p,A)\\
u^{[0]}(\om,p,A)&=u^{[2]}(\om,A)(p)\\
u^{[2]}(\om,A)\cdot\tit&=
   u^{[3]}(\om,A)(\tit)\hphantom{n}
                                    \\
u^{[3]}(\om,A)(\tit)(p)&=u^{[4]}(\om,A,\tit,p)=u^{[5]}(\om,p,A)(\tit)
\end{align*}
where $\om\in\ahat$, $p\in M$, $A\in\bhat$ and $\tit\in\cTsrM$.
\end{lemma}

\begin{remark}\label{bemeins}
(i) Note that, even for finite-dimensional $M$, (1) requires a theory of vector
bundles over infinite-dimensional smooth manifolds (in fact, over $\ahat\times
M\times\bhat$ in the case at hand); see the remarks preceding Lemma
\ref{kriegl} in Appendix B.

(ii) In order to give meaning to the various notions of smoothness occurring in
(0)--(5) of the preceding lemma,
we have to specify appropriate locally convex topologies resp.\
bornologies on the spaces involved. To this end,
we equip $\ahat$ and $\bhat$ with their
respective (LF)-topologies and $\cTsrM$ with its usual (F)-topology (recall our
convention stated in Section \ref{notation} for $M$ non-separable). On the
other hand, whenever a space of smooth functions such as $\Cinf(.,..)$ or
$\cTrsM$ appears at the second slot of some
$\Cinf$(...,....) it
carries the locally convex topology (C) defined in Appendix B. This is
indispensable for legitimizing the applications of
\cite[27.17]{KM} resp.\ of Lemma \ref{kriegl}
which are to follow. Whenever an explicit declaration of the
topology in question is needed we will use subscripts as, e.g., in $\CinfM_F$
resp.\ $\CinfM_C$.

(iii) In (4), $\cc_3$-linearity (resp.\ $\cc_1$-linearity in (5)), in fact, are
to be understood as
\begin{align*}
u(\om,A,f\cdot\tit,p)&:=u(\om,A,\tit,p)\cdot f(p)\\[-12pt]
 \intertext{resp.\ }\\[-30pt]
u(\om,p,A)(f\cdot\tit)&:=u(\om,p,A)(\tit)\cdot f(p)
\end{align*}
with $f\in\CinfM$, in order to guarantee compatibility with (3) which, in turn
could also be written as
$\Cinf(\ahat\times\bhat\,,\,\Lin^b_\CinfM(\cTsrM,\CinfM))$ where
$\Lin ^b_\CinfM(.,..)$ denotes the subspace of
$\Cinf(.,..)$ of $\CinfM$-linear
bounded (hence smooth, cf.\ \cite[2.11]{KM}) functions. \cite[5.3]{KM} shows
that this does not cause any ambiguity as to the meaning of a subset
of $\Lin^b_\CinfM(\cTsrM,\CinfM)$ being
(C)-bounded and, hence, of a mapping into that space being smooth.

(iv) (0) is precisely the expression for $\ehrsM$ given in Definition
\ref{basicspace}; (2) corresponds to writing $u(\om,A)(p)=u(\om,p,A)$,
introduced after Definition \ref{basicspace}.
As to the other items, cf.\ Remark \ref{explain05}.
\end{remark}

\medskip
{\bf Proof of Lemma \ref{reprbasic}.}
The form of the pullback bundle $\prtS\TrsM$
occurring in (1) shows that its smooth sections are precisely given by maps as
in (0) (compare also the discussion preceding Corollary \ref{corevbar}), the
correspondence being as stated in the lemma. The equality of (1) and (2)
as well as the relation
$u^{[1]}(\om,p,A)=((\om,p,A),u^{[2]}(\om,A)(p))$ are immediate from Lemma
\ref{kriegl}.
Moving on to (2)--(5), it is clear from the given relations that passing from
$u^{[i]}$ to
$u^{[i+1]}$ resp.\ vice versa yields maps having
appropriate domains, ranges (disregarding smoothness) and algebraic properties
for the cases $i=3,4$. From \cite[27.17]{KM} we conclude that
$u^{[4]}$ is smooth if and only if $u^{[5]}$ is smooth and takes {\it smooth}
functions on $\cTsrM$ as values on triples $(\om,p,A)$,
due to $u^{[5]}=(u^{[4]})^\vee$ in the terminology of \cite[3.12]{KM}, with
respect to the variable $\tit$ getting
separated from the variables $\om,p,A$. A twofold application of the same
argument achieves the transfer of smoothness between $u^{[3]}$ and $u^{[4]}$.

It remains to discuss $i=2$. Observe that the assignment $t\mapsto(\tit\mapsto
t\cdot\tit)$
embeds $\cTrsM$ into the space of $\CinfM$-linear bounded (hence smooth)
maps from $\cTsrM_F$ into $\CinfM_F$. (C) being
weaker than the (F)-topology on $\CinfM$, we obtain
smoothness from $\cTsrM_F$ into $\CinfM_C$,
as required for (3). Now we have the chain of inclusions
$$
\cTrsM\subseteq\Lin^b_\CinfM(\cTsrM_F,\CinfM_C)
         \subseteq\Lin_\CinfM(\cTsrM,\CinfM),
$$
where $\Lin_\CinfM(.,..)$ denotes the respective space of {\em all} 
$\CinfM$-linear maps. 
$\cTrsM$ being isomorphic to $\Lin_\CinfM(\cTsrM,\CinfM)$ via the assignment specified above,
all three spaces in the chain are, in fact,
identical. Moreover, by Theorem \ref{actingondual} of Appendix B, the
corresponding (C)-topologies on
$\cTrsM$ and $\Lin^b_\CinfM(\cTsrM,\CinfM)$ have the same bounded sets,
showing that also the
spaces given by
(2) and (3) are identical.
\ep
\medskip

\begin{remark}\label{explain05}
Observe that each of (0)--(5) serves a distinct yet prominent purpose:
(0) was used in introducing the basic space (cf.\ Definition
\ref{basicspace}), with a view to representing best the intuitive picture of a
(representative of a) generalized tensor field as a section of $\TrsM$ depending
on additional parameters $\om,A$. The drawback of (0) consists in the fact that
the range space $\TrsM$ is not a linear space, hence
$\Cinf(\ahat\times M\times\bhat\,,\,\TrsM)$ is not a (convenient) vector space.
Making up for this deficiency, (1) opens the gates to applying the apparatus of
infinite dimensional differential geometry as provided by \cite[Sec.\ 27--30]{KM}. 
In particular, it paves the way to (2)--(5):
(2) is optimal for defining $\sirs$ by $\sirs(t)(\om,A):=t$
(cf.\ Definition \ref{iors}\,(i)). (3) provides the crucial intermediate
step bringing $\tit$ into play. Concerning an
explicit definition of $\iors$, (0)--(2) are flawed by not providing a slot
for inserting $\tit$. Among (3)--(5), (5) (which, indeed, was used in (ii) of
Definition \ref{iors}) seems to be the optimal choice due to having $\om,p,A$ as
primary arguments and $\tit$ to be acted upon, yet also (3) and (4) are capable
of doing the job. As
to proving smoothness of $\iors(v)$, finally, (0) resp.\ (1) resp.\ (4) are to
be preferred, relying exclusively on the well-known (F)- resp.\ (LF)-topologies.
Altogether, for establishing smoothness of
$\iors(v)$, (4) turns out to be the best choice.
\end{remark}

Now, having Lemma \ref{reprbasic} at our disposal, we shall
demonstrate that $\iors(v)$ is smooth.

\medskip
{\bf Proof of Proposition \ref{iorsvsmooth}.}
By Lemma \ref{reprbasic}
it suffices to show that for $v\in\DprsM$, $\iors(v)$ is a member of (4) as
defined in Lemma \ref{reprbasic}. The crucial step of the proof consists in
establishing $\theta:(A,\tit,p)\mapsto(q\mapsto \Asr(p,q){\,}\tit(p))$ to be
smooth as a map $\theta:\bhat_{LF}\times\cTsrM_F\times M \to \cTsrM_F$. Once
this has been achieved, it suffices to note that the bilinear map
$$\cTsrM_F\times\OmncM_{LF}\ni(\ti\tit,\om)\mapsto\ti\tit\otimes\om\in
     (\cTsrM\otimes_\CinfM\OmncM)_{LF}$$
is bounded, hence smooth (which is immediate from an inspection of the seminorms
defining the (F)- resp.\ (LF)-topologies on section spaces of vector bundles,
cf.\ (\ref{seminormssection})) to conclude, finally, that
$$
\iors(v)^{[4]}(\om,A,\tit,p)=\lgl v, \theta(A,\tit,p)\otimes\om\rgl
$$
is a smooth function on
$\ahat_{LF}\times\bhat_{LF}\times\cTsrM_{F}\times M$, due to the
continuity (hence boundedness, hence smoothness) of $v$.

To see the smoothness of $\theta$, we write, using (\ref{spreading}),
$$
\theta(A,\tit,p)=\evbar(\evsr(A,\proS\tit),p).
$$
By Lemmata \ref{fstern} and \ref{evsr}  we obtain
continuity (hence boundedness resp.\ smoothness) 
of $(A,\tit)$ $\mapsto$
$(A,\proS\tit)\mapsto\evsr(A,\proS\tit)$ with respect to the (LF)-
resp.\ (F)-topologies, while Corollary \ref{corevbar} yields smoothness of
$(\evsr(A,\proS\tit),p)$ $\mapsto$
$\evbar(\evsr(A,\proS\tit),p)=\theta(A,\tit,p)$ with respect to the
(C)-topologies on the section spaces. (C) being weaker than (F), we can combine
both smoothness statements to obtain the smoothness of $\theta$ with respect to
the (C)-topology on the target space $\cTsrM$. Finally,
Corollary \ref{sectionKMeqF} permits us to replace the (C)-topology on $\cTsrM$ by
the (F)-topology.
\ep
\medskip
\begin{remark}\label{titfree}
Based upon a modification of the map $\theta$ employed above, it is possible to
arrive at a definition of the embedding $\iors:\DprsM\to\ehrsM$ not
explicitly containing $\tit\in\cTsrM$: Using results of \cite{mgrepdistr}, 
every distribution $v\in\DprsM$ can be
represented as a bounded resp.\ continuous $\CinfM$-linear map $v^\vee$ from
$\OmncM$ into $\cTsrM'$, the topological dual of the (F)-space $\cTsrM$. The
relation between $v$ and $v^\vee$ is given by $v(\tit\otimes\om)=\lgl
v^\vee(\om),\tit\rgl$, for $\om\in\OmncM$ and $\tit\in\cTsrM$. Introducing the
spreading operator
$\spr:\bhat\times M\to\Lin^b(\cTsrM,\cTsrM)$ by
$$
\spr(A,p)(\tit):=\theta(A,\tit,p)
$$
we obtain $\iors(v)(\om,p,A)=v^\vee(\om)\circ\spr(A,p)$ (where $\iors(v)$ is to
be understood as $\iors(v)^{[5]}$). Smoothness of $\spr$ follows from 
smoothness of $\theta$ (cf.\ the proof of Proposition \ref{iorsvsmooth}) via
the exponential law in \cite[27.17]{KM}.
\end{remark}

\section{Dynamics}\label{dynamics}
We now turn to the analytic core of our approach: the quotient construction of
tensor algebras of generalized functions displaying maximal 
compatibility properties with respect to smooth and distributional tensor fields.
As in the scalar case (\cite{vim}, see Section \ref{scalar}) our approach is based on singling out
subspaces of
moderate resp.\ negligible maps in the basic space $\ehrsM$ by requiring
asymptotic estimates
of the derivatives of representatives when evaluated along smoothing kernels
(Definition \ref{kernels}). We thereby adhere to the basic strategy (\cite[Ch.\ 9]{found}) 
of separating the basic definitions (the kinematics, in our current
terminology) from the testing (of the asymptotic estimates underlying
the quotient construction of the space of generalized tensor fields, or, for short,
the dynamics).
Since the representatives of generalized tensors depend not only on
points $p\in M$ and $n$-forms $\omega$ as in the scalar case \cite{vim} but also
on transport operators $A$, a new feature of the following construction is that
derivatives with respect to $A$ will have to be taken into account as well.

The notion of the core of a transport operator plays an important role in our
construction (in particular, in Lemma \ref{00dlemma} below and its applications).
In what follows, for any $U\subseteq M$, $U^\circ$ denotes the interior of
$U$.
\begin{definition}\label{coredef}
For any transport operator $A\in \bhat$ we define the core of $A$ by
$$
\core(A):=\{p\in M \mid A(p,p) = \mathrm{id}_{T_pM}\}^\circ.
$$
\end{definition}
\begin{remark}\label{corerem}
Given $K \comp M$ there always exists some $A\in \bhat$ with $K\subseteq \core(A)$.
Clearly such an $A$ can be obtained by gluing together local identity matrices.
For a more geometrical approach, choose any Riemannian metric $g$ on $M$ and 
denote by $r(p)$ the injectivity radius at
$p$ with respect to $g$. On the open neighborhood $W:= \{(p,q) \mid q \in B_{r(p)}(p)\}$ of the
diagonal in $M\times M$ we define a transport operator $A'$ by letting $A'(p,q)$ be parallel transport
along the unique radial geodesic in $B_{r(p)}(p)$ from $p$ to $q$. Now choose some $\chi \in
\mathcal{D}(W)$ with $\chi(p,p) = 1$ for all $p$ in a neighborhood of $K$. Then we may set $A:=\chi A'$.
\end{remark}

\begin{definition}\label{bhatumdef}
For $A\in \bhat$ we define the kernel of $A$ by
$$
\ker(A) := \{p\in M \mid A(p,p) = 0\}.
$$
If $U\subseteq M$ we set
$$\hat {\mathcal B}_U(M) := \{A\in \bhat \mid U \subseteq \ker(A)\}.
$$
\end{definition}
Note that for any $A\in \bhat$, $\LX A \in \hat {\mathcal B}_{\core(A)}(M)$ (which will be
used in Lemma \ref{t4t5lemma}). 
Based on these notions we are now ready to introduce the basic building
blocks of our construction:
\begin{definition} \label{moddef}
An element $u\in \ehrsM$ is called moderate if it satisfies the following
asymptotic estimates:\\
$\forall K \comp M$ $\forall A \in \bhat$ with $K\comp
\core(A)$\\
$\forall j\in \N$ $\forall B_1,\dots, B_j \in \hat {\mathcal B}_{\core(A)}(M)$\\
$\forall l\in \N_0$ $\exists N\in \N_0$ $\forall X_1,\dots,X_l \in
\mathfrak{X}(M)$
$\forall \Phi \in \atil$:
$$
\sup_{p\in K} \|\Lie_{X_1}\dots \Lie_{X_l} (\mathrm{d}_3^j
u(\Phi_{\eps,p},p,A)(B_1,\dots,B_j))\|_h
= O(\eps^{-N}) \qquad(\eps \to 0).
$$
The space of moderate tensor fields is denoted by
$\ehrsmM$.
\end{definition}
In this definition, $\|\,\|_h$ denotes the norm induced on the fibers of $\TrsM$
by any
Riemannian metric (changing $h$ does not affect the asymptotic estimates).
In the case $r=s=0$, $\|\,\|_h$ is to be replaced by the absolute value in
$\R$.
\begin{definition}\label{negdef}
An element $u\in \ehrsmM$ is called negligible if\\
$\forall K \comp M$ $\forall A \in \bhat$ with $K\comp
\core(A)$\\
$\forall j\in \N$ $\forall B_1,\dots, B_j \in  \hat {\mathcal
B}_{\core(A)}(M)$\\
$\forall l\in \N_0$ $\forall m\in \N_0$ $\exists k\in \N_0$ $\forall
X_1,\dots,X_l \in \mathfrak{X}(M)$
$\forall \Phi \in \tilde{\mathcal{A}}_k(M)$:
$$
\sup_{p\in K} \|\Lie_{X_1}\dots \Lie_{X_l} (\mathrm{d}_3^j
u(\Phi_{\eps,p},p,A)(B_1,\dots,B_j))\|_h
= O(\eps^{m}) \qquad(\eps \to 0).
$$
The space of negligible tensor fields is denoted by $\nhrsM$.
\end{definition}
The Lie derivatives in the asymptotic
estimates in Definitions \ref{moddef} and \ref{negdef} are to be understood as
$\Lie_{X_1}\dots \Lie_{X_l}$ acting on the smooth section
$p\mapsto\diff_3^ju(\Phi_{\eps,p},p,A)(B_1,\dots,B_j)$ of the
vector bundle $\TrsM$. The fact that the $B_i$ in Definitions \ref{moddef}
and \ref{negdef} are supposed to belong to $\hat {\mathcal B}_{\core(A)}(M)$
signifies their role as ``tangent vectors'' when differentiating with respect
to $A$.

Our first aim is to explore the relation between the ``scalar'' spaces $\ehzzmM$,
$\nhzzM$ and their counterparts $\ehmM$ and $\nhM$ from Definition \ref{modnegscalar}. 
The following basic
lemma introduces a reduction principle that will be referred to repeatedly in what follows.
\begin{lemma} (Reduction) \label{00dlemma}
Let $u\in \ehzzM$. Then for each $j\in \N_0$ and each 
$(A,B_1,$ $\dots,$ $B_j)
\in \hat {\mathcal B}(M)^{j+1}$, the map
$$
(\omega,p) \mapsto \mathrm{d}^j_3 u(\omega,p,A)(B_1,\dots,B_j)
$$
is a member of $\ehM$. It is also an element of $\eh(\core(A))$,
when restricted accordingly.\\
Moreover, $u\in \ehzzmM$ if and only if for all $j\in \N_0$, for all $A\in \bhat$
and all $B_1,\dots,B_j \in \hat {\mathcal B}_{\core(A)}(M)$ we have
$$
(\omega,p) \mapsto \mathrm{d}^j_3 u(\omega,p,A)(B_1,\dots,B_j) \in
\ehm(\core(A)).
$$
Analogous statements hold for $\nhzz(M)$ and $\nhM$.
\end{lemma}
\begin{proof}
This is immediate by inspecting Definitions \ref{moddef}, \ref{negdef} above and the
corresponding Definitions \ref{modnegscalar}\,(i) and (ii).
\end{proof}
As a first consequence we retain the important fact that negligibility for
elements of $\ehzzmM$ can be characterized without resorting to derivatives with
respect to slots 1 and 2:
\begin{corollary} \label{12300}
Let $u\in \ehzzmM$. Then $u\in \nhzzM$ if and only if \\
$\forall K \comp M$ $\forall A \in \bhat$ with $K\comp
\core(A)$\\
$\forall j\in \N$ $\forall B_1,\dots, B_j \in \hat {\mathcal B}_{\core(A)}$\\
$\forall m\in \N_0$ $\exists k\in \N_0$ $\forall \Phi \in
\tilde{\mathcal{A}}_k(M)$:
$$
\sup_{p\in K} \|\mathrm{d}_3^j u(\Phi_{\eps,p},p,A)(B_1,\dots,B_j)\|_h
= O(\eps^{m}) \qquad(\eps \to 0).
$$
\end{corollary}
\begin{proof}
This follows from Lemma \ref{00dlemma} and \cite[Cor.\ 4.5]{vim}.
\end{proof}

Next, in order to exploit these relations also for general $r$ and $s$, we
introduce
a ``saturation principle'' which characterizes moderateness and negligibility of
general tensor fields in terms of scalar fields obtained by saturating
$(r,s)$-tensor fields
with dual (smooth) $(s,r)$-fields. 
\begin{proposition}\label{satprop} (Saturation)
Let $u\in \ehrsM$. The following are equivalent:
\begin{itemize}
\item[(i)] $u\in \ehrsmM$.
\item[(ii)] For all $\tit \in {\mathcal T}^s_r(M)$, $u\cdot\tit\in \ehzzmM$.
\end{itemize}
An analogous statement holds for $\nhrsM$ and $\nhzzM$.
\end{proposition}
\begin{proof} It will suffice to prove the equivalence in the moderateness
case.

(i)$\Rightarrow$(ii): Given $u\in \ehrsmM$ and $\tit \in {\mathcal
T}^s_r(M)$,
let $K\comp \core(A)$, $A \in \bhat$, $l,\, j \in \N_0$, and
$B_1,\dots,B_j \in \hat {\mathcal B}_{\core(A)}(M)$.
Let $X_1,\dots,X_l \in \mathfrak{X}(M)$, $\Phi\in \atil$.
Then
$$
\Lie_{X_1}\dots \Lie_{X_l} \mathrm{d}_3^j [(u\cdot\tit)
(\Phi_{\eps,p},p,A)](B_1,\dots,B_j)
$$
is a sum of terms of the form
$$
\Lie_{X_{i_1}}\dots \Lie_{X_{i_k}} \mathrm{d}_3^j 
u(\Phi_{\eps,p},p,A)(B_1,\dots,B_j)
\cdot
\Lie_{X_{i_{k+1}}}\dots \Lie_{X_{i_l}} \tit (p).
$$
Here (on $K$) the first factor is bounded by some $\eps^{-N}$, and the 
second
is bounded independently
of $\eps$.

(ii)$\Rightarrow$(i):  By induction over $l$, we
deduce the following from (ii): For given $\tit\in\cTsrM$, $A\in\bhat$ such that
$K\comp\core (A)$ and $B_1,\dots,B_j\in\hat\cb_{\core(A)}(M)$, as well as
$l\in \N_0$ there exists some $N$ (depending on $\tit$) such that for 
all $X_1,\dots,X_l \in \XM$ 
\begin{equation}\label{liedertit}
\sup_{p\in K}\left|\big(\Lie_{X_1}\dots \Lie_{X_l} \mathrm{d}_3^j
u(\Phi_{\eps,p},p,A)(B_1,\dots,B_j)\big)\cdot\tit(p)\right| = O(\eps^{-N}).
\end{equation}
Indeed, for $l=0$ the assertion is immediate from (ii) and induction
proceeds by the Leibniz rule. We now 
note that in order to establish
moderateness of $u$, we may additionally suppose in the above that $K$ 
is contained
in some chart neighborhood $U$. Choose some $\chi \in \cd(U)$ which 
equals $1$ in a
neighborhood of $K$ and let $\{\tit_i \mid 1\le i \le n^{r+s}\}$ be a 
local basis of ${\mathcal T}^s_r(U)$. Then
inserting each $\chi \tit_i$ into (\ref{liedertit}) and choosing the 
maximum of the resulting
powers $\eps^{-N_i}$, we obtain the desired moderateness estimate
for $u$ on $K$.
\end{proof}
The above saturation principle allows one to extend the validity of Corollary
\ref{12300} to
general tensor fields:
\begin{theorem} \label{123}
Let $u\in \ehrsmM$. Then $u\in \nhrsM$ if and only if \\
$\forall K \comp M$ $\forall A \in \bhat$ with $K\comp
\core(A)$\\
$\forall j\in \N$ $\forall B_1,\dots, B_j \in \hat {\mathcal B}_{\core(A)}(M)$\\
$\forall m\in \N_0$ $\exists k\in \N_0$ $\forall \Phi \in
\tilde{\mathcal{A}}_k(M)$:
$$
\sup_{p\in K} \|\mathrm{d}_3^j u(\Phi_{\eps,p},p,A)(B_1,\dots,B_j)\|_h
= O(\eps^{m}) \qquad(\eps \to 0).
$$
\end{theorem}
\begin{proof}
Suppose that $u\in \ehrsmM$ satisfies the above condition.
By Proposition \ref{satprop}, for all $\tit\in {\mathcal T}^s_r(M)$,
$u\cdot\tit$ is a member of $\ehzzm(M)$
and satisfies the negligibility estimates
of order zero specified in Corollary \ref{12300}. It is therefore in $\nhzz(M)$
and the claim follows again from Proposition \ref{satprop}.
\end{proof}
The following result gives a local characterization of moderateness and
negligibility.
\begin{proposition} \label{wwloc}
Let $u\in \ehrsM$ and let $\mathcal U$ be an open cover of $M$. The following
are
equivalent:
\begin{itemize}
\item[(i)] $u\in \ehrsmM$.
\item[(ii)] $\forall U\in \mathcal{U}$ $\forall K\comp U$ $\forall A \in
\bhat$
with $K\subseteq \core(A)$ $\forall j,\, l$ $\forall B_1,\dots,B_j \in \hat
{\mathcal B}_{\core(A)}(M)$
$\exists N$ $\forall X_1,\dots,X_l\in \mathfrak{X}(U)$
$\forall \Phi\in$
$\tilde{\mathcal A}_0(U)$:
$$
\sup_{p\in K} \|\Lie_{X_1}\dots \Lie_{X_l} \mathrm{d}_3^j
u(\Phi_{\eps,p},p,A)(B_1,\dots,B_j)\|_h
= O(\eps^{-N}) \qquad(\eps \to 0)\,.
$$
\end{itemize}
An analogous result holds for the estimates defining $\nhrsM$.
\end{proposition}
\begin{proof} Again it will suffice to carry out the proof in
the moderateness
case.

(i)$\Rightarrow$(ii): Using suitable cut-off functions we may extend the $X_i$ to
global vector fields
$\tilde X_i$ on $M$ which coincide with $X_i$ on a neighborhood of $K$ ($1\le i
\le l$).
Moreover, given $\Phi\in \tilde{\mathcal A}_0(U)$, we pick any $\Psi\in
\tilde{\mathcal A}_0(M)$ and
$\chi \in \mathcal{D}(U)$ with $\chi=1$ in a neighborhood of $K$, and set
$$
\tilde \Phi(\eps,p) := \chi(p)\Phi(\eps,p) + (1-\chi(p))\Psi(\eps,p).
$$
Then $\tilde \Phi \in \tilde{\mathcal A}_0(M)$ and $\tilde \Phi(\eps,p) =
\Phi(\eps,p)$ for all $p$
in a neighborhood of $K$ and all $\eps\in (0,1]$. The moderateness estimate of
$u$ with respect to
$K$, $\tilde X_1,\dots,\tilde X_l$ and $\tilde\Phi$ then establishes (ii).

(ii)$\Rightarrow$(i): Let $K\comp M$, $A\in \hat{\mathcal{B}}(M)$, $K\subseteq
\core(A)$, $j$, $l$, and
$B_1,\dots,B_j \in \hat{\mathcal B}_{\core(A)}(M)$ be given. Without loss of
generality we may suppose that $K\comp U$ for some $U\in\mathcal{U}$.
For this set of data we obtain $N$ from (ii). Taking $\tilde X_1,\dots,\tilde
X_l \in \mathfrak{X}(M)$,
we set $X_i:=\tilde X_i|_U$ for $1\le i\le l$. Let $\tilde \Phi \in
\tilde{\mathcal A}_0(M)$. Our aim
is to construct $\Phi \in \tilde{\mathcal A}_0(U)$ such that
$\Phi(\eps,p) =
\tilde \Phi(\eps,p)$ for $p$ in a neighborhood of $K$ and $\eps$ sufficiently
small. Thus let $W$ be
a relatively compact neighborhood of $K$ in $U$ and choose $\chi\in
\mathcal{D}(M)$ with $\supp\chi\subseteq U$ and $\chi=1$ on $W$.
Since $\tilde \Phi$ is a smoothing kernel (Definition
\ref{kernels}) there exist
$C,\, \eps_0 >0$ such that
for all $p\in \supp \chi$ and all $\eps\le \eps_0$ we have $\supp \tilde
\Phi(\eps,p)\subseteq B_{\eps C}(p)
\subseteq U$. Choose $\lambda \in \cc^\infty(\R,I)$ such
that $\lambda = 1$ on $(-\infty,\eps_0/3]$
and $\lambda = 0$ on $[\eps_0/2,\infty)$. Finally, pick any $\Phi_1 \in \tilde
{\mathcal A}_0(U)$ and set
\begin{eqnarray*}
\Phi: I\times U &\to& \hat {\mathcal A}_0(M) \\
\Phi(\eps,p) &:=& (1-\chi(p)\lambda(\eps))\Phi_1(\eps,p) +
\chi(p)\lambda(\eps)\tilde \Phi(\eps,p).
\end{eqnarray*}
It is then easily checked that in fact $\Phi \in \tilde {\mathcal A}_0(U)$ and
that for $p\in W$
and $\eps \le \eps_0/3$ we have $\Phi(\eps,p) = \tilde \Phi(\eps,p)$. Thus the
moderateness test
(ii) with data $K,\, A,\, j,\, l,\, B_1,\dots,B_j,\, X_1,\dots,X_l$ and $\Phi$
gives the desired
$\ehrsmM$-estimate for the same data set, yet with
$\tilde X_1,\dots,\tilde X_l$, $\tilde\Phi$
replacing $X_1,\dots,X_l$, $\Phi$.
\end{proof}

\begin{remark}
Note that in the previous result, the transport operators employed in the local
tests on the open sets $U$ are supposed to be {\em global} operators, defined on
all of $M$. Nevertheless, if $\mathcal{U}$ is directed by inclusion as in Remark
\ref{sheafremark}\,(iii), then 
$$
u \in \ehrsmM \Leftrightarrow u|_U \in (\ehrs)_m(U) \ \forall U\in \mathcal {U},
$$
and analogously for $\nhrs$.
\end{remark}
We are now in a position to establish the main properties of the embeddings
$\iota^r_s$ and
$\sigma^r_s$ (cf.\ Def.\ \ref{sirs}). The following result corresponds to {\bf
(T1)} in the
general scheme of construction introduced in \cite[Ch.\ 3]{found}.
\begin{theorem} \label{T1} \
\begin{itemize}
\item[(i)] $\iors(\DprsM) \subseteq \ehrsm(M)$.
\item[(ii)] $\sirs(\mathcal{T}^r_s(M)) \subseteq \ehrsm(M)$.
\item[(iii)] $(\iors-\sirs)(\mathcal{T}^r_s(M)) \subseteq \nhrs(M)$.
\item[(iv)] If $v\in \DprsM$ and $\iors(v)\in \nhrs(M)$, then $v=0$.
\end{itemize}
\end{theorem}
\begin{proof}
(i) Let $v\in \DprsM$.
By saturation (Proposition \ref{satprop}) it suffices to show that for each
$\tit\in \mathcal{T}^s_r(M)$, $\iors(v)\cdot\tit\in
\ehzzm(M)$. In order to verify the moderateness estimates from Definition
\ref{moddef},
we first consider the case $j=0$ and $l=1$. For this, we have
\begin{align*}
 \Lin&_X(\iors(v)(\Phi(\eps,p),p,A)\cdot \tit(p))\\
 &=\LX
\langle v, \Asr(p,\,.\,){\,}\tit(p) \otimes \Phi(\eps,p)(\,.\,)\rangle \\
&= \langle v, \LXz(\Asr(p,\,.\,){\,}\tit(p)) \otimes
\Phi(\eps,p)(\,.\,)\rangle
+ \langle v, \Asr(p,\,.\,){\,}\tit(p) \otimes
\LXz\Phi(\eps,p)(\,.\,)\rangle.
\end{align*}
Here we are viewing $\Phi(\eps,p)(.)$ as a smooth function of $(p,.)$
on $M\times M$ (legitimized by Lemma \ref{kriegl}), enabling us to apply $\LXz$
similar to the case of $\Asr(p,\,.\,){\,}\tit(p)$ (cf.\ the proof of Proposition
\ref{lieembedding}). The transition from $\LX\circ v$ to
$v\circ\LXz$ in the preceding calculation has been argued in detail in the
proof of Proposition \ref{lieembedding}. Note that $\Lie'_X$ and $\LX$
introduced in equations (7) resp.\ (8) of \cite{vim} correspond to $\LXz$ resp.\
$\LzX$ in the present setting.

For $K\comp M$ given, let $K_1$ be some compact neighborhood of $K$.
Since $v$ is a continuous linear form on $\mathcal{T}^s_r(M)\otimes_\CinfM 
\Omega^n_\mathrm{c}(M)$, the seminorm estimate for $v$ on $K_1$ 
shows that there exist smooth vector fields
$Y^{(i)}_j$ ($i=1,\dots,k$; $j=1,\dots,m_i$)
and some $C>0$ such that for any $w\in
\mathcal{T}^s_r(M)\otimes
\Omega^n_\mathrm{c}(M)$
with $\supp(w) \subseteq K_1$ we have
$$
|\langle v,w\rangle| \le C
\max\limits_{i=1,\dots,k}\|\Lie_{Y_1^{(i)}}\dots\Lie_{Y_{m_i}^{(i)}} w\|_\infty.
$$
For obtaining the moderateness estimates, it suffices to consider a single term
$\|\Lie_{Y_1}\dots\Lie_{Y_{m}} w\|_\infty$.
If $\eps$ is sufficiently small, the arguments of $v$ in the expression above for
$\Lie_X(\iors(v)(\Phi(\eps,p),p,A)\cdot\tit(p))$ both have support
in $K_1$ with respect to $(.)$.
Furthermore, we obtain from the defining properties
of smoothing kernels (cf.\ Definition \ref{kernels}) that the supremum over $p\in K$
of the first
term
is of order $\eps^{-n-m}$. For the second, rewriting 
$\LXz = \Lie_X'$ as $\LXX-\LzX$ (corresponding to $(\Lie_X
+ \Lie_X') - \Lie_X$ in \cite{vim}), we obtain
an estimate of order $\eps^{-n-m-1}$. Higher order Lie derivatives
$\Lie_{X_1}\dots\Lie_{X_l}$ can clearly be treated in the
same way, yielding estimates of order $\eps^{-n-m-l}$,
as can derivatives with respect to $A$: As a formal calculation shows,
the latter do not influence the order of $\eps$ since only the boundedness
of $A$ and $B_1,\dots,B_j$ on $K_1$ is used. For interchanging the action of $v$ with 
directional derivatives $\diff_B$ in direction $B$ with respect to $A$, 
we note that for $p$ fixed, the map $\phi:
A\mapsto(\Asr(p,.)\,\tit(p))\otimes\om(.)$ is smooth from $\bhat$ into
$\cTrsM\otimes_\CinfM\OmncM$ with respect to the respective (LF)-topologies, as
is, by definition, the linear map $v$. Hence by \cite[3.18]{KM},
$\diff_B\lgl v,\phi(A)\rgl=\lgl v,\diff_B\phi(A)\rgl$, and
the claim follows.

(ii) Since for $t\in \mathcal{T}^r_s(M)$, $\sirs(t)(\omega,p,A) = t(p)$ it is
immediate
that the $\ehrsm$-estimates hold for $\sirs(t)$ on any compact set, with $N=0$.

(iii) Let  $t\in \mathcal{T}^r_s(M)$, $K\comp M$, $A\in
\hat{\mathcal{B}}(M)$,
$K\subseteq
\core(A)$, and $\Phi \in \tilde {\mathcal A}_0(M)$. Then for any $\tit \in
\mathcal{T}^s_r(M)$
and any $p\in K$
\begin{equation*}
[(\sirs - \iors)(t)(\Phi(\eps,p),A)\cdot\tit](p) =
(t\cdot\tit)(p) - \int_M t(q) \Asr(p,q){\,}\tit(p) \Phi(\eps,p)(q) \,dq.
\end{equation*}
By Proposition \ref{satprop} and Corollary \ref{12300} it suffices to show the
negligibility estimates for this difference
and its derivatives with respect to $A$. To this end we
introduce the notation $f_p(q) := t(q)\Asr(p,q)\tit(p)$. Then
$f\in \cc^\infty(M\times M)$ and the above expression reads
\begin{equation} \label{family}
f_p(p) - \int_M f_p(q) \Phi(\eps,p)(q)\,dq
= \int_M (f_p(p) - f_p(q)) \Phi(\eps,p)(q)\,dq.
\end{equation}
Lemma \ref{lemmagtozero} now yields that (\ref{family}) vanishes
 of order $\eps^{m+1}$, uniformly for $p\in K$, for
$\Phi\in\tilde\ca_m(M)$. Next, we consider derivatives of 
$(\sirs - \iors)(t)$ with respect to $A$. 
Since $\sirs(t)$ does not depend on $A$ we have to show that
all $A$-derivatives of $\iors(t)$ of order greater or equal one are negligible.
To fix ideas we first consider the special case $r=1$, $s=0$. Then for
$A,\, B\in \bhat$
and $\tit \in \mathcal{T}^0_1(M)$
$$
({\mathrm d}_3\iota^1_0(t))(\Phi(\eps,p),p,A)(B)\cdot \tit(p) =  \int_M t(q)
\cdot B^1_0(p,q){\,}\tit(p)
\ \Phi(\eps,p)(q)
\,dq.
$$

Now for $B\in \hat {\mathcal B}_{\core(A)}(M)$ and $p\in K\comp
\core(A)$, $B(p,p)=0$,
so again Lemma \ref{lemmagtozero} gives the desired estimate. For general
values of $r$ and $s$, since $A\mapsto A^s_r$ is the composition of a
multilinear map with the diagonal map,  we obtain a sum of terms
each of which has the form
$$
\int_M f_p(q) \Phi(\eps,p)(q)\,dq
$$
with $f$ smooth and $f_p(p)=0$ for all $p$, so the claim follows by a third
appeal to Lemma \ref{lemmagtozero}.

(iv) The (rather lengthy) direct proof would proceed along the
lines of
proof of Proposition \ref{assocprop} (the latter actually making a stronger
statement
than (iv) does). To minimize redundancy, we confine ourselves to noting that
(iv) follows from Proposition \ref{assocprop}, via Proposition \ref{satprop} and
Corollary \ref{12300} (with $m=1$).
\end{proof}
Our next aim is to establish stability of $\ehrsmM$ and $\nhrsM$ under the Lie
derivatives $\lh_X$ from Definition
\ref{lh}. Again we first consider the case $r=s=0$:
\begin{lemma} \label{t4t5lemma}
Let $X\in {\mathfrak X}(M)$ and $u\in \ehzzmM$ resp.\ $u\in \nhzzM$. Then also
$\lh_X u\in \ehzzmM$ resp.\ $\lh_X u\in \nhzzM$.
\end{lemma}
\begin{proof}
The proof will be achieved by reduction (Lemma \ref{00dlemma}) to the
setting of \cite{vim}. Recall from Definition \ref{lh} that
$$
 (\lh_Xu)(\om,p,A)=\Lie_X(u(\om,A))(p)-\mathrm{d}_1u(\om,p,A)(\Lie_X\om)-\mathrm{d}
_3u(\om,p,A)(\Lie_XA).
$$
Since $\Lie_X A \in \hat {\mathcal B}_{\core(A)}(M)$ for any $A\in \hat {\mathcal B}(M)$, 
the moderateness (resp.\ negligibility) estimates for $\LhzzX u$,
i.e., for ${\mathrm d}_3u(\om,p,A)(\Lie_XA)$
follow directly from the definitions. In order to show moderateness resp.\
negligibility of $\LhzXz\,u + \LhXzz\,u$, i.e.\ of $u_0:= (\omega,p,A) \mapsto
\Lie_X(u(\om,A))(p)-\mathrm{d}_1u(\om,p,A)(\Lie_X\om)$ we
employ
Lemma \ref{00dlemma}. Fix $j\in \N_0$, $A\in \hat {\mathcal B}(M)$
and $B_1,\dots,B_j \in \hat {\mathcal B}_{\core(A)}(M)$. Then
\begin{equation}
\begin{array}{l}
\mathrm{d}^j_3 u_0(\omega,p,A)(B_1,\dots,B_j) =\\[5pt]
 \Lie_X[\mathrm{d}^j_2 u(\omega,A)(B_1,\dots,B_j)](p)
- \mathrm{d}_1[\mathrm{d}^j_3 u(\omega,p,A)(B_1,\dots,B_j)](\Lie_X \omega).\label{d3j}
\end{array}
\end{equation}
In fact, we may interchange  $\diff_3$ with $\diff_1$ due to symmetry of higher
differentials (\cite[5.11]{KM}). Concerning $\diff_3$ and $\LX$
we can either use Lemma \ref{difflx} or note that 
due to $r=s=0$ the term $\Lie_X(u(\om,A))(p)$ can be written as
$\diff_2u(\om,p,A)(X)$ which again permits to resort to symmetry of higher
differentials. (\ref{d3j}) is precisely the Lie derivative in the sense 
of \cite[Def.\ 3.8]{vim} of the map $(\omega,p) \mapsto \mathrm{d}^j_3
u(\omega,p,A)(B_1,\dots,B_j)$,
hence is in $\ehm(\core(A))$ (resp.\ $\nh(\core(A))$) by \cite[Th.\ 4.6]{vim}.
Again by Lemma \ref{00dlemma}, the claim follows.
\end{proof}
\begin{theorem} \label{t4t5}
$\ehrsmM$ and $\nhrsM$ are stable under Lie derivatives $\LhX$ where $X\in\XM$.
\end{theorem}
\begin{proof} It will suffice to treat the case of moderateness. Thus let
$u\in \ehrsmM$ and $X\in\XM$.
Picking any $\tit$ from $\cTsrM$, saturation (Proposition
\ref{satprop}) yields $u\cdot\tit\in\ehzzmM$. By Lemma \ref{t4t5lemma}, also
$\LhX(u\cdot\tit)\in\ehzzmM$. However,
$\LhX(u\cdot\tit)=(\LhX u)\cdot\tit+u\cdot(\LhX\tit)$ due to Proposition
\ref{leibnizbasic}. The second term being a member of
$\ehzzmM$, again by Proposition \ref{satprop}, we infer that
$(\LhX u)\cdot\tit$ is moderate. Since $\tit$ was arbitrary, a third appeal to
Proposition \ref{satprop} establishes the moderateness of $\LhX u$.
\end{proof}
Thus we finally arrive at
\begin{definition}
The space of generalized $(r,s)$-tensor fields is defined as
\begin{equation*} 
\ghrsM := \ehrsmM / \nhrsM.
\end{equation*}
\end{definition}
$\ghrsM$ is both a $\cc^\infty(M)$- and a $\ghzzM$-module.
For $u\in \ehrsmM$ we denote by $[u]$ its equivalence class in $\ghrsM$.
From Theorem \ref{T1} it follows that $\iors$ and $\sirs$ induce maps from $\Dprs(M)$
resp.\ $\cTrsM$ into $\ghrsM$. These maps will be denoted by the same letters.
We collect the main properties of $\ghrsM$ in the following result.
\begin{theorem} \label{maintheorem}
The map
$$
\iors: \Dprs(M) \to \ghrsM
$$
is a linear embedding whose restriction to $\cTrsM$ coincides with
$$
\sirs: {\mathcal T}^r_s(M) \to \ghrsM\,.
$$
For any smooth vector field $X$ on $M$, the Lie
derivative
\begin{eqnarray*}
\lh_X: \ghrsM &\to& \ghrsM \\
\lh_X([u]) &:=& [\lh_X u]
\end{eqnarray*}
is a well-defined operation commuting with the embedding, i.e., for any $v\in
\Dprs(M)$, $\iors(\Lie_X v) = \lh_X\iors(v)$.
\end{theorem}
\begin{proof}
All the claimed properties of $\iors$ and $\sirs$ follow from Theorem \ref{T1}.
$\lh_X$ is well-defined
by Theorem  \ref{t4t5}. Finally, that Lie derivatives commute with the embedding was
already established in Proposition \ref{lieembedding}.
\end{proof}
Summing up, we obtain the following commutative diagram:\medskip\\
$$
  \xymatrix{
      \cTrsM \ar[rrr]^{\Lie_X} \ar[rd]^{\rho^r_s} \ar[dd]_\sirs &    &   &
\cTrsM \ar[ld]_{\rho^r_s} \ar[dd]^\sirs \\
       & \DprsM \ar[r]^{\Lie_X} \ar[ld]_\iors & \DprsM \ar[dr]^\iors &  \\
       \ghrsM \ar[rrr]^{\hat \Lie_X} & & &  \ghrsM
  }
$$
\vskip1.3em
As was highlighted in Section \ref{nogo}, the properties included in 
this diagram are optimal in light of Schwartz' impossibility result.

To extend these results to the universal tensor algebra over $M$, we first note
that
\begin{eqnarray*}
\ehrsmM \otimes (\hat{\mathcal{E}}^{r'}_{s'})_\mathrm{m}(M) &\subseteq&
(\hat{\mathcal{E}}^{r+r'}_{s+s'})_\mathrm{m}(M)  \\
\ehrsmM \otimes \hat{\mathcal{N}}^{r'}_{s'}(M) &\subseteq&
\hat{\mathcal{N}}^{r+r'}_{s+s'}(M).
\end{eqnarray*}
Thus we obtain the algebra $\mathcal{T}_{\ehm}(M) := \bigoplus_{r,s} \ehrsmM$
containing the ideal
$\mathcal{T}_{\nh}(M) := \bigoplus_{r,s} \nhrsM$. 
\begin{definition}
The universal algebra of generalized tensor fields is defined as
$$
\mathcal{T}_{\gh}(M) := \mathcal{T}_{\ehm}(M)/\mathcal{T}_{\nh}(M) \cong
\bigoplus_{r,s} \ehrsmM / \nhrsM = \bigoplus_{r,s} \ghrsM.
$$
\end{definition}
The operations of tensor product, contraction and Lie derivative with respect to smooth
vector fields naturally extend
to $\mathcal{T}_{\gh}(M)$ and we have, by Proposition
\ref{leibnizbasic},
$$
\LhX(u_1\otimes u_2) =
(\LhX u_1)\otimes u_2 + u_1\otimes (\LhX u_2).
$$

Furthermore, the embeddings $\iota^r_s$ and $\sigma^r_s$ extend to
$\mathcal{T}_{\Dp}(M) := \bigoplus_{r,s} \DprsM$ resp.\
$\mathcal{T}(M):=\bigoplus_{r,s}\cTrsM$. We will denote the
respective maps by $\iota$ resp.\ $\sigma$. From Theorem \ref{maintheorem} we obtain:
\begin{corollary}
The mapping
$$
\iota: \mathcal{T}_{\Dp}(M) \to  \mathcal{T}_{\gh}(M)
$$
is a linear embedding whose restriction to $\mathcal{T}(M)$ coincides with the algebra homomorphism
$$
\sigma: \mathcal{T}(M) \to \mathcal{T}_{\gh}(M),
$$
thereby rendering $\mathcal{T}(M)$ a subalgebra of $\mathcal{T}_{\gh}(M)$.
For any smooth vector field $X$ on $M$, the Lie derivatives
$\lh_X: \mathcal{T}_{\gh}(M)\to \mathcal{T}_{\gh}(M)$
resp.\ $\LX:\ct(M)\to\ct(M)$ intertwine
with the embedding $\iota$.
\end{corollary}
To conclude this section, we give the following characterization of 
$\ghrsM$ as a $\CinfM$-module.
\begin{theorem}\label{modulestructure}
The following chain of $\CinfM$-module isomorphisms holds:
$$
\ghrsM \cong \ghzzM \otimes_{\CinfM} \cTrsM \cong \Lin_{\CinfM}(\cTsrM,\ghzzM).
$$
\end{theorem}

\begin{proof}
The $\CinfM$-module $\cTrsM$ is projective and finitely generated (cf.\
\cite[2.23]{GHV}, applied to each connected (hence second countable) component
of $M$).
Thus by \cite[Ch.\ II, \S 4, 2]{bourbaki-algebra}, it follows that
\begin{eqnarray*}
\Lin_{\CinfM}(\cTsrM,\ghzzM) &\cong& \ghzzM \otimes_{\CinfM}
\Lin_{\CinfM}(\cTsrM,\CinfM)\\
&=& \ghzzM \otimes_{\CinfM} \cTrsM.
\end{eqnarray*}
We establish the theorem by showing
$\ghrsM\cong\Lin_{\CinfM}(\cTsrM,\ghzzM)$.
By the exponential law in \cite[27.17]{KM}, it is immediate from Lemma
\ref{reprbasic} (4) (or (3)) that
\begin{eqnarray*}
\ehrsM &\cong& \Lin^b_{\CinfM}(\cTsrM,\Cinf(\ahat\times\bhat\,,\,\CinfM))\\
            &=& \Lin^b_{\CinfM}(\cTsrM\,,\,\ehzzM).
\end{eqnarray*}
holds. Here, the boundedness assumption in the last term can be formally
dropped, i.e., $ \Lin^b_{\CinfM}$ can safely be replaced by $ \Lin_{\CinfM}$:
Since all the spaces involved are convenient, a $\CinfM$-linear map
$F:\cTsrM\to\ehzzM$ is bounded if and only if for all $\om\in\ahat$,
$A\in\bhat$, the maps $F_{\om,A}:\tit\mapsto F(\tit)(\om,A)$ (for
$\tit\in\cTsrM$) are bounded, due to the uniform boundedness principle 
\cite[5.26]{KM}. Being a member of
$\Lin_{\CinfM}(\cTsrM\,,\,\CinfM)$, however, the map $F_{\om,A}$ is of the
form
$\tit\mapsto t\cdot\tit$ for some $t\in\cTrsM$ and thus even continuous with
respect to the Fr\'echet topologies. Using saturation (Proposition \ref{satprop}),
it is straightforward to check that
$\ehrsM \cong\Lin_{\CinfM}(\cTsrM\,,\,\ehzzM)$
induces an isomorphism from
$\ghrsM$ onto $\Lin_{\CinfM}(\cTsrM\,,\,\ghzzM)$, thereby finishing the proof.
\end{proof}

\section{Association}\label{association}

In all versions of Colombeau's construction the Schwartz impossibility result
is circumvented by introducing a very narrow concept of equality, more precisely, by
introducing a very strict equivalence relation on the space of moderate elements. In particular,
this equivalence is finer than distributional equality. Nevertheless, raising the latter to the 
level of the algebra by introducing an equivalence relation called association
one can take advantage of using both notions of equality in the so-called ``coupled calculus''. 
For example, tensor products of continuous or $\cc^k$-fields are not preserved in the 
algebra $\ct_{\gh}$ in the sense that the embedding is not a homomorphism with
respect to the tensor product. It will, however, turn out to 
to be a homomorphism at the level of association.

In many situations of practical relevance, elements of the algebra are 
associated to distributions. This
feature has the advantage that often one may use the mathematical
power of the differential algebra to perform the calculations but then invoke
the notion of association to give a physical interpretation to the result obtained.
This is especially useful when it comes to modelling source terms in
nonlinear partial differential equations and, consequently, one often wants to
consider such equations in the sense of association rather than equality
(cf., e.g., \cite{c3, MOBook}). One of the applications we have in mind is
Einstein's equations where we seek generalized metrics which have an
Einstein tensor associated to a distributional energy-momentum tensor
representing, e.g., a cosmic string or a shell of matter (cf.\
\cite{clarke,vickersESI,SV06} and the references therein).

In this section we introduce an appropriate concept of association 
for generalized tensor fields.

Note that as an exception to our standard notation, in this section
we will use capitals for generalized scalar and tensor fields. This will
permit us to distinguish notationally between elements $u_1,\, u_2$ etc.\ of 
$\ehrsmM$ and their respective classes $U_1=[u_1]$, $U_2=[u_2]$ etc.\ in
$\ghrsM$. 
We start by briefly considering the scalar case (touched upon in \cite{vim}).
\begin{definition}
We say that a generalized scalar field $F=[f] \in \ghM$ is associated with $0$
(denoted $F \approx 0$), if for some (and hence any)
representative $f \in \ehmM$ of $F$ and for each $\omega \in
\Omega^n_\mathrm{c}(M)$ there exists some $m>0$ such that $\forall \Phi \in \amtil$
\begin{equation*}
\lim_{\eps \to 0}\int_{M}f(\Phi(\eps,p),p)\omega(p)=0 \,.
\end{equation*}
We say that two generalized functions $F,G$ are associated and write
$F\approx G$ if $F-G \approx 0$.
\end{definition}
At the level of association we regain the usual results for multiplication
of distributions (\cite[Prop.\ 6.2]{vim}).

\begin{proposition}
\item{(i)} If $f \in \cc^{\infty}(M)$ and $v \in \DpM$
then
\begin{equation*}
\iota(f)\iota(v) \approx \iota(fv).  
\end{equation*}
\item{(ii)} If $f, g \in \cc(M)$ then
\begin{equation*}
\iota(f)\iota(g) \approx \iota(fg). 
\end{equation*}
\end{proposition}

It is also useful to introduce the concept of associated distribution or
``distributional shadow'' of a generalized function.
\begin{definition}
We say that $F \in \ghM$ admits $v \in \DpM$ as an
associated distribution if $F \approx \iota(v)$.
\end{definition}
The notion of associated distribution can be expressed through the following
concept of convergence.
\begin{definition}
We say that $F \in \ghM$ converges weakly to  $v \in \DpM$, and write
$F \stackrel{\DpM}{\rightarrow} v$ if for some (hence any)
representative $f \in \ehmM$ of $F$ and for each $\omega \in
\Omega^n_\mathrm{c}(M)$  there exists some $m>0$ so that $\forall \Phi \in
\amtil$
\begin{equation*}
\lim_{\eps \to 0}\int_{M}f(\Phi(\eps,p), p)\omega(p)=\langle
v, \omega \rangle.  
\end{equation*}
\end{definition}
In fact the following proposition states that weak convergence to $v$
is equivalent to having $v$ as an associated distribution.
\begin{proposition}\label{scalarassocprop}
An element $F=[f]$ of $\ghM$ possesses $v \in \DpM$ as an associated
distribution if and only if $F \stackrel{\DpM}{\rightarrow} v$.
\end{proposition}
The proof of Proposition \ref{scalarassocprop} is a slimmed-down version of
that of Proposition \ref{assocprop}, compare Corollary \ref{weakcorol}.
Note that not all  generalized functions have a distributional
shadow. However, if $v \in \DpM$ and $\iota(v) \approx 0$ then $v=0$, so that
provided the distributional shadow exists it is unique.

We now extend this circle of ideas to the tensor case. We start by
defining association for generalized tensor fields.
\begin{definition}
A generalized tensor field $U=[u]\in \ghrsM$ is called associated with
$0$, $U\approx 0$, if for
one (hence any) representative $u$, we have:
\begin{eqnarray*}
& \forall \omega \in \Omega^n_\mathrm{c}(M)\ \forall A\in \bhat\ \forall \tilde
t\in \mathcal{T}^s_r(M) \ \exists m>0
\ \forall \Phi\in \amtil: & \\[5pt]
& {\displaystyle \lim_{\eps\to 0}} \int u(\Phi(\eps,p),p,A)\,\tilde t(p)\,
\omega(p) = 0. &
\end{eqnarray*}
$U_1$, $U_2\in \ghrsM$ are called associated, $U_1\approx U_2$, if $U_1-U_2\approx 0$.
\end{definition}
\begin{definition}
A generalized tensor field $U\in \ghrsM$ is said to admit $v\in \DprsM$
as an associated distribution and $v$ is called the distributional shadow
of $U$, if $U \approx \iors(v)$.
\end{definition}
Employing the localization techniques from the proof of Proposition \ref{wwloc}
we obtain:
\begin{lemma}\label{assrem}
The following statements are equivalent for any $U=[u]\in \ghrsM$:
\begin{itemize}
\item[(i)] $U\approx 0$ in $\ghrsM$.
\item[(ii)] $\forall \tilde t\in \mathcal{T}^s_r(M)$: $U\cdot \tilde t
\approx 0$ in $\ghzzM$.
\item[(iii)] $\forall W\subseteq M$ open: $\forall \omega \in
\Omega^n_\mathrm{c}(W)$
$\forall A\in \bhat$ $\forall \tilde t\in \mathcal{T}^s_r(W)$ $\exists m>0:
\ \forall \Phi\in {\tilde{\mathcal{A}}}_m(W):$
$$
{\displaystyle \lim_{\eps\to 0}} \int u(\Phi(\eps,p),p,A)\,\tilde t(p)\,
\omega(p) = 0.
$$
\end{itemize}
\end{lemma}
\begin{definition} \label{weak}
We say that a generalized tensor field $U \in \ghrsM$ converges weakly to
$v \in \DprsM$, and write $U \stackrel{\DpM}{\rightarrow} v$ if for some
(hence any) representative $u \in \ehrsmM$ of $U$ we have
\begin{eqnarray*}
& \forall \omega \in \Omega^n_\mathrm{c}(M)\ \forall A\in \bhat\ \forall \tilde
t\in \mathcal{T}^s_r(M) \ \exists m>0
\ \forall \Phi\in \amtil: & \\
& {\displaystyle \lim_{\eps\to 0}} \int u(\Phi(\eps,p),p,A)\,\tilde t(p)\,
\omega(p) = \langle v,\tit\otimes\omega\rangle. &
\end{eqnarray*}
\end{definition}

In the proof of the
following result we will make use of a refined version of (the technical core
of) Theorem 16.5 of \cite{found}: For $W$ an open subset of $\R^n$ let
$c:D(\subseteq I\times W)\to\R$ denote a smooth function in the sense of (5) and
(6) of \cite{vim} (with the range space ${\mathcal A}_0(\R^n)$ resp.\
$\cd(\R^n)$ replaced by $\R$). 
Moreover, let $K,L$ be compact subsets of $W$ such that
$K\comp L\comp
W$ and $(0,\eps_0)\times L$ is contained in the interior of $D$. Finally, let
$q>0$, $\eta>0$. If $\sup_{x\in L}|c(\eps,x)|=O(\eps^q)$
then it follows that
$\sup_{x\in K}|\pa^\bet c(\eps,x)|=O(\eps^{q-\eta})$
for every $\bet\in\N_0^n$ . This can be
established along the lines of the proof of Theorem 16.5 of \cite{found}.
\begin{proposition} \label{assocprop}
Let $v\in \DprsM$. Then
\begin{equation}\label{assocproplimit}
\begin{array}{c}
 \forall \omega \in \Omega^n_\mathrm{c}(M)\ \forall A\in \bhat\ \forall \tilde
t\in \mathcal{T}^s_r(M)
\ \forall \Phi\in \atil:  \\[5pt]
 {\displaystyle \lim_{\eps\to 0}} \int \iors(v)(\Phi(\eps,p),p,A)\,\tilde
t(p)\, \omega(p) =
\lgl v,\tit\otimes\om\rgl. 
\end{array}
\end{equation}
In particular, $\iors(v)$ satisfies the conditions of Definition \ref{weak} with $m=0$, so
$$
\iors(v) \stackrel{\Dp}{\rightarrow} v.
$$
\end{proposition}
\begin{proof}
Let $\omega$, $A$, $\tit$ and
$\Phi$ be given. Since both sides of (\ref{assocproplimit}) are
linear in $\om$ we may assume
that $\supp\om\comp W$ where $(W,\psi)$ is a chart on $M$.
Let us fix compact subsets $L',L$ of $W$ with $\supp\om \comp L' \comp L \comp
W$. We may suppose without loss of generality that the images of $W,L',L$ under
$\psi$ are balls in $\R^n$.
By the defining
properties of a smoothing kernel there exists $\eps_0>0$ such that for $\eps\leq
\eps_0$ and $p\in L'$ we have
$\supp\Phi(\eps,p)\comp L$. Thus we may further assume without loss of
generality that also $\supp v\comp W$.
Passing to coordinates we may therefore suppose that
$W$ is an (open) ball in $\R^n$,
with $\tit$ and $A$ defined on $W$
resp.\ $W\times W$ and $\om$, $v$
compactly supported in $W$.
Finally, by Lemma 4.2 of
\cite{vim}, the place of $\Phi$ is taken by $\eps^{-n}\phi(\eps,x)
(\frac{y - x}{\eps})$ where $\phi\in\cc^\infty_{b,w}(I\times
W,\mathcal{A}_0(\R^n))$ is defined by
$$
\phi(\eps,x)(y)\,d^ny := \eps^n((\psi^{-1})^*\Phi(\eps,\psi^{-1}(x)))(\eps
y + x).
$$
Writing $x,y$ for $p,q$ and $\vphi\, d^nx$ for $\om$, we obtain that for 
$\eps\le \eps_0$, the expression
$\int \iors(v)(\Phi(\eps,p),p,A)\tilde t(p) \omega(p) -
\lgl v,\tit\otimes\om \rgl$ locally takes the form
\begin{align*}
 \int \Big\lgl &v(y), \Asr(x,y)\tit(x)\eps^{-n}\phi(\eps,x)
 \left(\textstyle{\frac{y - x}{\eps}}\right)\Big\rgl \vphi(x)\, d^nx
 - \lgl v(y), \tit(y)\vphi(y) \rgl \\
&= \Big\lgl v(y), \int \big(\Asr(y-\eps z,y)\tit(y-\eps
z)\vphi(y-\eps z)\phi(\eps,y-\eps z)(z)\\
&\hphantom{mmmmmmmmmmmmmmm} -
\Asr(y,y)\tit(y)\vphi(y)\phi(\eps,y)(z)\big)\, d^nz\Big\rgl.
\end{align*}
Since $\phi: (0,\eps_0)\times (L')^\circ \to \ca_0(\R^n)$ is smooth with 
$\supp \phi(\eps,x)(\frac{.-x}{\eps}) \subseteq L$ for all $\eps,x$
and the evaluation map
$\ev:\ca_0(\R^n)\times\R^n\to\R$ is smooth by \cite[3.13\,(i)]{KM}, the expression
$J(x,y):=\Asr(x,y)\tit(x)\eps^{-n}\phi(\eps,x)
 (\frac{y - x}{\eps}) \vphi(x)$  represents a member of
$\cd(\R^n\times \R^n)^{n^{r+s}}$ (for $\eps<\eps_0$) with $\supp J\subseteq\supp
\vphi\times L$. Therefore, the combined action of
$v$ and integration with respect to $x$ can be viewed as the action of the
distribution $\mathds{1}(x)\otimes v(y)$ on $J$, allowing to interchange $v$
with the integral.
Since  $\phi\in\ca^\triangle_{m,w}(W)$ by Lemma 4.2(A) of \cite{vim}, we have
$\sup_{\xi\in L}|c(\eps,\xi)|=O(\eps^{m+1-|\al|})$ for
$c(\eps,x):=\int_{\R^n}\phi(\eps,x)(z)z^\al\, \diff^n z$ and $1\leq|\al|\leq
m$. From the analogue of Theorem 16.5 of \cite{found} discussed above we
infer, for every $\bet\in\N_0^n$,
$$\sup_{\xi\in L'}|\pa^\bet c(\eps,\xi)|
=\sup_{\xi\in L'}\left|\int_{\R^n}\pa^\bet\phi(\eps,\xi)(z)z^\al\, \diff^n
z\right|
=O(\eps^{m+1-|\al|-\eta}).$$
Now, applying Taylor expansion of order $m$ to every tensor component of
\begin{align*}
\psi_\eps(z,y):=\Asr(y-\eps z,y)\,\tit(y-\eps z)  \vphi(y-&\eps z) 
\phi(\eps,y-\eps z)(z)\\
&- \Asr(y,y)\tit(y)\vphi(y)\phi(\eps,y)(z)
\end{align*}
and integrating with respect to $z$ we obtain estimates of order
$\eps^{m+1-\eta}$ for the terms of the Taylor polynomials and of order
$\eps^{m+1}$ for the respective remainder terms, uniformly for $y\in L'$.
For $m=0$, the
Taylor expansions consist of the remainder terms solely, allowing overall
estimates even by $\eps^{m+1}$. Moreover, $\psi_\eps(z,y)$ vanishes for $y\notin
L'$. On the basis of analogous
asymptotics for $\int\pa_y^\bet\psi_\eps(z,y)\,\diff^n z$
it follows that $\eps^{-(m+1-\eta)}\cdot\int\psi_\eps(z,y)\,\diff^n z$ is
bounded
in $\cd(\R^n)^{n^{r+s}}$.  Altogether, we obtain $\lgl
v(.),\int\psi_\eps(z,.)\,\diff^n z\rgl$ being of order
$\eps^{m+1-\eta}$ resp.\ $\eps^{m+1}$ (for $m=0$), thereby establishing our
claim.
\end{proof}
\begin{corollary} \label{weakcorol}
An element $U$ of $\ghrsM$ possesses $v\in \DprsM$ as an associated
distribution if and only if
$U\stackrel{\Dp}{\rightarrow} v$.
\end{corollary}
It follows from Corollary \ref{weakcorol} and Proposition
\ref{assocprop} that the distributional shadow of a generalized tensor field
is unique (if it exists).

For continuous tensor fields we obtain stronger (locally uniform)
convergence properties:
\begin{proposition}\label{contassocprop}
Let $t$ be a continuous $(r,s)$-tensor field on $M$. Then
\begin{eqnarray*}
& \forall K\comp M\ \forall A\in \bhat\ \forall \tilde t\in
\mathcal{T}^s_r(M) \ \forall \Phi\in \atil: & \\[5pt]
& {\displaystyle \lim_{\eps\to 0}} \displaystyle\sup_{p\in K}
\left|\iors(t)(\Phi(\eps,p),p,A)\tilde t(p)
- t(p)\cdot\tit(p)\right|
= 0. &
\end{eqnarray*}
\end{proposition}
As a consequence of this result and the proof of Proposition \ref{assocprop} we
obtain compatibility
of the embedding $\iors$ with the standard products on $\mathcal{C}\times
\mathcal{C}$ and
$\Cinf \times \Dp$ in the sense of association:

\begin{corollary}\label{assprod} \
\begin{itemize}
\item[(i)] Let $t\in \mathcal{T}^{r_1}_{s_1}$, $v\in
{\mathcal{D}'}^{r_2}_{s_2}$.
Then $\iota^{r_1}_{s_1}(t)\otimes \iota^{r_2}_{s_2}(v) \approx
\iota^{r_1+s_1}_{r_2+s_2}(t\otimes v)$.
\item[(ii)] Let $t_1,\ t_2$
be continuous tensor fields of order
$(r_1,s_1)$ resp.\ $(r_2,s_2)$. Then
$\iota^{r_1}_{s_1}(t_1)\otimes \iota^{r_2}_{s_2}(t_2) \approx
\iota^{r_1+s_1}_{r_2+s_2}(t_1\otimes t_2)$.
\end{itemize}
\end{corollary}
\begin{remark}
Guided by the notion of convergence given in Proposition \ref{contassocprop},
we may introduce the concept
of $\mathcal{C}^0$-association: $U=[u]\in \ghrsM$ is called
$\mathcal{C}^0$-associated with $0$,
$U\approx_0 0$, if for one (hence any) representative $u$, we have:
\begin{eqnarray*}
& \forall K\comp M\ \forall A\in \bhat\ \forall \tilde t\in
\mathcal{T}^s_r(M) \
\exists m>0
\ \forall \Phi\in \amtil: & \\[5pt]
& {\displaystyle \lim_{\eps\to 0}} \displaystyle\sup_{p\in K}
\left|u(\Phi(\eps,p),p,A)\tilde t(p)\right| = 0. &
\end{eqnarray*}
$U_1$, $U_2\in \ghrsM$ are called $\mathcal{C}^0$-associated, $U_1\approx_0
U_2$, if $U_1-U_2\approx_0 0$.
Moreover, for $t$ a continuous $(r,s)$-tensor field we write $U\approx_0
t$ if $U\approx_0
\iors(t)$.

Analogously we may introduce the concept of $\mathcal{C}^k$-association
by considering
$\mathcal{C}^k$-convergence instead of $\mathcal{C}^0$-convergence in
the above definition
($k\in \N_0$).
With this notion, Corollary \ref{assprod}\,(ii) can be strengthened: if
$t_1,t_2$ are
$\mathcal{C}^k$-tensor fields then $\iota^{r_1}_{s_1}(t_1)\otimes
\iota^{r_2}_{s_2}(t_2)
\approx_k \iota^{r_1+s_1}_{r_2+s_2}(t_1\otimes t_2)$. In fact, for any
$\mathcal{C}^k$-$(r,s)$-tensor field $t$, $\iors(t)\approx_k t$.
\end{remark}

\begin{appendix}
\appsection{Transport operators and two-point tensors}\label{appendixA}

In this appendix we collect the main definitions, notations and properties of
transport operators resp.\ two-point tensors.

Let $M$, $N$ be smooth paracompact Hausdorff manifolds of (finite) dimensions $n$
and $m$, respectively.
We consider the vector bundle
$$
\TOMN :=\LTMTN := \bigcup_{(p,q)\in M\times N} \{(p,q)\} \times
\mathrm{L}(\TpM,\TqN)
$$
of transport operators on $M\times N$.
For charts $(U,\vphi=(x^1,\dots,x^n))$, $(V,\psi=(y^1,\dots,y^m))$ of $M$
resp.\ $N$,
a typical vector bundle chart (or vb-chart, for short) of $\TOMN$ is given by
\begin{eqnarray*}
\Psi_{\vphi\psi}: \bigcup_{(p,q)\in U\times V} \{(p,q)\} \times
\mathrm{L}(\TpM,\TqN)
&\to& \vphi(U)\times \psi(V)\times \R^{mn} \\
((p,q),A) &\mapsto& ((\vphi(p),\psi(q)), (dy^i(A\partial_{x_j}))_{i,j})\,.
\end{eqnarray*}
Setting $\varphi_{21}:=\varphi_2\circ\varphi_1^{-1}$ and analogously for
$\psi$, the
transition functions for $\TOMN$ are given by
$$
\Psi_{\vphi_2\psi_2} \circ \Psi_{\vphi_1\psi_1}^{-1}((x,y),a) =
((\vphi_{21}(x),\psi_{21}(y)), D\psi_{21}(y)\cdot a \cdot
D\vphi_{21}(x)^{-1})\,.
$$
Elements of $\Gamma(\TOMN)$\label{ro:transportoperators},
i.e., smooth sections of $\TOMN$, are also called transport operators.

Alternatively, transport operators can be viewed as sections of a suitable
pullback bundle. To fix notations (following \cite{dieudonne3}), let
$E\stackrel{\pi}{\to}B$ be a vector bundle, $B'$ a manifold and $f:
B'\to B$ a smooth map. We denote
by $E'=f^*(E)$ the pullback bundle of $E$ under $f$.  
The total space of $f^*(E)$ is the closed submanifold
$B' \times_B E := \left\{ (b', e) \in B' \times E\ |\ f(b') = \pi(e)
\right\}$
of $B'\times E$, and the projection is $\pi' = {\mathrm pr}_1|_{B'
\times_B E}$,
i.e.,
we have the following diagram, where $f' := \prt|_{B' \times_B E}$:
\[\begin{CD}
E' @>{f'}>> E\\
@V{\pi'}VV @VV{\pi}V\\
B' @>>f> B
\end{CD}
\]

If $(\pi^{-1}(U),\Psi)$ is a vb-chart for $E$ with
$\Psi = (e_b \mapsto (\Psi^{(1)}(b), \Psi^{(2)}(e_b)))$ and if $(V,\vphi)$ is
a chart in
$B'$ such that $V \cap f^{-1}(U)\neq \emptyset$, then
$$
\begin{aligned}
\varphi \times_B \Psi: {\pi'}^{-1}(V \cap f^{-1}(U)) &\rightarrow \varphi(V)
\times \mathbb R^N\\
(b', e) &\mapsto (\varphi(b'), \Psi^{(2)}(e))
\end{aligned}
$$
is a typical vb-chart for $f^*(E)$. To
obtain the explicit
form of
the transition functions of $f^*(E)$, let $\varphi_{21} := \varphi_2 \circ
\varphi_1^{-1}$ be
a change of charts in $B'$ and
$\Psi_2 \circ \Psi_1^{-1} (x, \xi) = (\Psi_{21}^{(1)}(x), \Psi_{21}^{(2)}(x)
\cdot \xi)$ a change
of vb-charts in $E$. Then the corresponding change of vb-charts in
$f^*(E)$ is
given by
$$
(\varphi_2 \times_B \Psi_2) \circ (\varphi_1 \times_B \Psi_1)^{-1}(x',
\eta) =
(\varphi_{21}(x'), \Psi_{21}^{(2)}(\Psi_1^{(1)} \circ f \circ
\varphi_1^{-1}(x')) \cdot \eta).
$$
For the particular case where $f$ is of the form $\prt:M\times N\to N$ for
manifolds $M$, $N$ and a vector bundle $E\stackrel{\pi}{\to}N$, there is a
simplified way of representing $\prtS E$ as a vector bundle over $M\times
N$ (a similar statement being true for $f=\pro$ and a vector bundle
$E\stackrel{\pi}{\to}M$) which we will use freely wherever convenient:
Namely,
$\prtS E$ can be realized in this case as the (full) product
manifold $M\times E$ (rather than as $(M\times N)\times_N E$), with
projection $\id_M\times \pi:(p,v)\mapsto (p,\pi(v))$.

Now for $M$, $N$ as above and $\pro$ and $\prt$ the projection maps of
$M\times N$ onto $M$ resp.\ $N$,
we apply the above constructions to obtain the bundle
$$
\TPMN := (\poTSM \otimes \ptTN,M\times N,\pi_\otimes)
$$
of {\em two-point tensors} on $M\times N$.

Sections of $\TPMN$
will also be called two-point tensors. Note that
any element of $\Gamma(\TPMN)$\label{ro:two-point tensors} is a finite sum
of sections of the form
\begin{equation} \label{generictpt}
(p,q)\mapsto f(p,q)\,\eta(p)\otimes \xi(q)
\end{equation}
where $\eta \in \Om^1(M)$, $\xi\in {\mathfrak X}(N)$ and $f\in \Cinf(M\times
N)$.
We will call two-point tensors of this form {\em generic}.

We will use the following notation for the obvious connection between
transport operators and two-point tensors:
For $V$, $W$ finite dimensional vector spaces we have the canonical
isomorphism
\begin{eqnarray*}
\bullet: V^*\otimes W &\to& \mathrm{L}(V,W) \\
(\alpha \otimes w)_\bullet(v) &:=& \langle \alpha,v\rangle w
\end{eqnarray*}
which induces a strong vb-isomorphism (in the sense of \cite[ch.\ II, \S 1]{GHV})
\begin{equation} \label{blob}
\bullet: \TPMN \to \TOMN\,.
\end{equation}
For a two-point tensor $\Upsilon \in \Gamma(\TPMN)$ we denote by
$\Upsilon_\bullet$
the corresponding transport operator in $\Gamma(\TOMN)$. 
Viewing transport operators as two-point tensors (as was done in \cite{sotonTF})
has some advantages when doing explicit calculations (cf., e.g., the formula
for the Lie derivative (\ref{tplie}) below). Moreover, we note that
$\TPMN$ is canonically isomorphic to the first jet bundle $J^1(M,N)$
(cf.\ \cite[12.9]{michorbook}). $\TPMN$ also appears as the
particular case $\TSM\boxtimes\TN$ of the so-called {\it external tensor
product} $E\boxtimes F$ of vector bundles $E\to M$, $F\to N$ in 
\cite[Ch.\ II, Problem 4]{GHV}.

Given diffeomorphisms $\mu: M_1 \to M_2$ and $\nu: N_1 \to N_2$, we have
a natural pullback action
$$(\mu,\nu)^*: \Ga(\TO(M_2,N_2))
\to \Ga(\TO(M_1,N_1))\,,
$$
given (for a transport operator $A\in \Gamma(\TO(M_2,N_2))$) by
\begin{equation}\label{topullback}
((\mu,\nu)^*A)(p,q) = (\mathrm{T}_q\nu)^{-1} \circ A(\mu(p),\nu(q)) \circ
\mathrm{T}_p\mu\,.
\end{equation}
Similarly, we obtain a natural pullback action
$$
(\mu,\nu)^*: \Gamma(\TP(M_2,N_2)) \to
\Ga(\TP(M_1,N_1)),
$$
defined on generic two-point tensors by
\begin{align}
((\mu,\nu)^*(f \eta\otimes \xi))(p,q) :&=
f(\mu(p),\nu(q))(\mathrm{T}_p\mu)\adj(\eta(\mu(p)))
\otimes (\mathrm{T}_q\nu)^{-1}(\xi(\nu(q))) \nonumber\\
&= ((\mu,\nu)^*f)(p,q) \mu^*\eta(p) \otimes \nu^*\xi(q).
\label{tppullback}
\end{align}
In case $M_1=M_2$, $N_1=N_2$ and $\mu=\nu$ we simply write $\mu^*$ instead of
$(\mu,\mu)^*$. Note, however, that this special case is not sufficient for the
purpose of this article, cf.\ the respective remark preceding Definition
\ref{def:mu}.

As can easily be checked, these pullback actions commute with the
isomorphism (\ref{blob}): for any $\Upsilon\in \Ga(\TP(M_2,N_2))$ we have
\begin{equation} \label{bulletcommute}
(\mu, \nu)^*(\Upsilon_\bullet) = ((\mu, \nu)^*\Upsilon)_\bullet\,.
\end{equation}

\label{ro:muxnu1}
Although $\TOMN$ is a vector bundle over $M\times N$, it is important to
note that an arbitrary diffeomorphism $\rho:M_1\times N_1\to M_2\times N_2$ does not,
in general,
induce a natural pullback action $\rho^*:\Gamma(\TO(M_2,N_2)) \to
\Ga(\TO(M_1,N_1))$ which reduces to (\ref{topullback}) for the particular
case
$\rho=\mu\times\nu$ as above. 
For a counterexample, consider the flip operator $\fl:M\times N\to N\times M$.
An analogous statement holds for the case of
two-point tensors. This is the reason why we use the notation
$(\mu,\nu)^*$ for the above actions rather than $(\mu\times\nu)^*$,
which would give the wrong impression of being the composition of
$\mu\times \nu$
with the (non-existent) pullback operation for general diffeomorphisms
alluded
to above. This underlines that TP and TO have to be treated as genuine
bifunctors and
cannot be factorized via $(M,N)\mapsto M\times N$ composed with a
single-argument functor.

Before introducing the Lie derivative of transport operators let us recollect
some basic facts on the Lie derivative of smooth sections of a vector bundle
$E$ over $M$ with respect to a smooth vector field $X\in\XM$.

Following \cite[6.14--15]{michorbook}, we assume that a functor $F$ is
given, assigning a vector bundle $F(M)$ over $M$ to every manifold $M$ of
dimension $n$. Moreover, to every local diffeomorphism $\mu:M\to N$, the functor
$F$ assigns a vector bundle homomorphism $F(\mu):F(M)\to F(N)$ over $\mu$,
acting as a linear isomorphism on each fiber. Given an arbitrary smooth vector
field $X\in\XM$, we assume that the local action $(\tau,v)\mapsto F(\FlXt)v$ is
a smooth function of $(\tau,v)$, mapping some $(-\tau_0,+\tau_0)\times F(M)|_U$
into $F(M)$, where $\tau_0>0$, $U$ is an open subset of $M$ and $\FlXt:
(-\tau_0,+\tau_0)\times U\to M$. Then for $\eta\in\Ga(F(M))$, we define the
pullback of $\eta$ under $\FlXt$ locally by
$(\FlXt)^*\eta:=F(\FlXmt)\,\circ\,\eta\,\circ\,\FlXt$. $(\FlXt)^*\eta$ being
smooth on 
$(-\tau_0,+\tau_0)\times U$, we set $(\LX\eta)(p):=\ddtz((\FlXt)^*\eta)(p)$.
Smoothness in $\tau$ (for $p$ fixed) yields existence of $(\LX\eta)(p)$ while
smoothness in $(\tau,p)$ yields smoothness of the local section $\LX\eta$
of $F(M)|_U$. The family of all such local sections consistently defines
$\LX\eta\in\Ga(F(M))$. All the preceding applies, in particular, to
$F(M):=\TsrM$ and $F(M):=\bigwedge^n\TSM$.

In the following, we fix $M$ and write $E$ for $F(M)$.

\begin{remark}\label{lierough}
Under specific assumptions, we can say more about the action of flows resp.\
about Lie derivatives:
\begin{itemize}
\item[(i)]
If $X$ is complete, we can take $U=M$ in the above which renders $\FlXt$,
$F(\FlXt)$ and $(\FlXt)^*\eta$ (smooth and) defined globally on $\R\times M$
resp.\ $\R\times E$ resp.\ $\R\times\GaE$. In this case,
$\frac{1}{\tau}[(\FlXt)^*\eta - \eta]$ tends to $\LX\eta$ as
$\tau\to0$ in the linear space $\GaE$ with respect to the topology of pointwise
convergence on $M$.
\item[(ii)]
If $\supp\eta$ is compact there exists $\tau_0$ such that a local version of
$(\FlXt)^*\eta$ can be extended from $(-\tau_0,+\tau_0)\times
U$ to $(-\tau_0,+\tau_0)\times M$ by values $0$. In this case,
$\frac{1}{\tau}[(\FlXt)^*\eta - \eta]$ tends to $\LX\eta$ as
$\tau\to0$ in the linear space $\GacE$ with respect to the topology of pointwise
convergence on $M$.
\end{itemize}
\end{remark}
In the sequel, $\GaE$ and $\GacE$ will always be equipped with the (F)-
resp.\ the (LF)-topology. The following Proposition strengthens the statements
of the preceding remark, making essential use of calculus in convenient vector
spaces (cf.\ Appendix B).

\begin{proposition}\label{lierefinement}
Let $F$ be a functor as specified above and let $X\in \mathfrak{X}(M)$.
\begin{enumerate}
\item[(1)]
Assume $X$ to be complete. Then 
\begin{enumerate}
\item[(i)]
$(\tau,\eta)\mapsto(\FlXt)^*\eta$ is smooth from $\R\times\GaE$ into $\GaE$.
\item[(ii)]
$(\tau,\om)\mapsto(\FlXt)^*\om$ is smooth from $\R\times\GacE$ into $\GacE$.
\item[(iii)]
$\LX\eta=\lim\limits_{\tau\to0}\frac{1}{\tau}[(\FlXt)^*\eta - \eta]$ in $\GaE$,
for every $\eta\in\GaE$.
\end{enumerate}
\item[(2)]
Let $X$ be arbitrary, $K\comp M$ and $\tau_0>0$ such that
$(\FlXt)^*\om$ is defined for all $\om\in\Ga_{\mathrm{c},K}(E)$ and all
$|\tau|\leq\tau_0$. Then 
\begin{enumerate}
\item[(i)]
$(\tau,\om)\mapsto(\FlXt)^*\om$ is smooth from
$(-\tau_0,+\tau_0)\times\Ga_{\mathrm{c},K}(E)$ into $\GacE$.
\item[(ii)]
$\LX\om=\lim\limits_
{\atopmg{\tau\to 0}{|\tau|<\tau_0}}
\frac{1}{\tau}[(\FlXt)^*\om - \om]$ in $\GacE$, for every
$\om\in\Ga_{\mathrm{c},K}(E)$.
\end{enumerate}
\end{enumerate}

\end{proposition}
\begin{proof}
(1)\ $(\tau,\eta,p)\mapsto\eta(\FlXt p)=\ev(\eta,\FlXt p)$ is smooth, due to
smoothness of $\FlX$ and of $\ev:\GaE\times M\to E$ (the latter
follows from the definition of the (C)-topology on spaces of smooth sections,
cf.\ \cite[30.1]{KM}). By our assumptions on $F$, also $\phi:(\tau,\eta,p)\mapsto
F(\FlXmt)\eta(\FlXt p)=((\FlXt)^*\eta)(p)$ is smooth; $\phi$ can be viewed as a
section of the vector bundle $\mathrm{pr}_3^*\GaE$ over $\R\times\GaE\times M$.
Applying Corollary \ref{expCF} we obtain that
$\phi^\vee$ is a smooth map from $\R\times\GaE$ into $\GaE$, which is (i). (iii)
now follows immediately by fixing $\eta$. In order to establish (ii), replace
$\eta\in\GaE$ by $\om\in\GacE$ in the proof of (i), yielding smoothness of
$(\tau,\om)\mapsto(\FlXt)^*\om$ into $\GaE$. Since $\FlX$ maps each set of the
form $[-\tau_0,+\tau_0]\times K$ (with $K\comp M$) onto some compact subset $L$
of $M$ we see that $\phi^\vee([-\tau_0,+\tau_0]\times\Ga_{\mathrm{c},K}(E))$ is
contained in $\Ga_{\mathrm{c},L}(E)$.
A slight generalization of \cite[Th.\ 2.2.1]{book} establishes
smoothness of $\phi^\vee$ as a
map from $\R\times\GacE$ into $\GacE$. (An alternative argument completing the
proof of (ii) exploits linearity in $\om$ by passing to $\phi^{\vee\vee}:
\Ga_{\mathrm{c},K}(E)\to\Cinf((-\tau_0,+\tau_0),\GacE)$ via the exponential
law from \cite[27.17]{KM}).

(2)\ The proof of (i) is similar to the proof of (ii) of part (1): Just
replace $\R$ by
$(-\tau_0,+\tau_0)$ and $\GacE$ by  $\Ga_{\mathrm{c},K}(E)$ in the domain of the
respective maps (note that $\tau_0$ depends on $K$, forcing us to restrict
statements (i) and (ii) to the subspace $\Ga_{\mathrm{c},K}(E)$ of $\GacE$).
(ii) again
follows from (i).
\end{proof}

As the respective proofs show, ``local'' variants of (1)(i) and (1)(iii) of the
preceding result hold for arbitrary vector fields $X$, in the following
sense: Denoting by $U_\tau$ the (open) set $\{p\in M\mid
\FlXt(p)\text{ and }\FlXmt(p)\text{ are defined}\}$, the map
$(\tau,\eta)\mapsto(\FlXt)^*\eta$ is smooth from
$(-\tau_0,+\tau_0)\times\GaE$ into $\Ga(E|_{U_{\tau_0}})$ and
$\LX\eta$ exists as the respective limit in $\Ga(E|_{U_{\tau_0}})$, both for
$\tau_0$ small enough as to make $U_{\tau_0}$ nonempty. However,
there is no local variant of (1)(ii) as a map from, say,
$(-\tau_0,+\tau_0)\times\GacE$ into $\Ga_\mathrm{c}(E|_{U_{\tau_0}})$: For
$(\FlXt)^*\om$
to have compact support in $U_{\tau_0}$ for all $|\tau|<\tau_0$ we would have to
assume $\supp\om\subseteq\bigcap_{|\tau|<\tau_0}\FlXt(U_{\tau_0})$
which excludes $(-\tau_0,+\tau_0)\times\GacE$ as domain of the map envisaged
above.

We will express statements (1)(iii) and (2)(ii) of Proposition \ref{lierefinement}
by saying that $\LX\eta$ exists in the (F)-sense resp.\ that $\LX\om$ exists in
the (LF)-sense.

Turning now to the definition of the Lie derivative for transport
operators and
two-point tensors
we have to extend the setting of \cite{michorbook}, capable of
handling only  single argument functors $F$ as outlined above, to
bifunctors $G$ (such as TP and TO), assigning a vector bundle $G(M,N)$ over
$M\times N$ to each pair $M,N$ of manifolds and a vector bundle isomorphism
$G(\mu,\nu):G(M_1,N_1)\to G(M_2,N_2)$ over $\mu\times\nu$ to every pair of
local diffeomorphisms $\mu:M_1\to M_2$, $\nu:N_1\to N_2$.
{\it Mutatis mutandis}, all statements of Remark \ref{lierough} and Proposition
\ref{lierefinement} remain valid for bifunctors of that type.
Thus,
let $X\in \mathfrak{X}(M)$, $Y \in
\mathfrak{X}(N)$ be complete vector fields with flows $\FlX$ and $\Fl^Y$,
respectively.
We then define the Lie derivative by
differentiating at $0$ the pullback (under the flow of
$(X,Y)$) of any given transport operator $A\in \GaTOMN$:
\begin{equation}\label{tolie}
\LXY A (p,q) := \ddttz (\mathrm{Fl}^X_\tau , \mathrm{Fl}^Y_\tau)^*A
(p,q)\,.
\end{equation}
Analogously, for $\Upsilon = f \eta \otimes \xi$ a generic element of
$\GaTPMN$ we set
\begin{equation}\label{tplie}
\begin{array}{rcl}
\LXY\Upsilon(p,q) &=& \ddtz (\mathrm{Fl}^X_\tau ,
\mathrm{Fl}^Y_\tau)^*\Upsilon (p,q)  \\[5pt]
&=&  (\LXz f + \LzY f)(p,q)\, \eta(p)\otimes \xi(q)\\[5pt]
&& + f(p,q) (\LX\eta(p) \otimes \xi(q) + \eta(p)\otimes \LY\xi(q)).
\end{array}
\end{equation}

From (\ref{bulletcommute}), we obtain for any $\Upsilon \in \GaTPMN$:
\begin{equation}\label{bulletliecommute}
\LXY(\Upsbl) = (\LXY\Upsilon)_\bullet\,.
\end{equation}

\label{ro:muxnu2}
Resuming the discussion of the pullback action of {\it pairs} of
diffeomorphism started after (\ref{bulletcommute}) we see that also in order
to implement a geometric approach to Lie derivatives via pullback action
of flows, definitions (\ref{tolie}) and (\ref{tplie})
had to be based on pairs of flows $(\FlXt,\FlYt)$
on $M$ resp.\ $N$ rather than on the flow of some single vector field $Z$
on $M\times N$.
This is emphasized by our notation $\Lie_{X,Y}$
rather than $\Lie_{X\times Y}$, reflecting the fact that it is precisely the
vector
fields of the form $Z=(X,Y)$ from the subspace $\XM\oplus\XN$ of
$\mathfrak{X}(M\times N)$ that induce a pullback action and a Lie
derivative on
(sections of) the bundle functors TO and TP. In this sense, our concept of Lie
derivative is, in fact, a proper extension resp.\ refinement of the usual
setting as presented, e.g.,
in \cite{michorbook} where only single-argument (vector bundle valued)
functors are considered. 

\appsection{Auxiliary results from calculus in convenient
vector spaces}\label{appendixB}

The notion of a smooth curve $c:\R\to E$ where $E$ is some locally convex space
is unambiguous. The space $\Cinf(\R,E)$ of smooth curves in $E$ will
always carry the topology of uniform convergence on compact intervals
in all derivatives separately. For locally convex spaces $E$, $F$, a map $f:E\to
F$ is defined to be smooth if $c\mapsto f\circ c$ takes smooth curves in $E$ to
smooth curves in $F$. This notion of smoothness depends only on the
respective families of bounded sets, i.e., if the topologies of $E$
and $F$ are changed in such a way that in each space the family of
bounded subsets is preserved then the space $\Cinf(E,F)$ of smooth
mappings from $E$ to $F$ remains the
same (\cite[1.8]{KM}). As a rule, we endow $\Cinf(E,F)$ with the
``(C)-topology'' (``C'' standing for ``curve''
resp.\ ``$\Cinf$'' resp.\ ``convenient''), defined as the initial (locally
convex) topology with respect to the family of all mappings
$c^*:\Cinf(E,F)\to\Cinf(\R,F)$, for $c\in\Cinf(\R,E)$ (\cite[3.11]{KM}).
The evaluation map $\ev:\Cinf(E,F)\times E\to F$ sending $(f,x)$ to $f(x)$ is
smooth by \cite[3.13 (1)]{KM}. Consequently, $\ev_x:\Cinf(E,F)\ni
f\mapsto f(x)\in F$ is (linear and) smooth (equivalently, bounded, due to
\cite[2.11]{KM}) for every fixed $x\in E$.

In the above, $E$ can be replaced by some open (even $c^\infty$-open, cf.\
\cite[2.12]{KM}) subset $U$ of $E$ resp.\ by some smooth
and smoothly Hausdorff (cf.\ \cite[p.\ 265]{KM}) manifold $M$
modelled over convenient vector spaces (\cite[27.17]{KM}; as to the evaluation
map, see the proof of Lemma \ref{kriegl}).

There are many equivalent ways how to define a convenient vector space, cf.\
\cite[2.14, Th.]{KM}. We will use condition (6) of this theorem saying that a
locally convex space $E$ is convenient if for each bounded absolutely convex
closed
subset $B$, the normed space $(E_B,p_B$) is complete. Every sequentially
complete locally convex vector space is convenient (\cite[2.2]{KM}); all the
spaces considered in this article are sequentially complete.

Note that, in general, smooth maps are not necessarily continuous. If, however,
$E$ is metrizable then any $f\in\Cinf(E,F)$ is
continuous, due to \cite[4.11, 2.12, and p.\ 8]{KM}.

If $E$ and $F$ have the property that smooth maps $f:E\to F$ map compact sets
to bounded sets (in particular, if $E$ is metrizable or an (LF)-space), 
there is a second natural locally convex topology on $\Cinf(E,F)$
which will be called (D)-topology: This is the topology of convergence of
differentials (hence ``D'') of all orders $l$ (separately), uniformly on
sets of the form $K\times B^l$ where $K$ is a compact and $B$ a bounded
subset of $E$. By the chain rule (\cite[3.18]{KM}),
the (D)-topology is finer
than the (C)-topology. In fact, even on $\Cinf(\R^2,\R)$, an inspection of the
form of the respective typical neighborhoods of $0$ reveals (D) to be
strictly finer than (C). However, Theorem \ref{KMeqF} below shows that (D) and
(C) have
the same bounded sets if $E$ is an (F)-space or an (LF)-space. Therefore, the
notion of
smoothness on $\Cinf(E,F)$ with respect to both topologies is the same in
that case. On $\Cinf(\R,E)$, the (D)-topology coincides with
the usual (F)-topology.

For the proof of the theorem below, note the following: Calling a sequence
$x_n\to x$ in a topological vector space fast converging if for each
$l\in\N$, the sequence $n^l (x_n-x)$ tends to $0$, then every convergent
sequence in a metrizable topological vector space or an (LF)-space
possesses a fast converging
subsequence. To see this, let $(V_k)_k$ be a decreasing neighborhood base of
$0$ consisting of circled sets and choose $n_k\in \N$ monotonically increasing
such that $(x_{n_k}-x)\in k^{-k}V_k$.

\begin{theorem}\label{KMeqF}
Let $E$ be an (F)-space or an (LF)-space and $F$ be an arbitrary
locally convex space. Then
every (C)-bounded subset of $\Cinf(E,F)$ is (D)-bounded, hence (C) and (D) have
the same bounded sets.
\end{theorem}
\begin{proof}
Let $B$ be bounded with respect to the (C)-topology, i.e.,
\begin{eqnarray*}
\lefteqn{\forall c\in\Cinf(\R,E)\ \forall K_0\comp\R\ \forall l\in\N_0:}\\
&&\hphantom{mmmmmmmmmm}\{(f\circ c)^{(l)}(\tau)\mid\tau\in K_0,\ f\in B\}
\textrm{ is bounded in }F.
\end{eqnarray*}
Assume, by way of contradiction, $B$ not to be bounded with respect to (D), i.e.
\begin{eqnarray*}\label{notbD}
\lefteqn{\exists K\comp E\ \exists D\,\text{(bounded}\subseteq
E)\ \exists l\in\N_0:}\\
&&\hphantom{mm}\{\diff^l f(p)(w,\dots,w)\mid
f\in B,\ p\in K,\ w\in D\}
\textrm{ is unbounded in }F.
\end{eqnarray*}
Here, we have already used polarization as, e.g., in \cite[7.13 (1)]{KM}, to
obtain equal vector arguments $(w,\dots,w)$. Fix $K,D,l$ as above.
By the preceding, we can choose
sequences
$f_k\in B$, $p_k\in K$, $w_k\in D$ such that
\begin{eqnarray}\label{nontozero}
\frac{1}{k^{2lk}} \,\diff^l f_k(p_k)(w_k,\dots,w_k)\not\to 0\quad(k\to\infty).
\end{eqnarray}
By passing to suitable subsequences we can assume that there is $p\in K$ such
that $p_k\to p$ fast as $k\to\infty$.
(Note that if $E$ is an (LF)-space, we are working within one fixed
Fr\'echet subspace of $E$, due to $K$ and $D$ being bounded.)
Setting
$v_k:=
k^{-k}w_k$ we obtain $v_k\to0$ fast, due to $D$
being bounded. Now, by \cite[2.10]{KM}, there are a smoothly parametrized polygon
$c:\R\to E$ and numbers $\tau_k\to0$ in $\R$ such that $c(\tau_k+\tau)=p_k+\tau
v_k$ for $\tau\in(-\de_k,+\de_k)$, for suitable $\de_k\in (0,1)$ .

Recalling that $B$ is (C)-bounded, choose a compact interval
$K_0\comp\R$ containing all intervals $(\tau_k-\de_k,\tau_k+\de_k)$ and take $l$
and $c$
as above. Then we conclude that $\{(f\circ c)^{(l)}(\tau)\mid\tau\in K_0,\ f\in
B\}$ is bounded in $F$. Defining $y_k:= (f_k\circ c)^{(l)}(\tau_k)=
\diff^lf_k(p_k)(v_k,\dots,v_k)$, we therefore have $k^{-lk}y_k\to 0$.
Consequently,
\begin{eqnarray*}
\frac{1}{k^{2lk}} \,\diff^l f_k(p_k)(w_k,\dots,w_k)
   &=&  \frac{1}{k^{lk}}\,y_k\to 0,
\end{eqnarray*}
contradicting (\ref{nontozero}).
\end{proof}
Note that the preceding theorem remains valid for $\Cinf(U,F)$ where $U$ is an
open subset of the (F)- resp.\ (LF)-space $E$: The relevant part of the polygon
$c$ constructed in the proof can be assumed to be contained in some given
absolutely convex neighborhood of $p$.

In order to carry over Theorem \ref{KMeqF} to spaces $\GaME$ of smooth
sections
of vector bundles $E\stackrel{\pi}{\to}M$, we need a suitable notion of
(C)-topology on section spaces $\GaME$, $E$ having some convenient vector space
$Z$
as typical fiber. It certainly would be tempting to proceed as
follows (\cite[p.\ 294]{KM}):
For any local trivialization $(U,\Psi)$ on $E\stackrel{\pi}{\to}M$
($\Psi:\pi^{-1}U\to U\times Z$) and $u\in\GaME$, let
$\ti\Psi(u)\in\Cinf(U,Z)$ be given by
$$
\ti\Psi(u):=\prt\circ\Psi\circ(u|_U).
$$
Now we could define the (C)-topology on $\GaME$ as the initial topology with
respect
to the family of all maps $\{\ti\Psi_\al\mid\al\in A\}:\GaME\to\Cinf(U_\al,Z)$
where
$\{(U_\al,\Psi_\al)\mid\al\in A\}$ is any vector bundle atlas\footnote{The term
vector bundle atlas denotes a compatible family $\{(U_\al,\Psi_\al)\mid\al\in
A\}$ of local trivializations on $E\stackrel{\pi}{\to}M$
 with $\bigcup_\al
U_\al=M$.} for
$E\stackrel{\pi}{\to}M$.
Here, $\Cinf(U_\al,Z)$, in turn, carries the (C)-topology in the sense of
\cite[27.17]{KM}. This construction, however, {\it does} depend on the atlas
used: Already for the trivial bundle $\R\times Z$ with atlas
$\{\id_{\R\times Z}\}$, the (C)-topology on $\Ga(\R,\R\times Z)$ would
become strictly finer by adjoining the local trivialization $\id_\R\times \phi$
where $\phi$ is a discontinuous bornological isomorphism of $Z$. It is not hard
to see that the (C)-topology, as defined above, does not depend on the atlas if
$Z$ is barreled and bornological. Thus, we could either accept this additional
condition or define the (C)-topology via the maximal vector bundle atlas. For
the present purposes, however, only the family of (C)-bounded subsets is
relevant which, fortunately, is independent of the atlas:
It suffices to note that for $w\in\Cinf(\Ualb,Z)$, the ``change of vector bundle
chart'' $w\mapsto \ev\circ(\psi_{\alb}\times w)\circ\De$ (where
$\De:U_\al\cap U_\bet=:\Ualb\to \Ualb\times \Ualb$ denotes the
diagonal map and $\psi_{\alb}:\Ualb\to\mathrm{GL}(Z)$ the transition
functions) is linear and smooth by \cite[3.13 (1)(6)(7)]{KM}, hence bounded
(this substantiates and extends the respective remarks preceding the proposition in 
\cite[p.\ 294]{KM}).

\begin{corollary}\label{sectionKMeqF}
Let $E\stackrel{\pi}{\to}M$ be a vector bundle over $M$ (with $\dim M$ and $\dim E$ finite). 
Then a subset of
$\GaME$ is (C)-bounded if and only if it is (F)-bounded.
\end{corollary}
\begin{proof}
We may assume that every local trivialization $(U_\al,\Psi_\al)$ as above is
defined over some chart $(U_\al,\psi_\al)$ of $M$.
Let $B$ denote a subset of $\GaME$. Then the following statements are
equivalent ($\al\in A$ as above):
\begin{enumerate}
\item[(i)] $B$ is (C)-bounded.
\item[(ii)] $\ti\Psi_\al(B)$ is (C)-bounded in $\Cinf(U_\al,Z)$, for every
$\al$.
\item[(iii)] $(\psi_\al^{-1})^*\ti\Psi_\al(B)$ is (C)-bounded in
   $\Cinf(\psi_\al(U_\al),Z)$, for every $\al$.
\item[(iv)] $(\psi_\al^{-1})^*\ti\Psi_\al(B)$ is (D)-bounded in
   $\Cinf(\psi_\al(U_\al),Z)$, for every $\al$.
\item[(v)] All seminorms generating the (F)-topology on $\GaME$
(cf.\ (\ref{seminormssection})) are bounded on $B$.
\item[(vi)] $B$ is (F)-bounded.
\end{enumerate}
The equivalences of (i)--(iv) are immediate from the definition of the
(C)-topologies on $\GaME$, resp.\ on $\Cinf(U_\al,Z)$, resp.\ from Theorem
\ref{KMeqF}, in turn (note that $\psi_\al(U_\al)$ is an open subset of the
(F)-space $\R^{\dim M}$). Taking into account
$((\psi_\al^{-1})^*\circ\ti\Psi_\al)(u)=\prt\circ\Psi_\al\circ(u|_{U_\al}
)\circ\psi_\al^{-1}
=(\psi^j_\al\circ(u|_{U_\al})\circ\psi_\al^{-1})_{j=1}^{\dim E}$,
(\ref{seminormssection}) yields (iv)$\Leftrightarrow$(v).
(v)$\Leftrightarrow$(vi), finally, holds by definition.
\end{proof}

\begin{remark}\label{CDF}
Corollary \ref{sectionKMeqF} shows, in particular, that smoothness of
members of (2) resp.\ (3) in Lemma \ref{reprbasic} is not affected by
replacing the (C)- by the (F)-topology on $\cTrsM$ resp.\ on $\CinfM$.
Due to Theorem \ref{KMeqF}, a similar statement is true for the (C)- and the
(D)-topologies on $\Cinf(\cTsrM,\CinfM)$ resp.\ on $\Cinf(\cTsrM,\R)$ in (3)
resp.\ (5) of Lemma \ref{reprbasic}, provided
that $M$ is separable, hence $\cTsrM$ is an (F)-space.
\end{remark}

In the proof of the following theorem, two applications of the uniform
boundedness principle as formulated in condition (2) of \cite[5.22]{KM}
will occur. One of the assumptions of this principle is that for the set
$B(\subseteq E)$ to be demonstrated as being bounded, the normed subspace
$E_B=\bigcup_n nB$ of $E$ (cf.\ \cite[p.\ 63]{Schaefer}) has to be a
Banach space with respect to the Minkowski functional $p_B$ of $B$. Lemma
\ref{banach} below makes sure that this condition is available when needed
below.
\begin{lemma}\label{banach}
Let $E$, $F$ be vector spaces, $F_\la$ (with $\la$ from some index set) locally
convex vector spaces,
$f:E\to F$ linear and $g_\la:F\to F_\la$
a family of linear maps
which is assumed to
be point separating on $F$. If $B$ is an absolutely
convex subset of $E$ such that $E_B$ is a Banach space, and for every $\la$
the set $(g_\la\circ f)(B)$ is bounded in $F_\la$
then also $F_{f(B)}$ is a Banach space.
\end{lemma}
\begin{proof}
By standard methods, it follows that the Minkowski functional
$p_{f(B)}$
of $f(B)$ is the quotient semi-norm of
$p_B$
on the space $F_{f(B)}\cong E_B/\ker (f|_{E_B})$.
Since $(g_\la\circ f)(B)$ is bounded, $p_{f(B)}$ is even a
norm
(observe that for $p_{f(B)}(x)=0$, $g_\la(x)$ is contained in the
intersection of all neighborhoods of zero in $F_\la$). Thus
$(F_{f(B)},p_{f(B)})$, being a quotient of the Banach space $(E_B,p_B)$, is
a Banach space in its own right.
\end{proof}

One more technical remark is in order: Recall that a convenient vector space is
a locally convex vector space in which for each bounded absolutely convex closed
subset $B$, the normed space $(E_B,p_B$) is complete (\cite[Th.\ 2.14]{KM}). 
Now, if we are given a linear bijection between two convenient vector
spaces then, in order to show that the respective families of bounded sets are
corresponding to each other, i.e., that the given map is a bornological
isomorphism, it is sufficient to show that for every bounded
absolutely convex closed subset $B'$ of one of the spaces such that $E_{B'}$ is
a Banach space, $B'$ is bounded also as a subset of the other space: Indeed, if
$B$ is an arbitrary bounded subset of the first space then for its absolutely
convex closed hull $B'=\overline{\Ga}(B)$, $E_{B'}$ is Banach. By assumption,
$B'$, hence {\it a fortiori} $B$, is bounded when viewed as a subset of the
second space. Recall further that by subscripts ``F'' resp.\ ``C'' we declare
the respective
space as being equipped with the (F)- resp.\ the (C)-topology.
\begin{theorem}\label{actingondual}
For every finite-dimensional manifold $M$, the linear isomorphism $\phi:t\mapsto(\tit\mapsto
t\cdot\tit)$ mapping $\cTrsM$ onto $\Lin^b_\CinfM(\cTsrM_F,\CinfM_C)$
is a bornological isomorphism with respect to the (C)-topologies.
\end{theorem}
\begin{proof} That $\phi$ indeed is a linear isomorphism was shown in the
proof of Lemma \ref{reprbasic}. Let $B$ be a subset of $\cTrsM$. The proof will
be achieved by showing (partially under a certain additional assumption on
$B$, cf.\ below) the mutual equivalence of the following
statements, where
always $p\in M$ and $\tit\in\cTsrM$:
\begin{enumerate}
\item[(i)] $B$ is bounded with respect to the (C)-topology of $\cTrsM$.
\item[(ii)] $\{t(p)\mid t \in B\}$ is bounded in
$\TrspM$, for every $p$.
\item[(iii)] $\{t(p)\cdot\tit(p)\mid t \in B\}$ is bounded in $\R$, for every
$p,\tit$.
\item[(iv)] $\{\phi(t)(\tit)(p)\mid t \in B\}$ is
bounded in
$\R$, for every $p,\tit$.
\item[(v)] $\{\phi(t)(\tit)\mid t \in B\}$ is (C)-bounded in $\CinfM$,
for every $\tit$.
\item[(vi)] $\phi(B)$ is bounded w.r.t\ the (C)-topology of
$\Lin^b_\CinfM(\cTsrM_F,\CinfM_C)$.
\end{enumerate}
Identifying (iii)$\Leftrightarrow$(iv) as a mere reformulation and
discerning the chains (i)$\Rightarrow$(ii)$\Leftrightarrow$(iii) and
(iv)$\Leftarrow$(v)$\Leftarrow$(vi) as being obvious (evaluation at
a particular argument always being smooth and linear, hence bounded,
essentially due to \cite[3.13 (1) and 2.11]{KM}), we are left with three
non-trivial implications. Turning to (v)$\Rightarrow$(vi), first of
all note that the (C)-bounded subsets of $\Lin^b_\CinfM(\cTsrM,\CinfM)$
(omitting from now on the subscripts ``$F$'' resp.\ ``$C$'') can
equivalently be viewed as those determined by the structures of
$\Cinf(\cTsrM,\CinfM)$ resp.\ of
$\Lin^b(\cTsrM,\CinfM)$, due to \cite[5.3.\ Lemma]{KM} (where the superscript ``$b$'' is
omitted generally).
Now an appeal to the uniform boundedness principle for spaces of (multi)linear
mappings (\cite[5.18.\ Th.]{KM}) yields (v)$\Rightarrow$(vi).

Whereas the implications established so far are valid for any subset
$B$ of $\cTrsM$, we will have to confine ourselves for
(ii)$\Rightarrow$(i) and (iv)$\Rightarrow$(v) to subsets
$B$ which are absolutely convex and for which $\cTrsM_B$ is a Banach space.
This being a purely algebraic matter, it is equivalent to saying that
$\phi(B)$ is absolutely convex and $\Lin^b_\CinfM(\cTsrM,\CinfM)_{\phi(B)}$ is a
Banach space. Under this additional assumption (in its first form), we obtain
(ii)$\Rightarrow$(i) from a straightforward application of the uniform
boundedness principle for section spaces (\cite[30.1 Prop.]{KM}). 
For a similar argument in favor of (iv)$\Rightarrow$(v) to be legitimate,
however, we need to know that
also $\CinfM_{\phi(B)(\tit)}$ is a Banach space, for every $\tit\in\cTsrM$.
Yet this follows---
assuming (iv) to be true---from $\Lin^b_\CinfM(\cTsrM,\CinfM)_{\phi(B)}$ being
a Banach space by applying Lemma \ref{banach} with
$E$, $F$, $F_\la$, $f$, $g_\la$ replaced by
$\Lin^b_\CinfM(\cTsrM,\CinfM)$, $\CinfM$, $\R$, $\ev_{\tit}$, $\ev_p$,
respectively. Now we are in a position to appeal to \cite[30.1 Prop.]{KM} 
once more, completing the proof of equivalence of
(i)--(vi) for subsets $B$ as specified. Keeping in mind the technical
remark made before the theorem, we conclude that, in fact,
$\phi$ is a bornological isomorphism.
\end{proof}
\begin{remark}
By Corollary \ref{sectionKMeqF}, the bounded sets with respect to the topologies
(C)
and
(F) on $\cTrsM$ are identical. Hence Theorem
\ref{actingondual} could
have been as well formulated for $\cTrsM_C$ replacing $\cTrsM_F$.
\end{remark}

\begin{lemma}\label{fstern}
For a vector bundle $E\stackrel{\pi}{\to}B$, a manifold $B'$ and a smooth map
$f:B'\to B$ ($B$, $B'$ and $E$ finite-dimensional), 
the pullback operator $f^*:\Ga(B,E)\to\Ga(B',f^*E)$
defined by $f^*(u)(p):=(p,u(f(p)))$ is continuous with respect to the
(F)-topologies.
\end{lemma}
\begin{proof}
This is clear from combining the seminorms (\ref{seminormssection}) with the
typical form of charts on pullback bundles (cf.\ Appendix A): Given $l\in\N_0$,
a chart
$(V,\vphi)$
in $B'$, a vector bundle
chart $(U,\Psi)$ over $(U,\psi)$ in $E$ such that $V \cap f^{-1}(U)\neq
\emptyset$ and $L\comp\vphi(V\cap f^{-1}(U))$, the values
$p_{s,\vphi\times_B\Psi,L}(f^*(u))$ are dominated by $p_{s,\Psi,L_1}(u)$ where
$L_1=(\psi\circ f\circ \vphi^{-1})(L)$.
\end{proof}

Also the following Lemma is immediate from the definition of the
seminorms on spaces of sections:
\begin{lemma}\label{evsr}
For finite-dimensional manifolds $M,N$, the operator
$$
\evsr: \Ga(\TOMN)\times\Ga(\proS\TsrM)\to\Ga(\prtS\TsrN)
$$
given by
$$
(A,\xi)\mapsto\Big((p,q)\mapsto\big((p,q)\,,\,\Asr(p,q){\,}\xi(p,
q)\big)\Big)
$$
is continuous with respect to the (F)-topologies.
\end{lemma}

For our final lemma and its corollaries, we will have to include 
infinite-dimensional manifolds
modelled over convenient vector spaces into our considerations, as well as
bundles over such manifolds
and respective spaces of sections. Again we follow \cite{KM}, this
time Sections 27--30. Note that the discussion concerning the (C)-topology
following Theorem \ref{KMeqF} is equally valid for $M$ and $E$ infinite-dimensional.
The main examples of infinite-dimensional manifolds
occurring in the present context are $\ahat\times M\times\bhat$ and
$\ahat\times\bhat$, where $\ahat$ is a closed affine hyperplane in the
(LF)-space $\OmncM$ and $\bhat$ is an (LF)-space itself.
Note that (F)-spaces, (LF)-spaces and closed hyperplanes thereof are convenient
by \cite[2.2 and 2.14]{KM}.
The construction of
pullback bundles in \cite[29.6]{KM}, proceeds in complete analogy to the
finite-dimensional case: In particular, for $\prt:\ahat\times M\times\bhat\to
M$, the pullback bundle $\prtS\TsrM$ can be realized as a manifold by
$(\ahat\times
M\times\bhat)\times_M\TsrM$,
or---more simply---by
$\ahat\times\bhat\times\TsrM$, with base
point map $(\om,A,t)\mapsto(\om,\pi(t),A)$ (the version which we
will exclusively use in what follows).

As a last prerequisite to the following lemma, we observe that
the evaluation map $\ev:\Cinf(M,E)\times M\to E$ is smooth for any
convenient vector space $E$ and for every manifold $M$. This result (which
does not appear explicitly in \cite{KM}) can be obtained by an argument
completely analogous to that of the vector space case (where an open subset
$U$
of a locally convex space $E$ takes the place of $M$): Due to the proof of
\cite[27.17]{KM}, $\Cinf(M,F)$ can be viewed as a closed linear subspace of a
product of spaces $\Cinf(\R,F)$. This statement replacing \cite[3.11 Lemma]{KM} in the
proof of \cite[Th.\ 3.12]{KM}, the latter as well as Cor.\ 
3.13\,(1) (saying that $\ev:\Cinf(U,F)\times U\to F$ is smooth) together with
their proofs carry over to the manifold case. As a by-product, we
obtain a fact which will be tacitly used in the proof of the lemma
below: By continuity of the
evaluation map and by the manifold analogue of \cite[3.12]{KM} just mentioned,
it follows that the obvious linear isomorphism
$\Cinf(M,\prod_\al E_\al)\cong\prod_\al\Cinf(M,E_\al)$ (for a manifold $M$ and
convenient vector spaces $E_\al$) is even a bornological isomorphism.

\begin{lemma}\label{kriegl}{\rm [A. Kriegl, personal communication]}
Let $M,N$ be manifolds and
$E\stackrel{\pi}{\to}N$ a smooth vector bundle over $N$ ($M$, $N$, $E$
possibly infinite-dimensional). Then we have a
bornological isomorphism
$$
\Cinf(M,\Ga(N,E))\cong\Ga(M\times N,\prtS E)
$$
with respect to the (C)-topologies. Elements $f,g$ corresponding to each other
by this isomorphism are
related by $(p,f(p)(q))=g(p,q)$; we write $g=f^\wedge$, $f=g^\vee$.
Moreover, the evaluation mapping
$$
\ev:\Cinf(M,\Ga(N,E))\times M\to\Ga(N,E)
$$
is smooth with respect to the (C)-topologies resp.\ the structure given on $M$.
\end{lemma}
\begin{proof}
For the proof, we will represent each of the two spaces as a closed
(with respect to the $c^\infty$-topology, cf.\ \cite[2.12]{KM}) subspace of
a respective product space; these product spaces are then seen to be isomorphic
by means of the exponential law for spaces of type $\Cinf(P,F)$ ($P$ a
manifold and $F$ convenient; \cite[27.17]{KM}). In all three steps, the
families of bounded sets are preserved. Finally, it is shown that under the
isomorphism between the product spaces in fact the subspaces
$\Cinf(M,\Ga(N,E))$ and $\Ga(M\times N,\prtS E)$ correspond to each other.

Choose a vector bundle atlas for $E\stackrel{\pi}{\to}N$ consisting of
local
trivializations $\psi_\al:\pi^{-1}(U_\al)\to U_\al\times Z$ where $Z$ denotes
the typical fiber of $E$. By \cite[30.1]{KM} we obtain a linear embedding
\begin{align*}
\Ga(N,E)&\hookrightarrow\prod_\al\Cinf(U_\al,Z)\\
((\prt\circ\psi_\al)_*)_\al:u&\mapsto (\prt\circ\psi_\al\circ u|_{U_\al})_\al
\end{align*}
having $c^\infty$-closed image, the latter due to the fact that that it can be
characterized as the subspace consisting of all $(f_\al)_\al$ for
which the maps $q\mapsto\psi_\al^{-1}(q,f_\al(q))$
form a coherent family of local sections of $E$.
For $c^\infty$-closedness, express coherence by the conditions
$\phi_{\alb}(f_\al)=\phi_{\bal}(f_\bet)$ for all $\al,\bet$ where
$$
\phi_{\alb}:\Cinf(U_\al,Z)\ni f_\al\mapsto\psi_\al^{-1}\circ(\id_{U_\al},
f_\al)|_{\Ualb}\in\Ga(\Ualb,E|\Ualb)
$$
and similarly for $\phi_{\bal}$. Now, it certainly would suffice to show
smoothness (hence $c^\infty$-continuity) of both $\phi_\alb$ and
$\phi_{\bal}$. Making things work for $\phi_\alb$ requires $\Ga(\Ualb,E|\Ualb)$
to be equipped with the topology induced by the atlas
$\ca_\al:=\{(\Ualb,\psi_\al|_{E|\Ualb})\}$ on $E|\Ualb$ whereas for
$\phi_\bal$, the atlas $\ca_\bet:=\{(\Ualb,\psi_\bet|_{E|\Ualb})\}$ would be
appropriate.

However, due to the remarks preceding Corollary
\ref{KMeqF}, we can safely use $\ca_\al\cup\ca_\bet$ on $E|\Ualb$ (doing justice
to both $\phi_\alb$ and $\phi_\bal$ simultaneously this way) without affecting
boundedness in $\Ga(\Ualb,E|\Ualb)$ resp.\ smoothness of $\phi_\alb$ and
$\phi_\bal$.

From the above, we get an embedding

\begin{align*}
\Cinf(M,\Ga(N,E))&\hookrightarrow\Cinf\Big(M,\prod_\al\Cinf(U_\al,Z)\Big)
\cong\prod_\al\Cinf(M,\Cinf(U_\al,Z))\\
((\prt\circ\psi_\al)_{**})_\al:f&\mapsto (p\mapsto\prt\circ\psi_\al\circ
f(p)|_{U_\al})_\al,
\end{align*}
again with $c^\infty$-closed image, consisting of those $(g_\al)_\al$ for which
$\psi_\al^{-1}(q,g_\al(p)(q))$
forms a coherent family of local sections of $E$
for every fixed $p\in M$.

On the other hand, noting the family of all $(M\times
U_\al,\id_M\times\psi_\al)$ to form a vector bundle atlas
for $\prtS E\to M\times N$, we obtain a linear embedding
\begin{align*}
\Ga(M\times N,\prtS E)&\hookrightarrow\prod_\al\Cinf(M\times U_\al,Z)\\
\ti u&\mapsto (\prt\circ(\id_M\times\psi_\al)\circ \ti u|_{M\times U_\al})_\al
\end{align*}
mapping $\Ga(M\times N,\prtS E)$ onto the $c^\infty$-closed subspace consisting
of all $(\ti f_\al)_\al$ for which
$(\id_M\times\psi_\al)^{-1}(p,q,\ti f_\al(p,q))=(p,\psi_\al^{-1}(q,\ti
f_\al(p,q)))$
forms a coherent family of local sections of $\prtS E$.
Via the bornological isomorphisms $\Cinf(M\times
U_\al,Z)\cong\Cinf(M,\Cinf(U_\al,Z))$
given by the exponential law \cite[27.17]{KM},
we obtain an isomorphism of the last two product spaces occurring
above. This, in turn, induces bornological isomorphisms
of the two aforementioned subspaces resp.\ of $\Cinf(M,\Ga(N,E))$
and $\Ga(M\times N,\prtS E)$. Tracing all the assignments involved in the
construction shows the explicit form of this last isomorphism to be the one
given in the lemma.


Finally, replacing $E$ by $\Ga(N,E)$ in the remarks on smoothness of the
evaluation map preceding the lemma shows that also
$\ev:\Cinf(M,\Ga(N,E))\times M\to\Ga(N,E)$
is smooth with respect to the (C)-topologies resp.\ the structure given on $M$.
\end{proof}

Note that, using the notations of the preceding lemma, sections $u$ of
$\prtS E$ are precisely given by smooth maps $\bar u:M\times N\to E$ with
$\pi\circ\bar u=\prt$, i.e., $\bar u(p,q)$ having $q$ as base point, for
all $p,q$. The section $u$ itself takes the form $u(p,q)=(p,\bar
u(p,q))$.

For the corollary to follow, recall that $\prt':\prtS E\to E$ denotes the
canonical projection as defined in Appendix A.
\begin{corollary}\label{corevbar}
For manifolds $M,N$ and a vector bundle $E\stackrel{\pi}{\to}N$ ($M$, $N$,
$E$ as in Lemma \ref{kriegl}), the
operator $\evbar:\Ga(M\times N,\prtS E)\times M\to\Ga(N,E)$ defined by
$\evbar(u,p)(q):=\prt'(u(p,q))$ is smooth when both section spaces are
equipped with their (C)-topologies.
\end{corollary}
\begin{proof}
According to Lemma \ref{kriegl}, $\evbar$ and $\ev$ correspond to each other via
the bornological isomorphism $\Cinf(M,\Ga(N,E))\cong\Ga(M\times N,\prtS E)$. By
\cite[2.11]{KM}, this isomorphism and its inverse
are smooth. Hence
the smoothness of $\evbar$ follows from that of $\ev$.
\end{proof}
\begin{corollary}\label{expCF}
Let $M$ and $N$ be manifolds, and
$E\stackrel{\pi}{\to}N$ a smooth vector bundle over $N$ ($N$ and $E$ finite-dimensional). 
Then for every $u\in\Ga(M\times N,\prtS E)$ the associated map $u^\vee:M\to\Ga(N,E)$ is smooth
with respect to the (F)-topology on $\Ga(N,E)$.
\end{corollary}
\begin{proof}
By Lemma \ref{kriegl}, $u^\vee\in\Cinf(M,\Ga(N,E))$ where $\Ga(N,E)$ carries
the (C)-topology. By Corollary \ref{sectionKMeqF}, the (C)- resp.\ the
(F)-bounded subsets on $\Ga(N,E)$ are the same. Therefore, $u^\vee$ is also
smooth into $\Ga(N,E)_F$.
\end{proof}
If also $M$ has finite dimension Corollary \ref{expCF} can be proved in the
well-known ``classical'' manner, using charts. However, we need the
result in the general case where $M$ is of infinite dimension.
\end{appendix}
\vskip1em
{\bf Acknowledgments:} We would like to thank James D.E.\ Grant and Andreas
Kriegl for helpful discussions. Work on this paper was supported by projects
P16742, Y237, and P20525 of the Austrian Science Fund.


\begin{thebibliography}{10}

\bibitem{herbertgeo}
{Balasin, H.}
\newblock Geodesics for impulsive gravitational waves and the multiplication of
  distributions.
\newblock {\em Class.~Quant.~Grav.}, {\bf 14}:455--462, 1997.

\bibitem{bourbaki-algebra}
{Bourbaki, N.}
\newblock {\em Algebra I, Chapters 1--3}.
\newblock Elements of Mathematics. Hermann, Addison-Wesley, Paris,
  Massachusetts, 1974.

\bibitem{clarke}
{Clarke, C.~J.~S., Vickers, J.~A., Wilson, J.~P.}
\newblock Generalised functions and distributional curvature of cosmic strings.
\newblock {\em Class.~Quant.~Grav.}, {\bf 13}:2485--2498, 1996.

\bibitem{c1}
{Colombeau, J.~F.}
\newblock {\em New Generalized Functions and Multiplication of Distributions}.
\newblock North Holland, Amsterdam, 1984.

\bibitem{c2}
{Colombeau, J.~F.}
\newblock {\em Elementary Introduction to New Generalized Functions}.
\newblock North Holland, Amsterdam, 1985.

\bibitem{Cbull}
{Colombeau, J.~F.}
\newblock Multiplication of distributions.
{\em Bull.\ Amer.\ Math.\ Soc.\ (N.S.)}, {\bf 23}:251--268, 1990.

\bibitem{c3}
{Colombeau, J.~F.}
\newblock {\em Multiplication of Distributions. A tool in mathematics,
  numerical engineering and theoretical physics}, {\em
  Lecture Notes in Mathematics} {\bf 1532}.
\newblock Springer, 1992.

\bibitem{cm}
{Colombeau, J.~F., Meril, A.}
\newblock Generalized functions and multiplication of distributions on
  {${\mathcal C}^\infty$} manifolds.
\newblock {\em J.~Math.~Anal.~Appl.}, {\bf 186}:357--364, 1994.

\bibitem{RD}
{De Roever, J.~W., Damsma, M.}
\newblock Colombeau algebras on a {${\cal C}^\infty$}-manifold.
\newblock {\em {Indag.~Mathem., N.S.}}, {\bf 2}(3), 1991.

\bibitem{D07}
{Delcroix, A.}
\newblock Kernel theorems in spaces of tempered generalized functions.
\newblock {\em Math.\ Proc.\ Cambridge Philos.\ Soc.}, {\bf 142}(3):557--572,
  2007.

\bibitem{dieudonne3}
{Dieudonn\'e, J.}
\newblock {\em Treatise on Analysis}, Vol.\ {\bf 3}.
\newblock Academic Press, New York, 1972.

\bibitem{G08}
{Garetto, C.}
\newblock Fundamental solutions in the {C}olombeau framework: applications to
  solvability and regularity theory.
\newblock {\em Acta Appl. Math.}, {\bf 102}(2-3):281--318, 2008.

\bibitem{waveq}
{Grant, J. D. E., Mayerhofer, E., Steinbauer, R.}
\newblock The wave equation on singular space times.
\newblock {\em Commun.\ Math.\ Phys.} {\bf 285}(2): 399--420, 2009.

\bibitem{GHV}
{Greub, W., Halperin, S., Vanstone, R.}
\newblock {\em Connections, Curvature, and Cohomology {I}}, 
  {\em Pure and Applied Mathematics} {\bf 47}.
\newblock Academic Press, 1972.

\bibitem{mgrepdistr}
{Grosser, M.}
\newblock Equivalent representations of distribution spaces on manifolds and
  automatic continuity.
\newblock {\em Preprint}, 2009.

\bibitem{mg-bed}
{Grosser, M.}
\newblock Tensor valued Colombeau functions on manifolds.
\newblock In {Kaminski, A., Pilipovi\'c, S., Oberguggenberger, M.}, Ed.,
  {\em Proceedings of the International Conference on Generalized Functions
  2007, Bedlewo, Poland}, Banach Center Publications, 2008.

\bibitem{found}
{Grosser, M., Farkas, E., Kunzinger, M., Steinbauer, R.}
\newblock On the foundations of nonlinear generalized functions {I}, {II}.
\newblock {\em Mem. Amer. Math. Soc.}, {\bf 153}(729), 2001.

\bibitem{book}
{Grosser, M., Kunzinger, M., Oberguggenberger, M., Steinbauer, R.}\\
\newblock {\em Geometric Theory of Generalized Functions}, {\em
  Mathematics and its Applications} {\bf 537}.
\newblock Kluwer Academic Publishers, Dordrecht, 2001.

\bibitem{vim}
{Grosser, M., Kunzinger, M., Steinbauer, R., Vickers, J.}
\newblock A global theory of algebras of generalized functions.
\newblock {\em Adv. Math.}, {\bf 166}:179--206, 2002.

\bibitem{HOP}
{Hörmann, G., Oberguggenberger, M., Pilipovi\'{c}, S.}
\newblock Microlocal hypoellipticity of linear partial differential operators
  with generalized functions as coefficients.
\newblock {\em Trans.\ Amer.\ Math.\ Soc.}, {\bf 358}(8):3363--3383, 2006.

\bibitem{Jel}
{Jel\'\i nek, J.}
\newblock An intrinsic definition of the {C}olombeau generalized functions.
\newblock {\em Comment.~Math.~Univ.~Carolinae}, {\bf 40}:71--95, 1999.

\bibitem{Jel3}
{Jel{\'{\i}}nek, J.}
\newblock On introduction of two diffeomorphism invariant {C}olombeau algebras.
\newblock {\em Comment. Math. Univ. Carolin.}, {\bf 45}(4):615--632, 2004.


\bibitem{Jel2}
{Jel{\'{\i}}nek, J.}
\newblock Equality of two diffeomorphism invariant {C}olombeau algebras.
\newblock {\em Comment. Math. Univ. Carolin.}, {\bf 45}(4):633--662, 2004.


\bibitem{KK}
{Keyfitz, B.~L., Kranzer. H.~C.}
\newblock Spaces of weighted measures for conservation laws with singular shock
  solutions.
\newblock {\em J.~Differential Equations} {\bf 118}(2), 420--451, 1995.

\bibitem{michorbook}
{Kolar, I., Michor, P.~W., Slovak, J.}
\newblock {\em Natural Operations in Differential Geometry}.
\newblock Springer, Berlin, 1993.

\bibitem{KM}
{Kriegl, A., Michor, P.~W.}
\newblock {\em The Convenient Setting of Global Analysis}, 
  {\em Math.~Surveys Monogr.} {\bf 53}, 
\newblock Amer.~Math.~Soc., Providence, RI, 1997.

\bibitem{symm}
{Kunzinger, M., Oberguggenberger, M.}
\newblock Group analysis of differential equations and generalized functions.
\newblock {\em SIAM J. Math. Anal.}, {\bf 31}(6):1192--1213, 2000.

\bibitem{penrose}
{Kunzinger, M., Steinbauer, R.}
\newblock A note on the {P}enrose junction conditions.
\newblock {\em Class. Quant. Grav.}, {\bf 16}:1255--1264, 1999.

\bibitem{gprg}
{Kunzinger, M., Steinbauer, R.}
\newblock Generalized {pseudo-}{R}iemannian geometry.
\newblock {\em Trans. Amer. Math. Soc.}, {\bf 354}(10):4179--4199, 2002.

\bibitem{conn}
{Kunzinger, M., Steinbauer, R., Vickers, J.}
\newblock Generalized connections and curvature.
\newblock {\em Math.\ Proc.\ Cambridge Philos.\ Soc.}, {\bf 139}(3):497--521,
  2005.

\bibitem{Marsden}
{Marsden, J.~E.}
\newblock Generalized {H}amiltonian mechanics.
\newblock {\em Arch.\ Rat.\ Mech.\ Anal.}, {\bf 28}(4):323--361, 1968.

\bibitem{NP06}
{Nedeljkov, M., Pilipovi{\'c}, S.}
\newblock Generalized function algebras and {PDE}s with singularities. {A}
  survey.
\newblock {\em Zb. Rad. (Beogr.)}, {\bf 11}(19):61--120, 2006.

\bibitem{NPS}
{Nedeljkov, M., Pilipovi\'c, S., Scarpal\'ezos, D.}
\newblock {\em The Linear Theory of Colombeau Generalized Functions},
 {\em Pitman Research Notes in Mathematics} {\bf 385}.
\newblock Longman, Harlow, U.K., 1998.

\bibitem{MOBook}
{Oberguggenberger, M.}
\newblock {\em Multiplication of Distributions and Applications to Partial
  Differential Equations}, {\em Pitman Research Notes in
  Mathematics} {\bf 259}.
\newblock Longman, Harlow, U.K., 1992.

\bibitem{MR08}
{Ruzhansky, M.}
\newblock On local and global regularity of {F}ourier integral operators.
\newblock {\em Preprint}, {\tt arXiv:0802.2142v1}, 2008.

\bibitem{Schaefer}
{Schaefer, H.~H.}
\newblock {\em Topological Vector Spaces}.
\newblock Springer, New York, 1986.

\bibitem{Schw}
{Schwartz, L.}
\newblock Sur l'impossibilit\'e de la multiplication des distributions.
\newblock {\em C.~R.~Acad.~Sci.~Paris}, {\bf 239}:847--848, 1954.

\bibitem{ro-bed}
{Steinbauer, R.}
\newblock A geometric approach to full Colombeau algebras.
\newblock In {Kaminski, A., Pilipovi\'c, S., Oberguggenberger, M.}, Ed.,
  {\em Proceedings of the International Conference on Generalized Functions
  2007, Bedlewo, Poland}, Banach Center Publications, 2008.

\bibitem{SV06}
{Steinbauer, R., Vickers, J. A.}
\newblock The use of generalized functions and distributions in general
  relativity.
\newblock {\em Classical Quantum Gravity}, {\bf 23}(10):R91--R114, 2006.

\bibitem{vickersESI}
{Vickers, J.~A.}
\newblock Nonlinear generalized functions in general relativity.
\newblock In {Grosser, M., Hörmann, G., Kunzinger, M., Oberguggenberger, M.},
  Ed., {\em Nonlinear Theory of Generalized Functions}, 
  {\em CRC Research Notes in Mathematics} {\bf 401}, 275--290, 
  CRC Press, Boca Raton, 1999. 

\bibitem{genhyp}
{Vickers, J.~A., Wilson, J.~P.}
\newblock Generalized hyperbolicity in conical spacetimes.
\newblock {\em Class.~Quantum.~Grav.}, {\bf 17}:1333--1360, 2000.

\bibitem{sotonTF}
{Vickers, J., Wilson, J.}
\newblock A nonlinear theory of tensor distributions.
\newblock {\em ESI-Preprint} {\bf 566} ({\tt http://www.esi.ac.at/ESI-Preprints.html}), 1998.

\end{thebibliography}
\end{document}